\documentclass[11pt]{article}
\usepackage[latin1]{inputenc}
\usepackage{amsmath,amsthm,amssymb}
\usepackage{amsfonts}
\usepackage{amsmath,amsthm,amssymb,amscd}
\usepackage{latexsym}
\usepackage{color}
\usepackage{graphicx}
\usepackage{mathrsfs}
\usepackage{cite}

\usepackage{color,enumitem,graphicx}
\usepackage[colorlinks=true,urlcolor=black,
citecolor=black,linkcolor=black,linktocpage,pdfpagelabels,
bookmarksnumbered,bookmarksopen]{hyperref}

\makeatletter \@addtoreset{equation}{section} \makeatother

\newtheorem{theorem}{Theorem}[section]
\newtheorem{definition}{Definition}[section]
\newtheorem{proposition}{Proposition}[section]
\newtheorem{lemma}{Lemma}[section]
\newtheorem{remark}{Remark}[section]

\allowdisplaybreaks

\begin{document}
\title{Normalized solutions for $p$-Laplacian equation with critical Sobolev exponent and mixed nonlinearities}

\author{Shengbing Deng\footnote{Corresponding author.}\ \footnote{E-mail address:\, {\tt shbdeng@swu.edu.cn} (S. Deng), {\tt qrwumath@163.com} (Q. Wu).}\  \ and Qiaoran Wu\\
\footnotesize  School of Mathematics and Statistics, Southwest University,
Chongqing, 400715, P.R. China}

\date{ }
\maketitle

\begin{abstract}
{In this paper, we consider the existence and multiplicity of normalized solutions for the following $p$-Laplacian critical equation
	    \begin{align*}
	    	 \left\{\begin{array}{ll}
	    		-\Delta_{p}u=\lambda\lvert u\rvert^{p-2}u+\mu\lvert u\rvert^{q-2}u+\lvert u\rvert^{p^*-1}u&\mbox{in}\ \mathbb{R}^N,\\
	    		\int_{\mathbb{R}^N}\lvert u\rvert^pdx=a^p,
	    	\end{array}\right.
	    \end{align*}
    where $1<p<N$, $2<q<p^*=\frac{Np}{N-p}$, $a>0$, $\mu\in\mathbb{R}$ and $\lambda\in\mathbb{R}$ is a Lagrange multiplier. Using concentration compactness lemma, Schwarz rearrangement, Ekeland variational principle and mini-max theorems, we obtain several existence results under $\mu>0$ and other assumptions. We also analyze the asymptotic behavior of there solutions as $\mu\rightarrow 0$ and $\mu$ goes to its upper bound. Moreover, we show the nonexistence result for $\mu<0$ and get that the $p$-Laplacian equation has infinitely solutions by genus theory when $p<q<p+\frac{p^2}{N}$.
}



\end{abstract}

\section{\textbf{Introduction}}

    In this paper, we consider the following $p$-Laplacian equation
        \begin{equation}\label{equation}
    	    -\Delta_{p}u=\lambda\lvert u\rvert^{p-2}u+\mu\lvert u\rvert^{q-2}u+\lvert u\rvert^{p^*-2}u\quad\mbox{in}\ \mathbb{R}^N,
        \end{equation}
    where $1<p<N$, $2<q<p^*=\frac{Np}{N-p}$, $\lambda,\mu\in\mathbb{R}$ and $\Delta_{p}=\mbox{div}(\lvert\nabla u\rvert^{p-2}\nabla u)$ is the $p$-Laplacian operator.

    If $p=2$, then equation (\ref{equation}) can be derived from the time-dependent equation as
        \begin{equation}\label{time}
        	i\psi_{t}+\Delta\psi+\mu\lvert\psi\rvert^{q-2}+\lvert\psi\rvert^{2^*-2}\psi=0\quad\mbox{in}\ \mathbb{R}_{+}\times\mathbb{R}^N,
        \end{equation}
    when we look for the stationary waves of the form $\psi(t,x)=e^{-i\lambda t}u(x)$. Equation (\ref{time}) can represent both the famous Schr\"odinger equation which describes the laws of particle motion\cite{hhsca,sn1,sn2}, and the Bose-Einstein condensates\cite{fg,ttvmzx,zz}. Consider the following equation
        \begin{equation}\label{f}
        	-\Delta u=\lambda u+f(u)\quad\mbox{in}\ \mathbb{R}^N,
        \end{equation}
    A direct way for studying the existence of solutions for (\ref{f}) is to find the critical points for the following  functional
        \[I(u)=\frac{1}{2}\int_{\mathbb{R}^N}|\nabla u|^2dx-\frac{\lambda}{2}\int_{\mathbb{R}^N}| u|^2dx-\int_{\mathbb{R}^N}F(u)dx,\]
    where $F(s)=\int_{0}^{s}f(t)dt$. In this case, particular attention is devoted to least action solutions, namely solutions minimizing $I(u)$ among all non-trivial solutions. Here we refer the readers to \cite{acosmasmm,fagf}. Another possible approach is to give a prescribed $L^2$ mass of $u$, that is, consider the constraint
        \begin{equation}\label{l2con}
        	\int_{\mathbb{R}^N}\lvert u\rvert^2dx=a^2,
        \end{equation}
    where $a>0$ is a constant. Then, the corresponding functional of (\ref{f}) is
        \[J(u)=\frac{1}{2}\int_{\mathbb{R}^N}|\nabla u|^2dx-\int_{\mathbb{R}^N}F(u)dx,\]
    and $\lambda$ appears as a Lagrange multiplier. The solutions of (\ref{f}) which has prescribed mass called normalized solutions.

    When we consider the normalized solutions of (\ref{f}), a new critical exponent $2+\frac{4}{N}$ appears which is called $L^2$-critical exponent. This constant can be derived from the Gagliardo-Nirenberg inequality: for every $p<q<p^*$, there exist an optimal constant $C_{N,p,q}>0$ such that
        \[\lVert u\rVert_{q}^q\leqslant C_{N,p,q}\lVert\nabla u\rVert_{p}^{q\gamma_{q}}\lVert u\rVert_{p}^{q(1-\gamma_{q})}\quad\forall u\in W^{1,p}(\mathbb{R}^N),\]
    where
        \[\gamma_{q}:=\frac{N(q-p)}{pq}.\]
    By \cite[section 1]{am}, $C_{N,p,q}$ can be attained for some $\psi_{0}\in W^{1,p}(\mathbb{R}^N)$ which satisfies
        \[-\Delta_{p}u+\lvert u\rvert^{p-2}u=\beta\lvert u\rvert^{q-2}\]
    for some $\beta>0$. Moreover, $\psi_{0}$ can be chosen non-negative, radially symmetric, radially non-increasing and tends to $0$ as $\lvert x\rvert\rightarrow+\infty$. If the nonlinearities of (\ref{f}) are pure $L^2$-subcritical terms, for example $f(u)=\lvert u\rvert^{q-2}u$ with $2<q<2+\frac{4}{N}$, then by Gagliardo-Nirenberg inequality, it is not difficult to prove that $J(u)$ is bounded from below if $u$ satisfies (\ref{l2con}), and we can find a global minimum solution of (\ref{f}). Here we refer the readers to \cite{lpl2,sma}. If the nonlinearities of (\ref{f}) are pure $L^2$-supercritical terms, for example $f(u)=\lvert u\rvert^{q-2}u$ with $2+\frac{4}{N}<q<2^*$, then $J(u)$ is unbounded from below if $u$ satisfies (\ref{l2con}). The first result to deal with $L^2$-supercritical case was studied by Jeanjean\cite{jl}. He proved that problem (\ref{f}) has a mountain-pass type solution under suitable assumptions. Compared with the pure $L^2$-subcritical or $L^2$-supercritical case,  the mixed nonlinearities terms are more complicated, Soave \cite{sn1} considered the following problem
        \begin{equation}\label{secn}
    	    \left\{\begin{array}{ll}
    		    -\Delta u=\lambda u+\mu\lvert u \rvert^{q-2}u+\lvert u \rvert^{p-2}u&\mbox{in}\ \mathbb{R}^N,\\
    		    \int_{\mathbb{R}^N}\lvert u\rvert^2dx=a^2,
    	    \end{array}\right.
        \end{equation}
    where $N\geqslant 3$, $\mu>0$, $2<q\leqslant 2+\frac{4}{N}\leqslant p<2^*$, and analyzed the existence, asymptotic behavior and stability of solutions. All the references listed above are considered in Sobolev subcritical case, the first result of the normalized solutions for Sobolev critical case was studied by Soave \cite{sn2}, that is, Soave studied the following problem
        \begin{equation}\label{secn}
        	\left\{\begin{array}{ll}
        		-\Delta u=\lambda u+\mu\lvert u \rvert^{q-2}u+\lvert u \rvert^{2^*-2}u&\mbox{in}\ \mathbb{R}^N,\\
        		\int_{\mathbb{R}^N}\lvert u\rvert^2dx=a^2,
        	\end{array}\right.
        \end{equation}
    where $N\geqslant 3$, $\mu>0$, $2<q<p^*$ and analyzed the existence, nonexistence, asymptotic behavior and stability of solutions. We refer to \cite{jljjltt,jlltt,wjwy} and references therein for the existence of normalized solutions for the mixed nonlinearities.

    If  $p\neq 2$, there are few papers on the normalized solution of $p$-Laplacian equation.
    Wang et al. \cite{wwlqzj} considered
        \begin{equation*}
        	\left\{\begin{array}{ll}
        		-\Delta_{p}u+\lvert u\rvert^{p-2}u=\mu u+\lvert u\rvert^{s-2}u&\mbox{in}\ \mathbb{R}^N,\\
        		\int_{\mathbb{R}^N}\lvert u\rvert^2dx=\rho,
        	\end{array}\right.
        \end{equation*}
    where $1<p<N$, $\mu\in\mathbb{R}$, $s\in(\frac{N+2}{N}p,p^*)$, they considered the $L^2$ constraint, by the Gagliardo-Nirenberg inequality, the $L^2$-critical exponent should be $\frac{N+2}{N}p$. Moreover, we know $L^2(\mathbb{R}^N)\not\subset W^{1,p}(\mathbb{R}^N)$, so the work space is $W^{1,p}(\mathbb{R}^N)\cap L^2(\mathbb{R}^N)$ which is a Hilbert space. The work space is Hilbert space is very important for \cite{wwlqzj}. \cite{zzzz} is the first paper to study the $p$-Laplacian equation with $L^p$ constraint:
        \begin{equation}\label{Sovlevsub}
        	\left\{\begin{array}{ll}
        		-\Delta_{p}u=\lambda\lvert u\rvert^{p-2}u+\mu\lvert u\rvert^{q-2}u+g(u)&\mbox{in}\ \mathbb{R}^N,\\
        		\int_{\mathbb{R}^N}\lvert u\rvert^pdx=a^p,
        	\end{array}\right.
        \end{equation}
    where $g\in C(\mathbb{R},\mathbb{R})$ and there exist $p+\frac{p^2}{N}<\alpha\leqslant\beta<p^*$ such that for all $s\in\mathbb{R}$, there is
        \[0<\alpha G(s)s\leqslant g(s)s\leqslant\beta G(s),\quad G(s)=\int_{0}^{s}g(t)dt.\]
    A simple example is $g(s)=\lvert s\rvert^{r-2}s$ with $p+\frac{p^2}{N}<r<p^*$. Moreover, Wang and Sun \cite{wcsj} considered both the $L^2$ constraint and the $L^p$ constraint for the following problem
        \begin{equation*}
        	\left\{\begin{array}{ll}
        		-\Delta_{p}u+V(x)\lvert u\rvert^{p-2}u=\lambda\lvert u\rvert^{r-2}u+\lvert u\rvert^{q-2}u&\mbox{in}\ \mathbb{R}^N\\
        		\int_{\mathbb{R}^N}\lvert u\rvert^rdx=c,
        	\end{array}\right.
        \end{equation*}
    where $1<p<N$, $\lambda\in\mathbb{R}$, $r=p$ or $2$, $p<q<p^*$ and $V(x)$ is a trapping potential satisfies
        \[V(x)\in C(\mathbb{R}^N),\quad\lim_{\lvert x\rvert\rightarrow+\infty}V(x)=+\infty\quad\mbox{and}\quad\inf_{x\in\mathbb{R}^N}V(x)=0.\]

    In this work, we study the existence of normalized solutions for (\ref{equation}) by fixing $L^p$-norm of $u$.
    Let
        \[S_{a}:=\{u\in W^{1,p}(\mathbb{R}^N):\lVert u\rVert_{p}^p=a^p\},\]
    where $a>0$ is a constant. Following \cite[definition 1]{sn2}, we give the definition of ground state as follows.
        \begin{definition}\label{grosta}
    	    {\rm We say $u$ is a ground state of {\rm (\ref{equation})} on $S_{a}$, if $u$ is a solution to {\rm (\ref{equation})} and have minimal energy among all the solutions which belongs to $S_{a}$, that is}
    	        \[dE_{\mu}|_{S_{a}}(u)=0,\quad\mbox{{\rm and}}\quad E_{\mu}(u)=\inf\big\{E_{\mu}(v): dE_{\mu}|_{S_{a}}(v)=0, u\in S_{a}\big\}.\]
    \end{definition}

    Since $u\in S_{a}$, there are some difficulties to observe the structure  of $E_{\mu}(u)$ directly. A possible approach is to consider an auxiliary function
        \[\Psi_{u}^{\mu}(s):=E_{\mu}(s\star u)=\frac{1}{p}e^{ps}\lVert\nabla u\rVert_{p}^p-\frac{\mu}{q}e^{q\gamma_{q}s}\lVert u\rVert_{q}^q-\frac{1}{p^*}e^{p^*s}\lVert u\rVert_{p^*}^{p^*},\]
    where
        \[s\star u:=e^{\frac{Ns}{p}}u(e^s\cdot).\]
    It is clear that $s\star u\in S_{a}$ for all $s\in\mathbb{R}$ if $u\in S_{a}$. Thus, we can investigate the structure of $\Psi_{u}^{\mu}$ to speculate the structure of $E_{\mu}|_{S_{a}}$.

    Assume that $u$ is a critical point of $E_{\mu}|_{S_{a}}$. Then $0$ may be a critical point of $\Psi_{u}^{\mu}$. If $0$ is the critical point of $\Psi_{u}^{\mu}$, we have $(\Psi_{u}^{\mu})'(0)=0$, that is
        \begin{equation}\label{Psider}
        	\lVert\nabla u\rVert_{p}^p=\mu\gamma_{q}\lVert u\rVert_{q}^q+\lVert u\rVert_{p^*}^{p^*}.
        \end{equation}
    In fact, by the Pohozaev indentity of (\ref{equation}), all critical point of $E_{\mu}$ satisfies (\ref{Psider})(see Proposition \ref{Pohozaev}). Therefore, if we consider such a manifold
        \[\mathcal{P}_{a,\mu}=\{u\in S_{a}:P_{\mu}(u)=0\},\]
    where
        \[P_{\mu}(u)=\lVert\nabla u\rVert_{p}^p-\mu\gamma_{q}\lVert u\rVert_{q}^q-\lVert u\rVert_{p^*}^{p^*},\]
    we know all critical points of $E_{\mu}|_{S_{a}}$ belong to $\mathcal{P}_{a,\mu}$ and $s\star u\in\mathcal{P}_{a,\mu}$ if and only if $(\Psi_{u}^{\mu}(s))'=0$. The manifold $\mathcal{P}_{a,\mu}$ is always called Pohozaev manifold.

    We divde $\mathcal{P}_{a,\mu}$ into three parts
        \begin{align*}
    	    \mathcal{P}_{a,\mu}^{+}&=\big\{u\in\mathcal{P}_{a,\mu}:(\Psi_{u}^{\mu})''(0)>0\big\}=\big\{u\in\mathcal{P}_{a,\mu}:p\lVert\nabla u\rVert_{p}^p>\mu q\gamma_{q}^2\lVert u\rVert_{q}^q+p^*\lVert u\rVert_{p^*}^{p^*}\big\},
        \end{align*}
        \begin{align*}
	        \mathcal{P}_{a,\mu}^{0}&=\big\{u\in\mathcal{P}_{a,\mu}:(\Psi_{u}^{\mu})''(0)=0\big\}=\big\{u\in\mathcal{P}_{a,\mu}:p\lVert\nabla u\rVert_{p}^p=\mu q\gamma_{q}^2\lVert u\rVert_{q}^q+p^*\lVert u\rVert_{p^*}^{p^*}\big\},
        \end{align*}
    and
        \begin{align*}
	        \mathcal{P}_{a,\mu}^{-}&=\big\{u\in\mathcal{P}_{a,\mu}:(\Psi_{u}^{\mu})''(0)<0\big\}=\big\{u\in\mathcal{P}_{a,\mu}:p\lVert\nabla u\rVert_{p}^p<\mu q\gamma_{q}^2\lVert u\rVert_{q}^q+p^*\lvert u\rVert_{p^*}^{p^*}\big\}.
        \end{align*}
    Define
        \[m(a,\mu)=\inf_{u\in\mathcal{P}_{a,\mu}}E_{\mu}(u),\quad m^{\pm}(a,\mu)=\inf_{u\in\mathcal{P}_{a,\mu}^{\pm}}E_{\mu}(u),\]
    and
        \[m_{r}(a,\mu)=\inf_{u\in\mathcal{P}_{a,\mu}\cap W_{rad}^{1,p}(\mathbb{R}^N)}E_{\mu}(u),\quad m_{r}^{\pm}(a,\mu)=\inf_{u\in\mathcal{P}_{a,\mu}^{\pm}\cap W_{rad}^{1,p}(\mathbb{R}^N)}E_{\mu}(u).\]
    Obviously, by Definition \ref{grosta}, if we can prove $u$ is a critical point of $E_{\mu}|_{S_{a}}$ and $E_{\mu}(u)=m(a,\mu)$, then $u$ is a ground state of (\ref{equation}).

    Although we give the method to observe the structure of $E_{\mu}$ on $S_{a}$, the difficulty that arises is how do we get the compactness of PS(Palais-Smale) sequence.
    In \cite[lemma 2.9]{zzzz}, the authors considered the equation (\ref{Sovlevsub}) and proved the compactness of PS sequence, but the nonlinearities is Sobolev subcritical. When $p=2$, in \cite[proposition 3.1]{sn2}, the author proved a compactness lemma of PS sequence for (\ref{equation}): Assume that $\{u_{n}\}\in S_{a,r}=S_{a}\cap H^1(\mathbb{R}^N)$ is a PS sequence of $E_{\mu}|_{S_{a}}$ at level $c$. Furthermore we assume that $P_{\mu}(u_{n})\rightarrow 0$. Then, one of the following alternatives holds: up to a subsequence, either $u_{n}\rightharpoonup u$ in $H^{1}(\mathbb{R}^N)$ but not strongly, and
        \[E_{\mu}(u)\leqslant c-\frac{1}{N}S^{\frac{N}{2}};\]
    or $u_{n}\rightarrow u$ in $H^1(\mathbb{R}^N)$. However, the proof of \cite[proposition 3.1]{sn2} need to use the Br\'{e}zis-Lieb lemma\cite{bhle} and the linearity of Laplace operator. By using concentration compactness lemma(see \cite[section 4]{sm} or \cite[lemma 1.1]{lpl}) and referring to the idea of \cite{hdly}, we prove a compactness result of PS sequence similar to \cite[proposition 3.1]{sn2}. Therefore, our main goal is to obtain the PS sequence and exclude the case of weak convergence, thereby obtaining strong convergence.

    Now, we can state the existence results. Even though $\lVert u\rVert_{p^*}^{p^*}$ is always a $L^p$-supercritical term, there are some differences in the existence results when $\lVert u\rVert_{q}^q$ is a $L^p$-subcritical, critical or supercritical term. Therefore, we will state the existence results separately when $q<(=,>)p+\frac{p^2}{N}$.

    If $p<q<p+\frac{p^2}{N}$. Since $q\gamma_{q}<p$, the function $\Psi_{u}^{\mu}$ may have two critical points on $\mathbb{R}$(such as $f(s)=50e^{2s}-50e^s-e^{6s}$), one is a local minimum point and the other is global maximum point, we note them as $s_{u}$ and $t_{u}$ respectively. Moreover, it is not difficult to prove that $s_{u}\star u\in\mathcal{P}_{a,\mu}^{+}$ and $t_{u}\star u\in\mathcal{P}_{a,\mu}^{-}$. Of course, $\Psi_{u}^{\mu}$ may not have any critical points on $\mathbb{R}$(such as $f(s)=50e^{2s}-200e^s-e^{6s}$). Therefore, it is natural to speculate $E_{\mu}$ has two critical points on $S_{a}$ under appropriate assumptions, one is a local minimizer and is also a minimizer of $E_{\mu}$ on $\mathcal{P}_{a,\mu}^{+}$, the other is mountain-pass type critical point and is also a minimizer of $E_{\mu}$ on $\mathcal{P}_{a,\mu}^{+}$.

    Let
        \begin{equation*}
    	    C'=\bigg(\frac{p^*S^{p^*/p}(p-q\gamma_{q})}{p(p^*-q\gamma_{q})}\bigg)^{\frac{p-q\gamma_{q}}{p^*-p}}\frac{q(p^*-p)}{pC_{N,q}^q(p^*-q\gamma_{q})},
        \end{equation*}
    and
        \begin{equation*}
    	    C''=\frac{pp^*}{N\gamma_{q}C_{N,q}^q(p^*-p)}\bigg(\frac{q\gamma_{q}S^{N/p}}{p-q\gamma_{q}}\bigg)^{\frac{p-q\gamma_{q}}{p}},
        \end{equation*}
    where $S$ is the optimal constant of Sobolev inequality
        \[S\lVert u\rVert_{p^*}^p\leqslant\lVert\nabla u\rVert_{p}^p\quad\forall u\in D^{1,p}(\mathbb{R}^N).\]
    Define
        \begin{equation}\label{alpha}
    	    \alpha(N,p,q):=\min\{C',C''\}.
        \end{equation}
    Then, the existence  result of local minimizer for $p<q<\frac{p^2}{N}$ can be stated as follows.
    \begin{theorem}\label{th1}
    	Let $N\geqslant 2$, $1<p<N$, $p<q<p+\frac{p^2}{N}$, and $a,\mu>0$. Assume that
    	    \begin{equation}\label{muconsub}
    	    	\mu a^{q(1-\gamma_{q})}<\alpha(N,p,q),
    	    \end{equation}
        then $E_{\mu}|_{S_{a}}$ has a ground state $u_{a,\mu}^{+}$ which is positive, radially symmetric, radially non-increasing, and solves {\rm (\ref{equation})} for some $\lambda_{a,\mu}^{+}<0$. Moreover,
            \[E_{\mu}(u_{a,\mu}^{+})=m(a,\mu)=m^{+}(a,\mu)<0,\]
        and $u_{a,\mu}^{+}$ is a local minimizer of $E_{\mu}$ on the set
            \[A_{k}:=\big\{u\in S_{a}: \lVert\nabla u\rVert_{p}\leqslant k\big\},\]
        for a suitable $k>0$ sufficiently small. Any other ground state of $E_{\mu}|_{S_{a}}$ is a local minimizer of $E_{\mu}|_{A_{k}}$.
    \end{theorem}

    In addition to guaranteeing that $\Psi_{u}^{\mu}$ has two critical points on $\mathbb{R}$(in fact, this conclusion can be obtained by $\mu a^{q(1-\gamma_{q})}<C'$), another important reason we assume (\ref{muconsub}) is to ensure the convergence of the PS sequence by using compactness lemma we have obtained.
    	
    The existence result of mountain-pass type solution for $p<q<p+\frac{p^2}{N}$ can be stated as follows.
    \begin{theorem}\label{th2}
    	Let $N\geqslant 2$, $1<p<N$, $p<q<p+\frac{p^2}{N}$, and $a,\mu>0$ satisfies {\rm(\ref{muconsub})}. Further we assume that $N\geqslant p^2$ or $N<p^2<9$. Then $E_{\mu}|_{S_{a}}$ has a critical point of mountain-pass type $u_{a,\mu}^{-}$ which is positive, radially symmetric, radially non-increasing, and solves {\rm (\ref{equation})} for some $\lambda_{a,\mu}^{-}<0$. Moreover, $E_{\mu}(u_{a,\mu}^{-})=m^{-}(a,\mu)$ and
    	    \[0<m^{-}(a,\mu)<m^{+}(a,\mu)+\frac{1}{N}S^{\frac{N}{p}}.\]
    \end{theorem}

    In order to use compactness lemma to obtain the convergence of PS sequences, the strict inequality $m^{-}(a,\mu)<m^{+}(a,\mu)+\frac{1}{N}$ is a crucial step in our proof. Here we refer to the ideas of \cite{gajpai2,wjwy} to prove the strict inequality. However, there are some difficulties to obtain the inequality for $N<p^2$ and $p\geqslant 3$, we don't know whether the result hold.

    For $q\geqslant p+\frac{p^2}{N}$. Since $q\gamma_{q}\geqslant p$ and hence $\Psi_{u}^{\mu}(s)$ has unique critical point $t_{\mu}$ of mountain-pass type under suitable assumptions of $\mu$ if $q=p+\frac{p^2}{N}$. Moreover, $t_{u}\star u\in\mathcal{P}_{a,\mu}^{-}$. Therefore, it is natural to speculate that $E_{\mu}|_{S_{a}}$ has a critical point of mountain-pass type which is also a minimizer of $E_{\mu}$ on $\mathcal{P}_{a,\mu}^{-}$.
    \begin{theorem}\label{th3}
    	Let $N\geqslant p^{\frac{3}{2}}$, $1<p<N$, $p+\frac{p^2}{N}\leqslant q<p^*$, and $a,\mu >0$. Further we assume that
    	    \begin{equation}\label{muconcri}
    	    	\mu a^{\frac{p^2}{N}}<\frac{q}{pC_{N,p,q}^q},
    	    \end{equation}
        if $q=p+\frac{p^2}{N}$. Then $E_{\mu}|_{S_{a}}$ has a ground state $u_{a,\mu}^{-}$ which is positive, radially symmetric, radially non-increasing, and solves {\rm (\ref{equation})} for some $\lambda_{a,\mu}^{-}<0$. Moreover, $u_{a,\mu}^{-}$ is a critical point of mountain pass-type and
            \[E_{\mu}(u_{a,\mu}^{-})=m^{-}(a,\mu)=m(a,\mu)\in\Big(0,\frac{1}{N}S^{\frac{N}{p}}\Big).\]
    \end{theorem}

    Similar to Theorem \ref{th2}, to obtain the convergence of PS sequence, the strict inequality $m^{-}(a,\mu)<\frac{1}{N}S^{\frac{N}{p}}$ plays an important role. However, obtaining this result seems somewhat difficult for $N<p^{\frac{3}{2}}$ by classical method(see \cite{bhnl,sn2}).

    Now, we start to analyze the asymptotic behavior of $u_{a,\mu}^{\pm}$ as $\mu\rightarrow 0$ and $\mu$ goes it's upper bound. To state these results, let us first introduce some necessary notations. Through a scaling, we know the equation
        \begin{equation}\label{uniqueness}
    	    -\Delta_{p}u+u^{p-1}=u^{q-1}\quad\mbox{in}\ \mathbb{R}^N,
        \end{equation}
    has a radial solution $u$ which is non-negative. Similar to \cite[theorem A.1]{gajpai2}, we can prove $u\in L_{loc}^{\infty}(\mathbb{R}^N)$. Then, by regularity result\cite{tp}, we know $u\in C_{loc}^{1,\alpha}$ for some $\alpha\in(0,1)$. Thus, by \cite{sjtm}, (\ref{uniqueness}) has unique radial ``ground state" $\phi_{0}$. Here, in \cite{sjtm}, the meaning of ``ground state" is a non-negative non-trivial $C^1$ distribution solution of (\ref{uniqueness}).

    The asymptotic result as $\mu\rightarrow 0$ can be stated as follows.
    \begin{theorem}\label{th4}
    	Let $N\geqslant 2$, $1<p<N$, $p<q<p^*$, $a>0$ and $\mu>0$ sufficiently small. Let $u_{a,\mu}^{+}$ be the local minimizer which is obtained by Theorem {\rm \ref{th1}} and $u_{a,\mu}^{-}$ be the mountain-pass solution which is obtained by Theorem {\rm \ref{th2}} and {\rm {\ref{th3}}}. Then,
    	
    	\noindent{\rm (1)}
    	\begin{minipage}[t]{\linewidth}
    		We have
    		    \[\sigma_{0}^{\frac{1}{p-q}}\mu^{-\frac{N}{p(p-q\gamma_{q})}}u_{a,\mu}^{+}\Big(\sigma_{0}^{-\frac{1}{p}}\mu^{-\frac{1}{p-q\gamma_{q}}}\cdot\Big)\rightarrow\phi_{0}\]
    		in $W^{1,p}(\mathbb{R}^N)$ as $\mu\rightarrow 0$, where $\phi_{0}$ is the unique radial ``ground state" solution of {\rm (\ref{uniqueness})} and
    		    \[\sigma_{0}=\Big(\frac{a^p}{\lVert\phi_{0}\rVert_{p}^p}\Big)^{\frac{p(q-p)}{p^2-N(q-p)}}.\]
    	\end{minipage}

        \noindent{\rm (2)}
        \begin{minipage}[t]{\linewidth}
        	For $N\leqslant p^2$, there exists $\sigma_{\mu}>0$ such that
        	\[w_{\mu}=\sigma_{\mu}^{\frac{N-p}{p}}u_{\mu}^{-}(\sigma_{\mu}\cdot)\rightarrow U_{\varepsilon_{0}}\]
        	in $D^{1,p}(\mathbb{R}^N)$ as $\mu\rightarrow 0$ for some $\varepsilon_{0}>0$, where $U_{\varepsilon_{0}}$ is given by {\rm (\ref{Udf})} and $\sigma_{\mu}\rightarrow 0$ as $\mu\rightarrow 0$.
        \end{minipage}

        \noindent{\rm (3)}
         \begin{minipage}[t]{\linewidth}
        	For $N>p^2$, we have $u_{a,\mu}^{-}\rightarrow U_{\varepsilon_{0}}$ in $W^{1,p}(\mathbb{R}^N)$ as $\mu\rightarrow 0$, where $U_{\varepsilon_{0}}$ satisfies $\lVert U_{\varepsilon_{0}}\rVert_{p}^p=a^p$.
        \end{minipage}
    \end{theorem}

    In fact, for $p<q<p+\frac{p^2}{N}$, we can prove $\lVert\nabla u_{a,\mu}^{+}\rVert_{p}^p\rightarrow 0$ and $m^{+}(a,\mu)\rightarrow 0$ as $\mu\rightarrow 0$. Thus, $\{u_{a,\mu}^{+}\}$ does not convergence strongly in $W^{1,p}(\mathbb{R}^N)$. For $p+\frac{p^2}{N}<q<p^*$, we can prove
        \[\lVert\nabla u_{a,\mu}^{-}\rVert_{p}^p,\ \ \lVert u_{a,\mu}^{-}\rVert_{p^*}^{p^*}\rightarrow S^{\frac{N}{p}}\quad and\quad m^{-}\rightarrow\frac{1}{N}S^{\frac{N}{p}}\]
    as $\mu\rightarrow 0$ which implies $\{u_{a,\mu}^{-}\}$ is a minimizing sequence of minimizing problem
        \[S=\inf_{u\in D^{1,p}(\mathbb{R}^N)\backslash\{0\}}\frac{\lVert\nabla u\rVert_{p}^p}{\lVert u\rVert_{p^*}^{p^*}}.\]
    When $N>p^2$, we can obtain strongly convergence result. But if $N\leqslant p^2$, since $U_{\varepsilon}\notin L^p(\mathbb{R}^N)$ for all $\varepsilon>0$, it is posible to obtain $u_{a,\mu}^{-}\rightarrow U_{\varepsilon_{0}}$ in $W^{1,p}(\mathbb{R}^N)$.

    Next we analyze the asymptotic behavior as $\mu$ goes its upper bound. Let
        \[\bar{\alpha}=\frac{1}{pa^{p^2/N}C_{N,p,q}^q},\]
    when $q=p+\frac{p^2}{N}$. We have following asymptotic result.
    \begin{theorem}\label{th5}
    	Let $N\geqslant 2$, $1<p<N$, $p+\frac{p^2}{N}\leqslant q<p^*$, and $a,\mu>0$ satisfies  {\rm (\ref{muconcri})} if $q=p+\frac{p^2}{N}$. Let $u_{a,\mu}^{+}$ be the mountain-pass solution which is obtained by Theorem {\rm \ref{th3}}. Then,
    	
    	\noindent{\rm (1)}
    	\begin{minipage}[t]{\linewidth}
    		For $q=p+\frac{p^2}{N}$, we have
    		    \[(\bar{\alpha}\sigma_{0})^{\frac{1}{p-q}}s_{\mu}^{\frac{N}{p}}u_{\mu}^{-}\Big(\sigma_{0}^{-\frac{1}{p}}s_{\mu}\cdot\Big)\rightarrow\phi_{0}\]
            in $W^{1,p}(\mathbb{R}^N)$ as $\mu\rightarrow\bar{\alpha}$, where $\phi_{0}$ is the unique radial "ground state" solution of {\rm (\ref{uniqueness})}, $s_{\mu}=(\bar{\alpha}-\mu)^{-(N-p)/p^2}$ and
                \[\sigma_{0}=\bar{\alpha}^{\frac{p^2}{N(q-p)=p^2}}\bigg(\frac{a^p}{\lVert\phi_{0}\rVert_{p}^p}\bigg)^{\frac{p(q-p)}{p^2-N(q-p)}}.\]
    	\end{minipage}

        \noindent{\rm (2)}
        \begin{minipage}[t]{\linewidth}
            For $p+\frac{p^2}{N}<q<p^*$, we have
        	    \[\sigma_{0}^{\frac{1}{p-q}}\mu^{-\frac{N}{p(p-q\gamma_{q})}}u_{a,\mu}^{+}\Big(\sigma_{0}^{-\frac{1}{p}}\mu^{-\frac{1}{p-q\gamma_{q}}}\cdot\Big)\rightarrow\phi_{0}\]
        	in $W^{1,p}(\mathbb{R}^N)$ as $\mu\rightarrow+\infty$, where $\phi_{0}$ is the unique radial "ground state" solution of {\rm (\ref{uniqueness})} and
        	    \[\sigma_{0}=\Big(\frac{a^p}{\lVert\phi_{0}\rVert_{p}^p}\Big)^{\frac{p(q-p)}{p^2-N(q-p)}}.\]
        \end{minipage}
    \end{theorem}

    For $q=p+\frac{p^2}{N}$, we can prove $m_{a,\mu}^{-}=0$ when $\mu\geq\bar{\alpha}$(see Lemma \ref{asycrinonexi}) and hence $u_{a,\mu}^{-}$ does not exist. By Theorem \ref{th3}, we know $u_{a,\mu}^{-}$ exist when $\mu<\bar{\alpha}$. Thus, $\bar{\alpha}$ is the sharp constant such that $u_{a,\mu}^{-}$ exist and we can analyze the asymptotic behavior as $\mu\rightarrow\bar{\alpha}$. However, For $p<q<p+\frac{p^2}{N}$, we can not claim that $\alpha(N,p,q)$ which is given by (\ref{alpha}) is the sharp constant such that $u_{a,\mu}^{\pm}$ exist. From our proof, it seems that $\alpha(N,p,q)$ is not optimal(see Section \ref{exisub}). Thus, it is impossible to study the asymptotic behavior as $\mu$ goes it's upper when $p<q<p+\frac{p^2}{N}$.

    Finally, we want to investigate the nonexistence and multiplicity of equation (\ref{equation}). The nonexistence result is attained for $\mu<0$.
    \begin{theorem}\label{th6}
    	Let $N\geqslant 2$, $1<p<N$, $p<q<p^*$, $a>0$ and $\mu<0$. Then
    	
    	\noindent{\rm (1)}
    	    \begin{minipage}[t]{\linewidth}
    	        If $u$ is a critical point of $E_{\mu}|_{S_{a}}$(not necessary positive), then then associated Lagrange multiplier $\lambda$ is positive, and $E_{\mu}(u)>S^{\frac{N}{p}}/N$.
    	    \end{minipage}

    	\noindent{\rm (2)}
    	    \begin{minipage}[t]{\linewidth}
    	    	The problem
    	    	    \begin{equation*}
    	    		    -\Delta_{p}u=\lambda u^{p-1}+\mu u^{q-1}+u^{p^*-1},\quad u>0\quad in\ \mathbb{R}^N
    	    	    \end{equation*}
    	    	has no solution $u\in W^{1,p}(\mathbb{R}^N)$ for any $\mu<0$.
    	    \end{minipage}
    \end{theorem}

    The reason why we speculate equation (\ref{equation}) has no solution is due to Theorem \ref{th6}(1). In the previous existence results, $c<\frac{1}{N}S^{\frac{N}{p}}$ is crucial to obtain the convergence of $\mbox{(PS)}_{c}$ sequence. Moreover, there is an example such that a $\mbox{(PS)}_{c}$ sequence with $c=\frac{1}{N}S^{\frac{N}{p}}$ without any convergent subsequence(see \cite{bh}). Thus, it is natural to guess that equation (\ref{equation}) has no solution.

    We use genus theory to prove multiplicity result. One of the crucial problem is that the functional should bounded from below when we use genus theory, but due to $\lVert u\rVert_{p^*}^{p^*}$ is a $L^2$-supercritical term, $E_{\mu}|_{S_{a}}(u)$ does not bounded from below. To overcome this problem, we introduce a truncation function to complete the proof.

    Let
        \[S_{a,r}=S_{a}\cap W_{rad}^{1,p}(\mathbb{R}^N).\]
    The multiplicity result can be stated as follows.
    \begin{theorem}\label{th7}
    	Let $N\geqslant 3$, $2<p<N$, $p<q<p+\frac{p^2}{N}$ and $a,\mu>0$ satisfies {\rm(\ref{muconsub})}. Then
    	equation $(\ref{equation})$ has infinitely many solutions on $S_{a,r}$ at negative levels.
    \end{theorem}

    In Theorem \ref{th7} we assume that $p>2$, since the quantitative deformation lemma\cite[lemma 5.15]{wm} will be used in the proof and it requires $\lVert u\rVert_{p}^p\in C^2(L^p(\mathbb{R}^N,\mathbb{R}))$. Therefore, we need $p>2$.

    \noindent\textbf{Notations.} Throughout this paper, $C$ is indiscriminately used to denote various absolutely positive constants. $a\sim b$ means that there exist $C>1$ such that $C^{-1}a\leqslant b\leqslant Ca$.

\section{\textbf{Prelimilaries}}\label{prelimilaries}
    In this section, we collect some results which will be used in the rest of the paper.  Firstly, Let us recall the Sobolev inequality.
    \begin{lemma}
    	For every $N\geqslant 2$ and $1<p<N$, there exists an optimal constant $S$ depends on $N$ and $p$ such that
    	    \[S\lVert u\rVert_{p^*}^p\leqslant\lVert\nabla u\rVert_{p}^p\quad\forall u\in D^{1,p}(\mathbb{R}^N),\]
    	where $D^{1,p}(\mathbb{R}^N)$ denotes the completion of $C_{c}^{\infty}(\mathbb{R}^N)$ with respect to the norm $\lVert u\rVert_{D^{1,p}}:=\lVert\nabla u\rVert_{p}^p$.
    \end{lemma}

    It is well knwon\cite{tg} that the optimal constant is attained by
        \begin{equation}\label{Udf}
       U_{\varepsilon,y}=d_{N,p}\varepsilon^{\frac{N-p}{p(p-1)}}\Big(\varepsilon^{\frac{p}{p-1}}+\lvert x-y\rvert^{\frac{p}{p-1}}\Big)^{\frac{p-N}{p}},
       \end{equation}
    where $\varepsilon>0$, $y\in\mathbb{R}^N$ and $d_{N,p}>0$ depends on $N$ and $p$ such that $U_{\varepsilon,y}$ satisfies
        \[-\Delta_{p}u=u^{p*-1},\quad u>0\quad\mbox{in}\ \mathbb{R}^N,\]
    and hence
        \[\lVert\nabla U_{\varepsilon,y}\rVert_{p}^p=\lVert U_{\varepsilon,y}\rVert_{p^*}^{p^*}=S^{\frac{N}{p}}.\]
    If $y=0$, we set $U_{\varepsilon}=U_{\varepsilon,0}$.

    Next, we introduce the Pohozaev identity for $p$-Laplacian.
    \begin{lemma}\label{Pohozaevf}
    	{\rm\cite{jlsm}} Assume that $N\geqslant 2$, $1<p<N$, $f\in C(\mathbb{R},\mathbb{R})$ such that $f(0)=0$ and let $u$ be a local weak solution of
    	    \[-\Delta_{p}u=f(u)\quad\mbox{in}\ D'(\mathbb{R}^N),\]
    	where $D(\mathbb{R}^N)=C_{c}^{\infty}(\mathbb{R}^N)$, and $D'(\mathbb{R}^N)$ is the dual space of $D(\mathbb{R})$. Suppose that
    	    \[u\in L_{loc}^{\infty}(\mathbb{R}^N),\quad\nabla u\in L^p(\mathbb{R}^N),\quad\mbox{and}\quad F(u)\in L^1(\mathbb{R}^N).\]
    	Then $u$ satisfies
    	    \[(N-p)\lVert\nabla u\rVert_{p}^p=Np\int_{\mathbb{R}^N}F(u)dx.\]
    \end{lemma}

    By the Pohozaev identity, we can prove that all critical points belong to the Pohozaev manifold.
    \begin{proposition}\label{Pohozaev}
    	Assume that $u\in S_{a}$ is a solution to {\rm(\ref{equation})}, then $u\in\mathcal{P}_{a,\mu}$.
    \end{proposition}
    \begin{proof}
    	Similar to \cite[lemma A1]{gajpai2}, we can prove that $u\in L_{loc}^{\infty}(\mathbb{R}^N)$. It is clear that
    	    \[\nabla u\in L^p(\mathbb{R}^N)\quad\mbox{and}\quad F(u)=\frac{\lambda}{p}\lvert u\rvert^p+\frac{\mu}{q}\lvert u\rvert^q+\frac{1}{p^*}\lvert u\rvert^{p^*}\in L^1(\mathbb{R}^N).\]
    	Thus, by Lemma \ref{Pohozaevf}, we have
    	    \begin{equation}\label{pohozaev}
    	    	(N-p)\lVert\nabla u\rVert_{p}^p=\lambda Na^p+\frac{\mu Np}{q}\lVert u\rVert_{q}^q+(N-p)\lVert u\rVert_{p^*}^{p^*}.
    	    \end{equation}
    	Using the euqation (\ref{equation}), we have
    	    \[\lVert\nabla u\rVert_{p}^p=\lambda a^p+\mu\lVert u\rVert_{q}^q+\lVert u\rVert_{p^*}^{P^*},\]
    	which together with (\ref{pohozaev}), implies
    	    \[\lVert\nabla u\rVert_{p}^p=\mu\gamma_{q}\lVert u\rVert_{q}^q+\lVert u\rVert_{p^*}^{p^*}.\]
    \end{proof}

\section{\textbf{Compactness of PS sequence}}\label{compactness}
    In this section, we prove a compactness lemma of PS sequence under suitable assumptions. This is a crucial result to obtain the existence of critical point for $E_{\mu}|_{S_{a}}$.

	\begin{proposition}\label{compactnesslemma}
		Let $N\geqslant 2$, $p<q<p^*$, and $a,\mu>0$. Let $\{u_{n}\}\subset S_{a,r}$ be a PS sequence for $E_{\mu}|_{S_{a}}$ at level $c$, with
		    \begin{equation*}
		    	c<\frac{1}{N}S^{\frac{N}{p}}\quad\mbox{and}\quad c\neq 0.
		    \end{equation*}
	    Suppose in addition that $P_{\mu}(u_{n})\rightarrow 0$ as $n\rightarrow\infty$. Then, one of the following two conclusions is true:
	    { \rm(i)}
	        \begin{minipage}[t]{\linewidth}
	     	   either up to a sequence $u_{n}\rightharpoonup u$ in $W^{1,p}(\mathbb{R}^N)$ but not strongly, where $u\not\equiv 0$ is a solution to {\rm(\ref{equation})} for some $\lambda<0$, and
	    	        \begin{equation*}
	    	    	    E_{\mu}(u)<c-\frac{1}{N}S^{\frac{N}{p}};
	    	        \end{equation*}
	        \end{minipage}
        {\rm (ii)}
            \begin{minipage}[t]{\linewidth}
            	or up to a subsequence $u_{n}\rightarrow u$ in $W^{1,p}(\mathbb{R}^N)$, $E_{\mu}(u)=m$, and $u$ solves {\rm (\ref{equation})} for some $\lambda<0$.
            \end{minipage}
	\end{proposition}

    Since the proof of Proposition \ref{compactnesslemma} is relatively long, We will divide the proof into some lemmas.
    \begin{lemma}\label{bounded}
    	$\{u_{n}\}$ is bounded in $W^{1,p}(\mathbb{R}^N)$.
    \end{lemma}
    \begin{proof}
    	At first, we assume that $p<q<p+\frac{p^2}{N}$, so that $q\gamma_{q}<p$. Since $P_{\mu}(u_{n})\rightarrow 0$, by the Gagliardo-Nirenberg inequality, we have
    	    \begin{align*}
    	    	E_{\mu}(u_{n})&=\frac{1}{N}\lVert\nabla u_{n}\rVert_{p}^p-\mu\gamma_{q}\Big(\frac{1}{q\gamma_{q}}-\frac{1}{p^*}\Big)\lVert u_{n}\rVert_{q}^q+o_{n}(1)\\
    	    	&\geqslant\frac{1}{N}\lVert\nabla u_{n}\rVert_{p}^p-\mu\gamma_{q}\Big(\frac{1}{q\gamma_{q}}-\frac{1}{p^*}\Big)C_{N,p,q}^qa^{q(1-\gamma_{q})}\lVert\nabla u_{n}\rVert_{q}^{q\gamma_{q}}+o_{n}(1).
    	    \end{align*}
        Then, using $E_{\mu}(u_{n})\rightarrow c$ as $n\rightarrow\infty$, we deduce that $\{u_{n}\}$ is bounded in $W^{1,p}(\mathbb{R}^N)$.

        Now, let $q=p+\frac{p^2}{N}$, so that $q\gamma_{q}=p$. Then $P_{\mu}(u_{n})\rightarrow 0$ gives
            \begin{equation*}
            	E_{\mu}(u_{n})=\frac{1}{N}\lVert u_{n}\rVert_{p^*}^{p^*}+o_{n}(1),
            \end{equation*}
        which implies $\{u_{n}\}$ is bounded in $L^{p^*}(\mathbb{R}^N)$. By the H\"older inequality
            \begin{equation*}
            	\lVert u_{n}\rVert_{q}^q\leqslant\lVert u_{n}\rVert_{p^*}^{q\gamma_{q}}\lVert u_{n}\rVert_{p}^{q(1-\gamma_{q})}=a^{q(1-\gamma_{q})}\lVert u_{n}\rVert_{p^*}^{q\gamma_{q}},
            \end{equation*}
        we obtain $\{u_{n}\}$ is bounded in $L^{q}(\mathbb{R}^N)$. Using again that $P_{\mu}(u_{n})\rightarrow 0$, we know $\{u_{n}\}$ is bounded in $W^{1,p}(\mathbb{R}^N)$.

        Finally, let $p+\frac{p^2}{N}<q<p^*$, so that $q\gamma_{q}>p$. Since $P_{\mu}(u_{n})\rightarrow 0$, we have
            \begin{equation*}
            	E_{\mu}(u_{n})=\mu\gamma_{q}\Big(\frac{1}{p}-\frac{1}{q\gamma_{q}}\Big)\lVert u_{n}\rVert_{q}^q+\frac{1}{N}\lVert u_{n}\rVert_{p^*}^{p^*}+o_{n}(1),
            \end{equation*}
        and the cofficient of $\lVert u_{n}\rVert_{q}^q$ is positive. Therefore, $\{\lVert u_{n}\rVert_{q}\}$ and $\{\lVert u_{n}\rVert_{p^*}\}$ are both bounded which implies $\{\lVert\nabla u_{n}\rVert_{p}\}$ is bounded, since $P_{\mu}(u_{n})\rightarrow 0$.
    \end{proof}

    Now, we can state the concentration compactness lemma of $\{u_{n}\}$, the proof can be found in \cite[section 4]{sm} and \cite[lemma 1.1]{lpl}.
    \begin{lemma}\label{concentration}
    	Suppose $u_{n}\rightharpoonup u$ in $W^{1,p}(\mathbb{R}^N)$ and $\lvert\nabla u_{n}\rvert^p\rightharpoonup\kappa, \lvert u_{n}\rvert^{p^*}\rightharpoonup\nu$ in the sense of measures where $\kappa$ and $\nu$ are bounded non-negative measures on $\mathbb{R}^N$. Then, for some at most countable set $J$, we have
    	    \begin{equation*}
    	    	\kappa\geqslant\lvert\nabla u\rvert^p+\sum_{j\in J}\kappa_{j}\delta_{x_{j}},\quad\nu=\lvert u\rvert^{p^*}+\sum_{j\in J}\nu_{j}\delta_{x_{j}},
    	    \end{equation*}
        where $\kappa_{j}, \nu_{j}>0$ satisfies $S\nu_{j}^{\frac{p}{p^*}}\leqslant\kappa_{j}$ for all $j\in J$.
    \end{lemma}

    \begin{lemma}\label{concentrationnorm}
    	There is
    	    \begin{equation}\label{pstar}
    	    	\lim_{n\rightarrow\infty}\lVert u_{n}\rVert_{p^*}^{p^*}=\lVert u\rVert_{p^*}^{p^*}+\sum_{j\in J}\nu_{j}.
    	    \end{equation}
    \end{lemma}
	\begin{proof}
		For every $R>0$, let $\varphi_{R}\in C_{c}^{\infty}(\mathbb{R}^N)$ be such that
		    \begin{equation*}
		    	0\leqslant\varphi_{R}\leqslant 1,\quad\varphi_{R}=1\ \mbox{for}\ \lvert x\rvert\leqslant R,\quad\mbox{and}\quad\varphi_{R}=0\ \mbox{for}\ \lvert x\rvert\geqslant R+1.
		    \end{equation*}
		 Since $\lvert u_{n}\rvert^{p^*}\rightharpoonup\nu$, we have
		    \begin{align*}
		    	\lim_{n\rightarrow\infty}\int_{\mathbb{R}^N}\lvert u_{n}\rvert^{p^*}dx&=\lim_{n\rightarrow\infty}\Big(\int_{\mathbb{R}^N}\lvert u_{n}\rvert^{p^*}\varphi_{R}dx+\int_{\mathbb{R}^N}\lvert u_{n}\rvert^{p^*}(1-\varphi_{R})dx\Big)\\
		    	&=\int_{\mathbb{R}^N}\lvert u\rvert^{p^*}\varphi_{R}dx+\sum_{j\in J}\varphi_{R}(x_{j})\nu_{j}+\lim_{n\rightarrow\infty}\int_{\mathbb{R}^N}\lvert u_{n}\rvert^{p^*}(1-\varphi_{R})dx.
		    \end{align*}
	    Let $R\rightarrow +\infty$, by the Lebesgue dominated convergence theorem, we obtain
	        \begin{equation*}
	        	\lim_{n\rightarrow\infty}\int_{\mathbb{R}^N}\lvert u_{n}\rvert^{p^*}dx=\int_{\mathbb{R}^N}\lvert u\rvert^{p^*}dx+\sum_{j\in J}\nu_{j}+\lim_{R\rightarrow +\infty}\lim_{n\rightarrow\infty}\int_{\mathbb{R}^N}\lvert u_{n}\rvert^{p^*}(1-\varphi_{R})dx.
	        \end{equation*}
        Now, we prove that
            \begin{equation}\label{Rinfty}
            	\lim_{R\rightarrow +\infty}\lim_{n\rightarrow\infty}\int_{\mathbb{R}^N}\lvert u_{n}\rvert^{p^*}(1-\varphi_{R})dx=0,
            \end{equation}
        which leads to (\ref{pstar}). Since $\{u_{n}\}$ is bounded in $W_{rad}^{1,p}(\mathbb{R}^N)$, we have
            \begin{equation*}
            	\lvert u_{n}(x)\rvert\leqslant C\lvert x\rvert^{\frac{1-N}{p}}\quad \mbox{a.e. in}\ \mathbb{R}^N,
            \end{equation*}
        where $C>0$ is a constant independent of $n$. It follows that
            \begin{equation*}
            	\int_{\mathbb{R}^N}\lvert u_{n}\rvert^{p^*}(1-\varphi_{R})dx\leqslant\int_{\lvert x\rvert\geqslant R}\lvert u_{n}\rvert^{p^*}dx\leqslant CR^{\frac{N(1-p)}{N-p}},
            \end{equation*}
        which implies (\ref{Rinfty}) holds.
	\end{proof}

    \noindent\textbf{Proof of Proposition \ref{compactnesslemma}} Since $\{u_{n}\}$ is a bounded PS sequence for $E_{\mu}|_{S_{a}}$, there exists $\{\lambda_{n}\}\in\mathbb{R}$ such that for every $\psi\in W^{1,p}(\mathbb{R}^N)$,
        \begin{equation}\label{Lagrange}
        	\int_{\mathbb{R}^N}\Big(\lvert\nabla u_{n}\rvert^{p-2}\nabla u_{n}\cdot\nabla\psi-\lambda_{n}\lvert u_{n}\rvert^{p-2}u_{n}\psi-\mu\lvert u_{n}\rvert^{q-2}u_{n}\psi-\lvert u_{n}\rvert^{p^*-2}u_{n}\psi\Big)=o_{n}(1)\lVert\psi\rVert_{W^{1,p}}
        \end{equation}
    as $n\rightarrow\infty$. Choosing $\psi=u_{n}$, we deduce that $\{\lambda_{n}\}$ is bounded as well, and hence, up to a subsequence $\lambda_{n}\rightarrow\lambda\in\mathbb{R}$. Using the fact that $P_{\mu}(u_{n})\rightarrow 0$ and $\gamma_{q}<1$, we know that
        \begin{align}\label{lambda}
        	\lambda a^p&=\lim_{n\rightarrow\infty}\lambda_{n}\lVert u_{n}\rVert_{p}^{p}=\lim_{n\rightarrow\infty}\Big(\lVert\nabla u_{n}\rVert_{p}^{p}-\mu\lVert u_{n}\rVert_{q}^{q}-\lVert u_{n}\rVert_{p^*}^{p^*}\Big)\nonumber\\
        	&=\lim_{n\rightarrow\infty}\mu(\gamma_{q}-1)\lVert u_{n}\rVert_{q}^q=\mu(\gamma_{q}-1)\lVert u\rVert_{q}^q\leqslant 0,
        \end{align}
    with $\lambda=0$ is and only if $u\equiv 0$.

    We consider $\varphi_{\varepsilon}\in C_{c}^{\infty}(\mathbb{R}^N)$ such that
        \begin{equation*}
        	0\leqslant\varphi_{\varepsilon}\leqslant 1,\quad\varphi_{\varepsilon}=1\ \mbox{in}\ B_{\varepsilon}(x_{j}),\quad\varphi_{\varepsilon}=0\ \mbox{in}\ B_{2\varepsilon}(x_{j}),\quad\mbox{and}\quad\lvert\nabla\varphi_{\varepsilon}\rvert\leqslant\frac{2}{\varepsilon}.
        \end{equation*}
    It is clear that the sequence $\{\varphi_{\varepsilon}u_{n}\}$ is bounded in $W^{1,p}(\mathbb{R}^N)$, then, testing (\ref{Lagrange}) with $\psi=\varphi_{\varepsilon}u_{n}$, we obtain
        \begin{align}\label{phiepsilon}
        	&\lim_{n\rightarrow\infty}\int_{\mathbb{R}^N}\lvert\nabla u_{n}\rvert^{p-2}u_{n}\nabla u_{n}\cdot\nabla\varphi_{\varepsilon}dx\nonumber\\
        	=&\lim_{n\rightarrow\infty}\int_{\mathbb{R}^N}\Big(\lambda_{n}\lvert u_{n}\rvert^p\varphi_{\varepsilon}+\mu\lvert u_{n}\rvert^q\varphi_{\varepsilon}+\lvert u_{n}\rvert^{p^*}\varphi_{\varepsilon}-\lvert\nabla u_{n}\rvert^p\varphi_{\varepsilon}\Big)dx\nonumber\\
        	=&\lambda\int_{\mathbb{R}^N}\lvert u\rvert^p\varphi_{\varepsilon}dx+\mu\int_{\mathbb{R}^N}\lvert u\rvert^q\varphi_{\varepsilon}dx+\int_{\mathbb{R}^N}\varphi_{\varepsilon}d\nu-\int_{\mathbb{R}^N}\varphi_{\varepsilon}d\mu.
        \end{align}
    By the H\"older inequality,
        \begin{align*}
        	\Big|\int_{\mathbb{R}^N}\lvert\nabla u\rvert^{p-2}u_{n}\nabla u_{n}\cdot\nabla\varphi_{\varepsilon}dx\Big|&\leqslant\frac{2}{\varepsilon}\int_{B_{2\varepsilon}(x_{j})}\lvert\nabla u\rvert^{p-1}\lvert u_{n}\rvert dx\\
        	&\leqslant\frac{2}{\varepsilon}\Big(\int_{B_{2\varepsilon}(x_{j})}\lvert\nabla u_{u}\rvert^pdx\Big)^{\frac{p-1}{p}}\Big(\int_{B_{2\varepsilon}(x_{j})}\lvert u_{n}\rvert^pdx\Big)^{\frac{1}{p}}\\
        	&\leqslant\frac{C}{\varepsilon}\lVert u_{n}\rVert_{L^p(B_{2\varepsilon}(x_{j}))},
        \end{align*}
    where $C>0$ is a constant independent of $n$. Thus, using the H\"older inequality again, we have
        \begin{equation*}
        	\lim_{n\rightarrow\infty}\Big|\int_{\mathbb{R}^N}\lvert\nabla u\rvert^{p-2}u_{n}\nabla u_{n}\cdot\nabla\varphi_{\varepsilon}dx\Big|\leqslant\frac{C}{\varepsilon}\lVert u\rVert_{L^p(B_{2\varepsilon}(x_{j}))}\leqslant C\lVert u\rVert_{L^{p^*}(B_{2\varepsilon}(x_{j}))},
        \end{equation*}
    which implies
        \begin{equation*}
        	\lim_{n\rightarrow\infty}\int_{\mathbb{R}^N}\lvert\nabla u\rvert^{p-2}u_{n}\nabla u_{n}\cdot\nabla\varphi_{\varepsilon}dx\rightarrow 0
        \end{equation*}
	as $\varepsilon\rightarrow 0$.
	
	If $J\neq\emptyset$. Let $\varepsilon\rightarrow 0$ on both sides of (\ref{phiepsilon}), we obtain $\nu_{j}=\mu_{j}$. By Lemma \ref{concentration}, since $\mu_{j}\geqslant S\nu_{j}^{p/p^*}$, we have $\nu_{j}\geqslant S^{\frac{N}{p}}$. Therefore, by Lemma \ref{concentrationnorm},
	    \begin{equation*}
	    	c=\lim_{n\rightarrow\infty}E_{\mu}(u_{n})\geqslant E_{\mu}(u)+\Big(\frac{1}{p}-\frac{1}{p^*}\Big)\sum_{k\in J}\nu_{k}\geqslant E_{\mu}(u)+\frac{1}{N}S^{\frac{N}{p}}.
	    \end{equation*}
    Since $m<S^{N/p}/N$, so $E_{\mu}(u)<0$ which implies $u\not\equiv 0$. Following the idea of \cite[lemma 2.2]{yjf}, we can prove
	    \begin{equation*}
	    	\lvert\nabla u_{n}\rvert^{p-2}\nabla u_{n}\rightharpoonup\lvert\nabla u\rvert^{p-2}\nabla u\ in\ \big(L^p(\mathbb{R}^N)\big)^{*}.
	    \end{equation*}
    Thus, passing to the limit in (\ref{Lagrange}) by weak convergence, we know $u$ is a solution to (\ref{equation}). Now, case (i) in the proposition \ref{compactnesslemma} holds.

    If instead $J=\emptyset$, then the Br\'{e}zis-Lieb Lemma \cite{bhle} and (\ref{pstar}) implies $u_{n}\rightarrow u$ in $L^{p^*}(\mathbb{R}^N)$. Now, we prove $u\not\equiv 0$ and hence by (\ref{lambda}), we know $\lambda<0$. Suppose by contradiction that $u\equiv 0$. Then, by $P_{\mu}(u_{n})\rightarrow 0$,
        \[c=\lim_{n\rightarrow\infty}E_{\mu}(u_{n})=\lim_{n\rightarrow\infty}\bigg(\frac{1}{N}\lVert u_{n}\rVert_{p^*}^{p^*}+\mu\gamma_{q}\Big(\frac{1}{p}-\frac{1}{q\gamma_{q}}\Big)\lVert u_{n}\rVert_{q}^q\bigg)=0,\]
    which contradicts our assumptions. Let $T: W^{1,p}(\mathbb{R}^N)\rightarrow\big(W^{1,p}(\mathbb{R}^N)\big)^{*}$ be the mapping given by
        \begin{equation*}
        	<Tu,v>=\int_{\mathbb{R}^N}\Big(\lvert\nabla u\rvert^{p-2}\nabla u\cdot\nabla v-\lambda\lvert u\rvert^{p-2}uv\Big)dx.
        \end{equation*}
    Then, slightly modifying the proof in \cite[Lemma 3.6]{hdly}, we can derive that $u_{n}\rightarrow u$ in $W^{1,p}(\mathbb{R}^N)$, and hence case (ii) in the proposition \ref{compactnesslemma} holds.$\hfill\qed$

\section{\textbf{Existence result to the case $p<q<p+\frac{p^2}{N}$}}\label{exisub}

    In this section, we prove that under assumption (\ref{muconsub}), $E_{\mu}|_{S_{a}}$ has two critical points, one is a local minimizer and the other is a mountain-pass type solution. In order to prove this result, we need some properties of $E_{\mu}$ and $\mathcal{P}_{a,\mu}$ by analyzing the structure of $\Psi_{u}^{\mu}$.
\subsection{Some properties of $E_{\mu}$ and $\mathcal{P}_{a,\mu}$}
    \
    \newline
    \indent For Pohozaev manifold $\mathcal{P}_{a,\mu}$, we have the following properties.
     \begin{lemma}\label{empty}
    	$\mathcal{P}_{a,\mu}^{0}=\emptyset$, and $\mathcal{P}_{a,\mu}$ is a smooth manifold of codimension $2$ in $W^{1,p}(\mathbb{R}^N)$.
    \end{lemma}
    \begin{proof}
    	Let us assume that there exists $u\in\mathcal{P}_{a,\mu}^{0}$. Then, combining $P_{\mu}(u)=0$ and $\big(\Psi_{u}^{\mu}\big)''(0)=0$, we deduce that
    	\begin{equation*}
    		\mu\gamma_{q}(p-q\gamma_{q})\lVert u\rVert_{q}=(p^*-p)\lVert u\rVert_{p^*}^{p^*}.
    	\end{equation*}
    	Using this equation in $P_{\mu}(u)=0$, we obtain
    	\begin{equation}\label{nablapstar}
    		\lVert\nabla u\rVert_{p}^p=\frac{p^*-q\gamma_{q}}{p-q\gamma_{q}}\lVert u\rVert_{p^*}^{p^*}\leqslant\frac{S^{p^*/p}(p^*-q\gamma_{q})}{p-q\gamma_{q}}\lVert\nabla u\rVert_{p}^{p^*},
    	\end{equation}
    	and
    	\begin{equation}\label{nablaq}
    		\lVert\nabla u\rVert_{p}^p=\frac{\mu\gamma_{q}(p^*-q\gamma_{q})}{p^*-p}\lVert u\rVert_{q}^q\leqslant\frac{\mu\gamma_{q}(p^*-q\gamma_{q})}{p^*-p}a^{q(1-\gamma_{q})}\lVert\nabla u\rVert_{p}^{q\gamma_{q}}.
    	\end{equation}
    	From (\ref{nablapstar}) and (\ref{nablaq}), we infer that
    	\begin{equation*}
    		\bigg(\frac{p-q\gamma_{q}}{S^{p^*/p}(p^*-q\gamma_{q})}\bigg)^{\frac{1}{p^*-p}}\leqslant\bigg(\frac{\mu\gamma_{q}(p^*-q\gamma_{q})}{p^*-p}C_{N,q}^qa^{q(1-\gamma_{q})}\bigg)^{\frac{1}{p-q\gamma_{q}}},
    	\end{equation*}
    	that is
    	\begin{equation}\label{betacontradiction}
    		\mu a^{q(1-q\gamma_{q})}\geqslant\bigg(\frac{p-q\gamma_{q}}{S^{p^*/p}(p^*-q\gamma_{q})}\bigg)^{\frac{p-q\gamma_{q}}{p^*-p}}\frac{p^*-p}{C_{N,q}^q\gamma_{q}(p^*-q\gamma_{q})}.
    	\end{equation}
    	We can check that this is contradicts with (\ref{muconsub}): it is sufficient to verify that the right hand side in (\ref{muconsub}) is less than or equal to the right hand side in (\ref{betacontradiction}), and this is equivalent to
    	\begin{equation}\label{monotone}
    		\bigg(\frac{q\gamma_{q}}{p}\bigg)^{p^*-p}\bigg(\frac{p^*}{p}\bigg)^{p-q\gamma_{q}}\leqslant 1.
    	\end{equation}
    	Since the function $\varphi(t)=\log t/(t-1)$ is monotone decreasing on $(0,+\infty)$, we have
    	\begin{equation*}
    		\varphi\bigg(\frac{q\gamma_{q}}{p}\bigg)\leqslant\varphi\bigg(\frac{p^*}{p}\bigg),
    	\end{equation*}
    	that is (\ref{monotone}).
    	
    	Now, we can check that $\mathcal{P}_{a,\mu}$ is a smooth manifold of codimension $2$ in $W^{1,p}(\mathbb{R}^N)$. Recall the definition of $\mathcal{P}_{a,\mu}$, since $P_{\mu}(u)$ and $G(u)=\lVert u\rVert_{p}^p-a^p$ are of class $C^1$ in $W^{1,p}(\mathbb{R}^N)$, so we just show that the differential $\big(dG(u),dP_{\mu}(u)\big): W^{1,p}(\mathbb{R}^N)\longmapsto\mathbb{R}^2$ is surjective for every $u\in\mathcal{P}_{a,\mu}$. To this end, we prove that for every $u\in\mathcal{P}_{a,\mu}$ there exists $\varphi\in T_{u}S_{a}$ such that $dP_{\mu}(u)[\varphi]\neq 0$. Once such $\varphi$ exist, the system
    	\begin{equation*}
    		\left\{\begin{array}{l}
    			dG(u)[\alpha\varphi+\theta u]=x,\\
    			dP_{\mu}(u)[\alpha\varphi+\theta u]=y,
    		\end{array}\right.
    		\quad\Longleftrightarrow\quad\left\{\begin{array}{l}
    			a^p\theta=x,\\
    			\alpha dP_{\mu}(u)[\varphi]+\theta dP_{\mu}(u)[u]=y,
    		\end{array}\right.
    	\end{equation*}
    	is solvable with respect to $\alpha$ and $\theta$, for every $(x,y)\in\mathbb{R}^2$, and hence the surjectivity is proved. Suppose by contradiction that for $u\in\mathcal{P}_{a,\mu}$ such that $dP_{\mu}(u)[\varphi]=0$ for every $\varphi\in T_{u}S_{a}$. Then $u$ is a constrained critical point for the functional $P_{\mu}|_{S_{a}}$, and hence by the Lagrange multipliers rule there exists $\nu\in\mathbb{R}$ such that
    	\[-\Delta_{p}u=\nu\lvert u\rvert^{p-2}u+\frac{\mu q\gamma_{q}}{p}\lvert u\rvert^{q-2}u+\frac{p^*}{p}\lvert u\rvert^{p^*-2}u,\quad in\ \mathbb{R}^N.\]
    	But, by the Pohozaev identity, this implies that
    	\[p\lVert\nabla u\rVert_{p}^{p}=\mu q\gamma_{q}^2\lVert u\rVert_{q}^q+p^*\lVert u\rVert_{P^*}^{p^*},\]
    	that is $u\in\mathcal{P}_{a,\mu}^{0}$, a contradiction.
    \end{proof}
    By the H\"older inequality, we have
        \begin{equation}\label{Egreat}
    	    E_{\mu}(u)\geqslant\frac{1}{p}\lVert\nabla u\rVert_{p}^p-\frac{\mu}{q}C_{N,q}^qa^{q(1-\gamma_{q})}\lVert\nabla u\rVert_{p}^{q\gamma_{q}}-\frac{1}{p^*S^{p*/p}}\lVert\nabla u\rVert_{p}^{p^*}=h(\lVert\nabla u\rVert_{p}),
        \end{equation}
    where
        \begin{equation*}
    	    h(t)=\frac{1}{p}t^p-\frac{\mu}{q}C_{N,q}^qa^{q(1-\gamma_{q})}t^{q\gamma_{q}}-\frac{1}{p^*S^{p*/p}}t^{p^*}.
        \end{equation*}

    \begin{lemma}\label{hfunction}
    	Under assumption {\rm (\ref{muconsub})}, the function $h$ has exactly two critical points, one is a local minimum at negative level, and the other is a global maximum at positive level. In addition, there exists $R_{1}>R_{0}>0$ such that $h(R_{0})=h(R_{1})=0$, and $h(t)>0$ if and only if $t\in (R_{0},R(1))$.
    \end{lemma}
    \begin{proof}
    	For every $t>0$, we know $h(t)>0$ if and only if
    	    \begin{equation*}
    	    	\varphi (t)>\frac{\mu}{q}C_{N,q}^qa^{q(1-\gamma_{q})},\quad with\quad\varphi (t)=\frac{1}{p}t^{p-q\gamma_{q}}-\frac{1}{pS^{p^*/p}}t^{p^*-q\gamma_{q}}.
    	    \end{equation*}
        We can check that $\varphi$ has a unique critical point on $(0,+\infty)$, which is a global maximum point at positive level. The critical point is
            \begin{equation*}
            	\bar{t}=\bigg(\frac{p^*S^{p^*/p}(p-q\gamma_{q})}{p(p^*-q\gamma_{q})}\bigg)^{\frac{1}{p^*-p}},
            \end{equation*}
        and the maximum level is
            \begin{equation*}
            	\varphi (\bar{t})=\frac{p^*-p}{p(p^*-q\gamma_{q})}\bigg(\frac{p^*S^{p^*/p}(p-q\gamma_{q})}{p(p^*-q\gamma_{q})}\bigg)^{\frac{p-q\gamma_{q}}{p^*-p}}.
            \end{equation*}
        Therefore, $h$ is positive on $(R_{0},R_{1})$ if and only if
            \begin{equation*}
            	\varphi (\bar{t})>\frac{\mu}{q}C_{N,q}^qa^{q(1-\gamma_{q})},
            \end{equation*}
        that is $\mu a^{q(1-\gamma_{q})}<C'$. It follows that $h$ has a global maximum at positive level on $(R_{0},R_{1})$. Moreover, since $h(0^{+})=0^{-}$, there exists a local minimum point at negative level in $(0,R_{0})$. The fact that $h$ has no other critical points can be derived from $h'(t)=0$ with only two zeros.
    \end{proof}

    Using the properties of $h(t)$, we can analyze the structure of $\Psi_{u}^{\mu}$ and $E_{\mu}$.
    \begin{lemma}\label{structure}
    	For every $u\in S_{a}$, the function $\Psi_{u}^{\mu}$ has exactly two critical points $s_{u}<t_{u}$ and two zeros $c_{u}<d_{u}$, with $s_{u}<c_{u}<t_{u}<d_{u}$. Moreover,
    	
    	{\rm (i)}
    	\begin{minipage}[t]{\linewidth}
    		$s_{u}\star u\in\mathcal{P}_{a,\mu}^{+}, t_{u}\star u\in\mathcal{P}_{a,\mu}^{-}$, and if $s\star u\in\mathcal{P}_{a,\mu}$, then either $s=s_{u}$ or $s=t_{u}$.
    	\end{minipage}

        {\rm (ii)}
        \begin{minipage}[t]{\linewidth}
        	$\lVert\nabla(s\star u)\rVert_{p}\leqslant R_{0}$ for every $s\leqslant c_{u}$, and
        	    \begin{equation*}
        	    	E_{\mu}(s_{u}\star u)=\min\Big\{E_{\mu}(s\star u):s\in\mathbb{R},\ \lVert\nabla(s\star u)\rVert_{p}\leqslant R_{0}\Big\}<0.
        	    \end{equation*}
        \end{minipage}

        {\rm (iii)}
        \begin{minipage}[t]{\linewidth}
        	We have
        	    \[E_{\mu}(t_{u}\star u)=\max_{s\in\mathbb{R}}E_{\mu}(s\star u)>0.\]
        	and if $t_{u}<0$, then $P_{\mu}(u)<0$.
        \end{minipage}

        {\rm (iv)}
        \begin{minipage}[t]{\linewidth}
        	The maps $u\in S_{a}\longmapsto s_{u}\in\mathbb{R}$ and $u\in S_{a}\longmapsto t_{u}\in\mathbb{R}$ are of class $C^1$.
        \end{minipage}
    \end{lemma}
    \begin{proof}
    	By (\ref{Egreat}),
    	    \[\Psi_{u}^{\mu}(s)=E_{\mu}(s\star u)\geqslant h(\lVert\nabla(s\star u)\rVert_{p})=h(e^{s}\lVert\nabla u\rVert_{p}).\]
    	Thus, by Lemma \ref{hfunction}, $\Psi_{u}^{\mu}(s)$ is positive on $\big(\log(R_{0}/\lVert\nabla u\rVert_{p}),\log(R_{1}/\lVert\nabla u\rVert_{p})\big)$. It is clearly that $\Psi_{u}^{\mu}(-\infty)=0^{-}$ and $\Psi_{u}^{\mu}(+\infty)=-\infty$, hence $\Psi_{u}^{\mu}$ has at least two critical points $s_{u}<t_{u}$, with $s_{u}$ is local minimum point on $(-\infty,\log(R_{0}/\lVert\nabla u\rVert_{p}))$ at negative level, and $t_{u}$ is global maximum point at positive level. Now, we can check that there are no other critical points of $\Psi_{u}^{\mu}$. Indeed, the equation $\big(\Psi_{u}^{\mu}\big)'(s)=0$ has only two zeros. Now, the zero point theorem implies $\Psi_{u}^{\mu}$ has exactly two zeros $c_{u}<d_{u}$, with $s_{u}<c_{u}<t_{u}<d_{u}$.
    	
    	By Lemma \ref{Pohozaev}, $s\star u\in\mathcal{P}_{a,\mu}$ if and only if $\big(\Psi_{u}^{\mu}\big)'(s)=0$, that is either $s=s_{u}$ or $s=t_{u}$. Since $0$ is a local minimum of $\Psi_{s_{u}\star u}^{\mu}(s)\big(=\Psi_{u}^{\mu}(e^{s_{u}+s})\big)$, we have $\big(\Psi_{s_{u}\star u}^{\mu}\big)''(0)\geqslant 0$, which implies $s_{u}\star u\in\mathcal{P}_{a,\mu}^{+}\cup\mathcal{P}_{a,\mu}^{0}$. Lemma \ref{empty} gives $\mathcal{P}_{a,\mu}^{0}=\emptyset$, hence $s_{u}\star u\in\mathcal{P}_{a,\mu}^{+}$. In the same way $t_{u}\star u\in\mathcal{P}_{a,\mu}^{+}$.
    	
    	Since $\Psi_{u}^{\mu}$ is positive on $\big(\log(R_{0}/\lVert\nabla u\rVert_{p}),\log(R_{1}/\lVert\nabla u\rVert_{p})\big)$, we deduce that $c_{u}<\log(R_{0}/\lVert\nabla u\rVert_{p})$. Thus, $\lVert\nabla(c_{u}\star u)\rVert_{p}\leqslant R_{0}$ which implies $\lVert\nabla(s\star u)\rVert_{p}\leqslant R_{0}$ for all $s\leqslant c_{u}$. We know $s_{u}$ is a local minimum on $\big(-\infty,\log(R_{0}/\lVert\nabla u\rVert_{p})\big)$ at negative level, so
    	    \begin{align*}
    	    	E_{\mu}(s_{u}\star u)&=\Psi_{u}^{\mu}(s_{u})=\min\Big\{\Psi_{u}^{\mu}(s): s\leqslant\log\big( R_{0}/\lVert\nabla u\rVert_{p}\big)\Big\}\\
    	    	&=\min\Big\{E_{\mu}(s\star u): \lVert\nabla(s\star u)\rVert_{p}\leqslant R_{0}\Big\}.
    	    \end{align*}

        Since $t_{u}$ is a global maximum of $\Psi_{u}^{\mu}$, we conclude that $\Psi_{u}^{\mu}$ is decreasing on $(t_{u},+\infty)$. Therefore, $\big(\Psi_{u}^{\mu}\big)'(s)<0$ for all $s\in (t_{u},+\infty)$(since there are no critical points on $(t_{u},+\infty)$). If $t_{u}<0$, then $P_{\mu}(u)=\big(\Psi_{u}^{\mu}\big)'(0)<0$.

        It remain to show that $u\longmapsto s_{u}$ and $u\longmapsto t_{u}$ are class of $C^1$. Considering the function $\Phi(s,u):=\big(\Psi_{u}^{\mu}\big)'(s)$. Since $\Phi(s_{u},u)=0, \partial_{t}\Phi(s_{u},u)=\big(\Psi_{u}^{\mu}\big)''(s_{u})>0$, and it is possible to pass with continuity from $\mathcal{P}_{a,\mu}^{+}$ to $\mathcal{P}_{a,\mu}^{-}$(since $\mathcal{P}_{a,\mu}^{0}=\emptyset$), by the implicit theorem, we known $u\longmapsto s_{u}$ is of class $C^1$. The same argument proves that $u\longmapsto t_{u}$ is of class $C^1$.  	
    \end{proof}

\subsection{Existence of local minimizer}
    \
    \newline
    \indent In this subsection, we always assume that the assumptions of Theorem \ref{th1} hold. Set
        \begin{equation*}
    	    A_{k}=\big\{u\in S_{a}: \lVert\nabla u\rVert_{p}\leqslant k\big\},
        \end{equation*}
    where $k>0$ is a constant.

    \begin{lemma}\label{mplus}
    	$m^{+}(a,\mu)=\inf_{u\in A_{R_{0}}}E_{\mu}(u)<0<m^{-}(a,\mu)$.
    \end{lemma}
    \begin{proof}
    	For every $v\in\mathcal{P}_{a,\mu}^{+}$, by Lemma \ref{structure} (i) and (ii), we know $s_{v}=0$ and $\lVert\nabla v\rVert_{p}\leqslant R_{0}$. Thus,
    	    \[\inf_{u\in A_{R_{0}}}E_{\mu}(u)\leqslant E_{\mu}(v)<0,\]
    	which implies
    	    \[\inf_{u\in A_{R_{0}}}E_{\mu}(u)\leqslant m^{+}(a,\mu)<0.\]
    	For every $v\in A_{R_{0}}$, using again Lemma \ref{structure} (i) and (ii), we have $s_{v}\star v\in\mathcal{P}_{a,\mu}^{+}$ and $E_{\mu}(s_{v}\star v)\leqslant E_{\mu}(v)$. Hence,
    	    \[m^{+}(a,\mu)\leqslant E_{\mu}(s_{v}\star v)\leqslant E_{\mu}(v),\]
    	which implies
    	    \[m^{+}(a,\mu)\leqslant\inf_{u\in A_{R_{0}}}E_{\mu}(u).\]
    	
    	It remain to prove that $m^{-}(a,\mu)>0$. For every $v\in\mathcal{P}_{a,\mu}^{-}$, by Lemma \ref{structure} (i) and (iii), $t_{v}=0$ and $E_{\mu}(v)\geqslant E_{\mu}(s\star v)$ for all $s\in\mathbb{R}$. Now, using (\ref{Egreat}) and Lemma \ref{hfunction}, we have
    	    \[E_{\mu}(v)\geqslant\max_{s\in\mathbb{R}}E_{\mu}(s\star v)\geqslant\max_{s\in\mathbb{R}}h(e^s\lVert\nabla v\rVert_{p})=\max_{t>0}h(t)>0,\]
    	which implies
    	    \[m^{-}(a,\mu)\geqslant\max_{t>0}h(t)>0.\]
    \end{proof}

    By (\ref{Egreat}) and Lemma \ref{hfunction},
        \[\inf_{u\in\partial A_{R_{0}}}E_{\mu}(u)\geqslant\inf_{u\in\partial A_{R_{0}}}h(\lVert\nabla u\rVert_{p})=0.\]
    Therefore, if we can find a minimizer for $E_{\mu}|_{A_{R_{0}}}$, it must be a minimizer for $E_{\mu}|_{S_{a}}$.\\

    \noindent\textbf{Proof of Theorem \ref{th1}:} Let $\{v_{n}\}$ be a minimizing sequence for $E_{\mu}|_{A_{R_{0}}}$. It is not restrictive to assume that $v_{n}$ is radially symmetric and radially decreasing(if this is not the case, we can replace $v_{n}$ with $\lvert v_{n}\rvert^*$, the Schwarz rearrangement of $\lvert v_{n}\rvert$, and we obtain another minimizing sequence for $E_{\mu}|_{A_{R_{0}}}$). Furthermore, by Lemma \ref{mplus}, we can assume that $v_{n}\in\mathcal{P}_{a,\mu}^{+}$.

    Now, the Ekeland's variational principle gives a new minimizing sequence $\{u_{n}\}\subset A_{R_{0}}$ which is also a PS sequence for $E_{\mu}|_{S_{a}}$, with the property that $\lVert u_{n}-v_{n}\rVert\rightarrow 0$ as $n\rightarrow\infty$. The condition $\lVert u_{n}-v_{n}\rVert\rightarrow 0$ implies that $P_{\mu}(u_{n})=P_{\mu}(v_{n})+o(1)\rightarrow 0$ as $n\rightarrow\infty$. Hence one of the cases in proposition \ref{compactnesslemma} holds. If case(i) occurs, that is $u_{n}\rightharpoonup u$ in $W^{1,p}({\mathbb{R}^N})$, where $u$ solves (\ref{equation}) for some $\lambda <0$, and
        \begin{equation}\label{Eless}
        	E_{\mu}(u)\leqslant m^{+}(a,\mu)+\frac{1}{N}S^{\frac{N}{p}}<-\frac{1}{N}S^{\frac{N}{p}}.
        \end{equation}
    Since $u$ solves (\ref{equation}), by the Pohozaev identity $P_{\mu}(u)=0$. Therefore, the Gagliardo-Nirenberg inequality implies
        \begin{align*}
        	E_{\mu}(u)&=\frac{1}{N}\lVert\nabla u\rVert_{p}^p-\mu\gamma_{q}\Big(\frac{1}{q\gamma_{q}}-\frac{1}{p^*}\Big)\lVert u\rVert_{q}^q\\
        	&=\frac{1}{N}\lVert\nabla u\rVert_{p}^p-\mu\gamma_{q}\Big(\frac{1}{q\gamma_{q}}-\frac{1}{p^*}\Big)C_{N,q}^qa^{q(1-\gamma_{q})}\lVert\nabla u\rVert_{p}^{q\gamma_{q}},
        \end{align*}
    where we used the fact that $\lVert u\rVert_{p}\leqslant a$ by the Fatou lemma. we introduce the function
        \[\varphi(t)=\frac{1}{N}t^p-\mu\gamma_{q}\Big(\frac{1}{q\gamma_{q}}-\frac{1}{p^*}\Big)C_{N,q}^qa^{q(1-\gamma_{q})}t^{q\gamma_{q}},\quad t>0.\]
    Then, we can check that $\varphi$ has a unique critical point on $(0,+\infty)$, which is a global minimum point at negative level. The critical point is
        \[\bar{t}=\bigg(\frac{\mu\gamma_{q}(p^*-q\gamma_{q})NC_{N.q}^q}{pp^*}\bigg)^{\frac{1}{p-q\gamma_{q}}}a^{\frac{q(1-\gamma_{q})}{p-q\gamma_{q}}},\]
    and the minimum is
        \[\varphi(\bar{t})=-\frac{p-q\gamma_{q}}{q}(N\gamma_{q})^{\frac{q\gamma_{q}}{p-q\gamma_{q}}}\bigg(\mu a^{q(1-\gamma_{q})}\frac{(p^*-q\gamma_{q})C_{N,q}^q}{pp^*}\bigg)^{\frac{p}{p-q\gamma_{q}}}<0.\]
    By (\ref{alpha}), since $\mu a^{q(1-\gamma_{q})}<\alpha(N,p,q)\leqslant C''$, we have $\varphi(\bar{t})\geqslant-S^{N/p}/N$ which contradicts with (\ref{Eless}). This means that $u_{n}\rightarrow u$ in $W^{1,p}(\mathbb{R}^N)$ and $u\in S_{a}$ is a solution to (\ref{equation}) for some $\lambda<0$. Since $u_{n}$ is radially symmetric and radially decreasing, we know $u$ is radially symmetric and radially non-increasing. Now, by the strong maximum principle \cite{}, $u$ is positive.

    In fact, $u$ is a ground state, since $E_{\mu}(u)=\inf_{\mathcal{P}_{a,\mu}}E_{\mu}$, and any other normalized solution stays on $\mathcal{P}_{a,\mu}$. It remains to show that any other ground state is a local minimizer for $E_{\mu}$ on $A_{R_{0}}$. Let $v$ be a critical point of $E_{\mu}|_{S_{a}}$ and $E_{\mu}(v)=m^{+}(a,\mu)$, then $v\in\mathcal{P}_{a,\mu}$. By Lemma \ref{structure}, we know $s_{v}=0$ and $\lVert\nabla v\rVert_{p}\leqslant R_{0}$. Therefore, $v$ is a local minimizer for $E_{\mu}$ on $A_{R_{0}}$.$\hfill\qed$

\subsection{Existence of mountain-pass type solution}
    \
    \newline
    \indent In this section, we prove Theorem \ref{th2}. By Lemma \ref{structure} and Lemma \ref{mplus}, we can construct a minimax structure. Let
        \[E^{c}:=\Big\{u\in S_{a}:E_{\mu}(u)\leqslant c\Big\}.\]
    We introduce the minimax class
        \[\Gamma:=\Big\{\gamma=(\alpha,\theta)\in C([0,1],\mathbb{R}\times S_{a,r}): \gamma(0)\in(0,\mathcal{P}_{a,\mu}^{+}), \gamma(1)\in(0,E^{2m^{+}(a,\mu)})\Big\},\]
    with associated minimax level
        \[\sigma(a,\mu):=\inf_{\gamma\in\Gamma}\max_{(\alpha,\theta)\in\gamma([0,1])}\tilde{E}_{\mu}(\alpha,\theta),\]
    where
        \[\tilde{E}_{\mu}(s,u):=E_{\mu}(s*u).\]

    In order to obtain the compactness of PS sequence by using Proposition \ref{compactnesslemma}, the following energy estimates are required.
    \begin{lemma}\label{m-<m+1}
    	Let $N\geqslant p^2$. Then, we have
    	    \[m^{-}(a,\mu)<m^{+}(a,\mu)+\frac{1}{N}S^{\frac{N}{p}}.\]
    \end{lemma}
    \begin{proof}
    	By Appendix \ref{A}, we have
    	    \begin{equation*}
    		    \lVert\nabla u_{\varepsilon}\rVert_{p}^p=S^{\frac{N}{p}}+O(\varepsilon^{\frac{N-p}{p-1}}),\quad\lVert u_{\varepsilon}\rVert_{p^*}^{p^*}=S^{\frac{N}{p}}+O(\varepsilon^{\frac{N}{p-1}}),
    	    \end{equation*}
    	and
    	    \begin{equation*}
    		    \lVert u_{\varepsilon}\rVert_{r}^{r}=\left\{\begin{array}{ll}
    			    C\varepsilon ^{N-\frac{(N-p)r}{p}}+O({\varepsilon^{\frac{(N-p)r}{p(p-1)}}}) &N>p^2 \ or\ p<r<p^*\\
    			    C\varepsilon^{p}\lvert \log\varepsilon\rvert+O(\varepsilon ^{p}) &N=p^2\ and\ r=p,
    		    \end{array}\right.
    	    \end{equation*}
    	where $p\leqslant r<p^*$.
    	
    	Let $V_{\varepsilon,\tau}=u_{a,\mu}^+ +\tau u_{\varepsilon}(\cdot-x_{\varepsilon})$ and
    	    \begin{equation}\label{wet}
    	    	W_{\varepsilon,\tau}(x)=\big(a^{-1}\lVert V_{\varepsilon,\tau}\rVert_{p}\big)^{\frac{N-p}{p}}V_{\varepsilon,\tau}\big(a^{-1}\lVert V_{\varepsilon,\tau}\rVert_{p}x\big),
    	    \end{equation}
    	where $\lvert x_{\varepsilon}\rvert=\varepsilon^{-1}$. Then, we have
    	    \[\lVert W_{\varepsilon,\tau}\rVert_{p}^p=a^p,\quad\lVert\nabla W_{\varepsilon,\tau}\rVert_{p}^p=\lVert\nabla V_{\varepsilon,\tau}\rVert_{p}^p,\quad\lVert W_{\varepsilon,\tau}\rVert_{p^*}^{p^*}=\lVert V_{\varepsilon,\tau}\rVert_{p^*}^{p^*},\]
    	and
    	    \[\lVert W_{\varepsilon,\tau}\rVert_{q}^q=(a\lVert V_{\varepsilon,\tau}\rVert_{p}^{-1})^{q(1-\gamma_{q})}\lVert V_{\varepsilon,\tau}\rVert_{q}^{q}.\]
    	Thus, there exists a unique $t_{\varepsilon,\tau}\in\mathbb{R}$ such that $t_{\varepsilon,\tau}\star W_{\varepsilon,\tau}\in\mathcal{P}_{a,\mu}^{-}$, that is
    	    \begin{equation}\label{tet}
    		    e^{pt_{\varepsilon,\tau}}\lVert\nabla W_{\varepsilon,\tau}\rVert_{p}^p=\mu\gamma_{q,s} e^{q\gamma_{q}t_{\varepsilon,\tau}}\lVert W_{\varepsilon,\tau}\rVert_{q}^q+e^{p^*t_{\varepsilon,\tau}}\lVert W_{\varepsilon,\tau}\rVert_{p^*}^{p^*}.
    	    \end{equation}
    	Since $u_{a,\mu}^+\in\mathcal{P}_{a,\mu}^{+}$,  we know $t_{\varepsilon,0}>0$. By (\ref{tet}),
    	    \[e^{(p^*-p)t_{\varepsilon,\tau}}<\frac{\lVert\nabla W_{\varepsilon,\tau}\rVert_{p}^p}{\lVert W_{\varepsilon,\tau}\rVert_{p^*}^{p^*}}=\frac{\lVert\nabla V_{\varepsilon,\tau}\rVert_{p}^p}{\lVert V_{\varepsilon,\tau}\rVert_{p^*}^{p^*}},\]
    	which implies $t_{\varepsilon,\tau}\rightarrow -\infty$ as $\tau\rightarrow +\infty$. By Lemma \ref{structure} , $t_{\varepsilon,\tau}$ is continuous for $\tau$, hence we can choose a suitable $\tau =\tau_{\varepsilon}>0$ such that $t_{\varepsilon,\tau_{\varepsilon}}=0$. It follows that
    	    \begin{align}\label{wete}
    		    m^-(a,\mu)&\leqslant E_{\mu}(W_{\varepsilon,\tau_{\varepsilon}})=\frac{1}{p}\lVert\nabla W_{\varepsilon,\tau_{\varepsilon}}\rVert_{p}^p-\frac{\mu}{q}\lVert W_{\varepsilon,\tau_{\varepsilon}}\rVert_{q}^q-\frac{1}{p^*}\lVert W_{\varepsilon,\tau_{\varepsilon}}\rVert_{p^*}^{p^*}\nonumber\\
    		    &=\frac{1}{p}\lVert\nabla V_{\varepsilon,\tau_{\varepsilon}}\rVert_{p}^p-\frac{\mu}{q}(a\lVert V_{\varepsilon,\tau_{\varepsilon}}\rVert_{p}^{-1})^{q(1-\gamma_{q})}\lVert V_{\varepsilon,\tau_{\varepsilon}}\rVert_{q}^q-\frac{1}{p^*}\lVert V_{\varepsilon,\tau_{\varepsilon}}\rVert_{p^*}^{p^*}.
    	    \end{align}
    	If $\liminf_{\varepsilon\rightarrow 0}\tau_{\varepsilon}=0$ or $\limsup_{\varepsilon\rightarrow 0}\tau_{\varepsilon}=+\infty$, then
    	    \begin{equation*}
    		    m^-(a,\mu)\leqslant\liminf_{\varepsilon\rightarrow 0} E_{\mu}(W_{\varepsilon,\tau_{\varepsilon}})\leqslant E_{\mu}(u_{a,\mu}^+)=m^+(a,\mu),
    	    \end{equation*}
    	a contradiction with Lemma \ref{mplus}. Therefore, there exists $t_{2}>t_{1}>0$ independent of $\varepsilon$ such that $\tau_{\varepsilon}\in[t_{1},t_{2}]$.
    	
    	Now, we estimate $\lVert\nabla V_{\varepsilon,\tau_{\varepsilon}}\rVert_{p}^p$. Using the inequality
    	    \[(a+b)^p\leqslant a^p+b^p+C(a^{p-1}b+ab^{p-1})\quad\forall a,b\geqslant 0,\]
    	we have
    	    \begin{align}\label{nablaVe1}
    	    	\lVert\nabla V_{\varepsilon,\tau_{\varepsilon}}\rVert_{p}^p&\leqslant\lVert\nabla u_{a,\mu}^{+}\rVert_{p}^p+\tau_{\varepsilon}^p\lVert\nabla u_{\varepsilon}\rVert_{p}^p\nonumber\\
    	    	&\qquad\qquad+C\int_{\mathbb{R}^N}\lvert\nabla u_{a,\mu}^{+}\rvert^{p-1}\lvert\nabla u_{\varepsilon}\rvert dx+C\int_{\mathbb{R}^N}\lvert\nabla u_{a,\mu}^{+}\rvert\lvert\nabla u_{\varepsilon}\rvert^{p-1}dx\nonumber\\
    	    	&=\lVert\nabla u_{a,\mu}^{+}\rVert_{p}^p+S^{\frac{N}{p}}\tau_{\varepsilon}^p\nonumber\\
    	    	&\qquad\qquad+C\int_{\mathbb{R}^N}\lvert\nabla u_{a,\mu}^{+}\rvert^{p-1}\lvert\nabla u_{\varepsilon}\rvert dx+ C\int_{\mathbb{R}^N}\lvert\nabla u_{a,\mu}^{+}\rvert\lvert\nabla u_{\varepsilon}\rvert^{p-1}dx+O(\varepsilon^{\frac{N-p}{p-1}}).
    	    \end{align}
    	By the H\"older inequality,
    	    \begin{align*}
    	    	\int_{\mathbb{R}^N}\lvert\nabla u_{a,\mu}^{+}\rvert^{p-1}\lvert\nabla u_{\varepsilon}\rvert dx&=\int_{B_{2}(x_{\varepsilon})}\lvert\nabla u_{a,\mu}^{+}\rvert^{p-1}\lvert\nabla u_{\varepsilon}\rvert dx\nonumber\\
    	    	&\leqslant\lVert\nabla u_{a,\mu}^{+}\rVert_{L^p(B_{2}(x_{\varepsilon}))}^{p-1}\lVert\nabla u_{\varepsilon}\rVert_{p}\nonumber\\
    	    	&\leqslant C\lVert\nabla u_{a,\mu}^{+}\rVert_{L^p(B_{2}(x_{\varepsilon}))}^{p-1}.
    	    \end{align*}
        We know there exists $\lambda_{a,\mu}^{+}<0$ such that
            \[-\Delta_{p}u_{a,\mu}^+=\lambda_{a,\mu}^{+}\lvert u_{a,\mu}^+\rvert^{p-1}+\mu\lvert u_{a,\mu}^+\rvert^{q-1}+\lvert u_{a,\mu}^+\rvert^{p^*-1}.\]
        Then, by \cite[theorem 8]{gfsj}, there exists $a,b>0$ such that
            \[\lvert\nabla u_{a,\mu}^{+}(x)\rvert\leqslant ae^{-b\lvert x\rvert}\]
        for $\lvert x\rvert$ sufficiently large. This means
            \[\lVert\nabla u_{a,\mu}^{+}\rVert_{L^{p}(B_{2}(x_{\varepsilon}))}\leqslant Ce^{-b\varepsilon^{-1}},\]
        and hence
            \[\int_{\mathbb{R}^N}\lvert\nabla u_{a,\mu}^{+}\rvert^{p-1}\lvert\nabla u_{\varepsilon}\rvert dx\leqslant Ce^{e^{-b(p-1)\varepsilon^{-1}}}.\]
        Similarly, we have
            \[\int_{\mathbb{R}^N}\lvert\nabla u_{a,\mu}^{+}\rvert\lvert\nabla u_{\varepsilon}\rvert^{p-1}dx\leqslant Ce^{e^{-b\varepsilon^{-1}}}.\]
        Therefore, by (\ref{nablaVe1}), we obtain
            \begin{equation}\label{nablaVe}
            	\lVert\nabla V_{\varepsilon,\tau_{\varepsilon}}\rVert_{p}^{p}\leqslant\lVert\nabla u_{a,\mu}^{+}\rVert_{p}^{p}+S^{\frac{N}{p}}\tau_{\varepsilon}^{p}+O\big(\varepsilon^{\frac{N-p}{p-1}}\big).
            \end{equation}

        Next, we estimate $\lVert V_{\varepsilon,\tau_{\varepsilon}}\rVert_{p}^{p}$. We have
            \begin{align*}
            	\int_{\mathbb{R}^N}\lvert u_{a,\mu}^{+}+\tau_{\varepsilon}u_{\varepsilon}\rvert^{p}&=\int_{B_{2}(x_{\varepsilon})}\lvert u_{a,\mu}^{+}+\tau_{\varepsilon}u_{\varepsilon}\rvert^{p}dx+\int_{B_{2}^{c}(x_{\varepsilon})}\lvert u_{a,\mu}^{+}\rvert^{p}dx\\
            	&\leqslant a^{p}+\int_{B_{2}(x_{\varepsilon})}\lvert u_{a,\mu}^{+}+\tau_{\varepsilon}u_{\varepsilon}\rvert^{p}dx\\
            	&\leqslant a^{p}+C\int_{B_{2}(x_{\varepsilon})}\big(\lvert u_{a,\mu}^{+}\rvert^{p}+\lvert u_{\varepsilon}\rvert^{p}\big)dx\\
            	&\leqslant a^{p}+C(e^{-pb\varepsilon^{-1}}+\varepsilon^{p}\lvert\log\varepsilon\rvert)\\
            	&=a^{p}+O(\varepsilon^{p}\lvert\log\varepsilon\rvert).
            \end{align*}
        Thus,
            \begin{equation}\label{Vep}
            	\big(a\lVert V_{\varepsilon,\tau_{\varepsilon}}\rVert_{p}^{-1}\big)^{q(1-\gamma_{q})}\geqslant 1+O(\varepsilon^{p}\lvert\log\varepsilon\rvert).
            \end{equation}

        It is easy to know that
            \begin{equation}\label{Veq}
            	\lVert V_{\varepsilon,\tau_{\varepsilon}}\rVert_{q}^{q}\geqslant\lvert u_{a,\mu}^{+}\rVert_{q}^{q}+\tau_{\varepsilon}^{q}\lvert u_{\varepsilon}\rvert_{q}^{q}\geqslant\lvert u_{a,\mu}^{+}\rVert_{q}^{q}+C\varepsilon^{N-\frac{(N-p)q}{p}},
            \end{equation}
        and
            \begin{equation}\label{Vepstar}
            	\lVert V_{\varepsilon,\tau_{\varepsilon}}\rVert_{p^*}^{p^*}\geqslant\lVert u_{a,\mu}^{+}\rVert_{p^*}^{p^*}+\tau_{\varepsilon}^{p^*}\lVert u_{\varepsilon}\rVert_{p^*}^{p^*}=\lVert u_{a,\mu}^{+}\rVert_{p^*}^{p^*}+S^{\frac{N}{p}}\tau_{\varepsilon}^{p^*}+O(\varepsilon^{\frac{N}{p-1}}).
            \end{equation}
        Combining (\ref{wete}), (\ref{nablaVe}), (\ref{Vep}) (\ref{Veq}) and (\ref{Vepstar}), we obtain
            \begin{align*}
            	m^{-}(a,\mu)&\leqslant m^{+}(a,\mu)+S^{\frac{N}{p}}\Big(\frac{1}{p}\tau_{\varepsilon}^{p}-\frac{1}{p^*}\tau_{\varepsilon}^{p^*}\Big)-C\varepsilon^{N-\frac{(N-p)q}{p}}+O(\varepsilon^{p}\lvert\log\varepsilon\rvert)\\
            	&<m^{+}(a,\mu)+\frac{1}{N}S^{\frac{N}{p}},
            \end{align*}
        by taking $\varepsilon$ sufficiently small.
    \end{proof}

    \begin{lemma}\label{m-<m+2}
    	Let $N<p^2<9$. Then, we have
    	    \[m^{-}(a,\mu)<m^{+}(a,\mu)+\frac{1}{N}S^{\frac{N}{p}}.\]
    \end{lemma}
    \begin{proof}
    	We set $V_{\varepsilon,\tau_{\varepsilon}}=u_{a,\mu}^{+}+\tau u_{\varepsilon}$ and the definition of $W_{\varepsilon,\tau_{\varepsilon}}$ is same to (\ref{wet}). Then, we can choose $\tau=\tau_{\varepsilon}>0$ such that $W_{\varepsilon,\tau_{\varepsilon}}\in\mathcal{P}_{a,\mu}^{-}$. Moreover, there exists $t_{2}>t_{1}>0$ independent of $\varepsilon$ such that $\tau_{\varepsilon}\in[t_{1},t_{2}]$. Therefore,
    	    \begin{equation}
    	    	m^{-}(a,\mu)\leqslant\frac{1}{p}\lVert\nabla V_{\varepsilon,\tau_{\varepsilon}}\rVert_{p}^p-\frac{\mu}{q}(a\lVert V_{\varepsilon,\tau_{\varepsilon}}\rVert_{p}^{-1})^{q(1-\gamma_{q})}\lVert V_{\varepsilon,\tau_{\varepsilon}}\rVert_{q}^q-\frac{1}{p^*}\lVert V_{\varepsilon,\tau_{\varepsilon}}\rVert_{p^*}^{p^*}.
    	    \end{equation}

        Now, we estimate $\lVert\nabla V_{\varepsilon,\tau_{{\varepsilon}}}\rVert_{p}^{p}$ and $\lVert V_{\varepsilon,\tau_{\varepsilon}}\rVert_{p}^{p}$. We can prove that for $a,b\geqslant 0$, there is
            \[(a^2+b^2+2ab\cos\alpha)^{\frac{p}{2}}\leqslant a^p+b^p+pa^{p-1}b\cos\alpha+Ca^{\frac{p-1}{2}}b^{\frac{p+1}{2}},\]
        uniformly in $\alpha$. Thus,
            \begin{align}\label{nablaVete1}
            	\lVert\nabla V_{\varepsilon,\tau_{\varepsilon}}\rVert_{p}^{p}&=\int_{\mathbb{R}^N}\big(\lvert\nabla u_{a,\beta}^{+}\rvert^2+\tau_{\varepsilon}^2\lvert\nabla u_{\varepsilon}\rvert^2+2\tau_{\varepsilon}\nabla u_{a,\mu}^{+}\cdot\nabla u_{\varepsilon}\big)^{\frac{p}{2}}dx\nonumber\\
            	&\leqslant\lVert\nabla u_{a,\mu}^{+}\rVert_{p}^{p}+\tau_{\varepsilon}^{p}\lVert u_{\varepsilon}\rVert_{p}^{p}\nonumber\\
            	&\qquad\qquad+p\tau_{{\varepsilon}}\int_{\mathbb{R}^N}\lvert\nabla u_{a,\mu}^{+}\rvert^{p-2}\nabla u_{a,\mu}^{+}\cdot\nabla u_{\varepsilon}dx+C\int_{\mathbb{R}^N}\lvert\nabla u_{a,\mu}^{+}\rvert^{\frac{p-1}{2}}\lvert\nabla u_{\varepsilon}\rvert^{\frac{p+1}{2}}dx.
            \end{align}
        By \cite{tp}, $\nabla u_{a,\mu}^{+}$ is local H\"older continuous, hence
            \begin{equation}\label{interact}
            	\int_{\mathbb{R}^N}\lvert\nabla u_{a,\mu}^{+}\rvert^{\frac{p-1}{2}}\lvert\nabla u_{\varepsilon}\rvert^{\frac{p+1}{2}}dx\leqslant C\int_{\mathbb{R}^N}\lvert\nabla u_{\varepsilon}\rvert^{\frac{p+1}{2}}dx=O(\varepsilon^{\frac{(N-p)(p+1)}{2p(p-1)}}).
            \end{equation}
        By (\ref{equation}), we know
            \begin{align}\label{lambda+}
            	&\int_{\mathbb{R}^N}\lvert\nabla u_{a,\mu}\rvert^{p-2}\nabla u_{a,\mu}^{+}\cdot\nabla u_{\varepsilon}dx\nonumber\\
            	=&\lambda_{a,\mu}^{+}\int_{\mathbb{R}^N}\lvert u_{a,\mu}^{+}\rvert^{p-1}u_{\varepsilon}dx+\mu\int_{\mathbb{R}^N}\lvert u_{a,\mu}^{+}\rvert^{q-1}u_{\varepsilon}dx+\int_{\mathbb{R}^N}\lvert u_{a,\mu}^{+}\rvert^{p^*-1}u_{\varepsilon}dx.
            \end{align}
        From (\ref{nablaVete1}), (\ref{interact}) and (\ref{lambda+}), we obtain
            \begin{equation}\label{nablaVete}
            	\begin{split}
            		\lVert\nabla V_{\varepsilon,\tau_{{\varepsilon}}}\rVert_{p}^{p}&
            		\leqslant\lVert\nabla u_{a,\mu}^{+}\rVert_{p}^{p}+S^{\frac{N}{p}}\tau_{\varepsilon}^{p}+p\tau_{\varepsilon}\lambda_{a,\mu}^{+}\int_{\mathbb{R}^N}\lvert u_{a,\mu}^{+}\rvert^{p-1}u_{\varepsilon}dx\\
            		&\qquad\qquad+p\tau_{\varepsilon}\mu\int_{\mathbb{R}^N}\lvert u_{a,\mu}^{+}\rvert^{q-1}u_{\varepsilon}dx+p\tau_{\varepsilon}\lvert u_{a,\mu}^{+}\rvert^{p^*-1}u_{\varepsilon}dx+O(\varepsilon^{\frac{(N-p)(p+1)}{2p(p-1)}}).
            	\end{split}
            \end{equation}
        In the same way, we have
            \begin{align*}
            	\lVert V_{\varepsilon}\rVert_{p}^{p}&\leqslant\lVert u_{a,\mu}^{+}\rVert_{p}^{p}+\tau_{\varepsilon}^{p}\lVert u_{\varepsilon}\rVert_{p}^{p}+p\tau_{\varepsilon}\int_{\mathbb{R}^N}\lvert u_{a,\mu}^{+}\rvert u_{\varepsilon}dx+C\int_{\mathbb{R}^N}\lvert u_{a,\mu}^{+}\rvert^{\frac{p-1}{2}}\lvert u_{\varepsilon}\rvert^{\frac{p+1}{2}}dx\nonumber\\
            	&\leqslant a^p+p\tau_{\varepsilon}\int_{\mathbb{R}^N}\lvert u_{a,\mu}^{+}\rvert^{p-1}u_{\varepsilon}dx+O(\varepsilon^{\frac{(N-p)(p+1)}{2p(p-1)}}),
            \end{align*}
        which implies
            \begin{equation}\label{Vepp}
            	\big(a\lVert V_{\varepsilon,\tau_{\varepsilon}}\rVert_{p}^{-1}\big)^{q(1-\gamma_{q})}\geqslant 1-\frac{q(1-\gamma_{q})\tau_{\varepsilon}}{a^{p}}\int_{\mathbb{R}^N}\lvert u_{a,\mu}^{+}\rvert^{p-1}u_{\varepsilon}dx+O(\varepsilon^{\frac{(N-p)(p+1)}{2p(p-1)}}).
            \end{equation}

        Next, we estimate $\lVert V_{\varepsilon,\tau_{\varepsilon}}\rVert_{q}^{q}$ and $\lVert V_{\varepsilon,\tau_{\varepsilon}}\rVert_{p^*}^{p^*}$. For every $a,b\geqslant 0$, we know
            \[(a+b)^{r}\geqslant\left\{\begin{array}{ll}
            	a^r+b^r+r(a^{r-1}b+ab^{r-1})&\forall r\geqslant 3\\
            	a^r+ra^{r-1}b&\forall r\geqslant 1.
            \end{array}
            \right.\]
        Thus, we have
            \begin{equation}\label{Veqq}
            	\lVert V_{\varepsilon,\tau_{\varepsilon}}\rVert_{q}^{q}\geqslant\lVert u_{a,\mu}^{+}\rVert_{q}^{q}+q\tau_{\varepsilon}\int_{\mathbb{R}^N}\lvert u_{a,\mu}^{+}\rvert^{q-1}u_{\varepsilon},
            \end{equation}
        and
            \begin{align}\label{Vepstarp}
            	\lVert V_{\varepsilon,\tau_{\varepsilon}}\rVert_{p^*}^{p^*}&\geqslant\lVert u_{a,\mu}^{+}\rVert_{p^*}^{p^*}+\tau_{\varepsilon}^{p^*}\lVert u_{\varepsilon}\rVert_{p^*}^{p^*}\nonumber\\
            	&\qquad\qquad+p^*\tau_{\varepsilon}\int_{\mathbb{R}^N}\lvert u_{a,\mu}^{+}\rvert^{p^*-1}u_{\varepsilon}dx+p^*\tau_{\varepsilon}^{p^*-1}\int_{\mathbb{R}^N}u_{a,\mu}^{+}\lvert u_{\varepsilon}\rvert^{p^*-1}dx\nonumber\\
            	&=\lVert u_{a,\mu}^{+}\rVert_{p^*}^{p^*}+S^{\frac{N}{p}}\tau_{\varepsilon}^{p^*}\nonumber\\
            	&\qquad\qquad+p^*\tau_{\varepsilon}\int_{\mathbb{R}^N}\lvert u_{a,\mu}^{+}\rvert^{p^*-1}u_{\varepsilon}dx+p^*\tau_{\varepsilon}^{p^*-1}\int_{\mathbb{R}^N}u_{a,\mu}^{+}\lvert u_{\varepsilon}\rvert^{p^*-1}dx+O(\varepsilon^{\frac{N}{p-1}}).
            \end{align}

        Combining (\ref{m-<m+2}), (\ref{nablaVete}), (\ref{Vepp}), (\ref{Veqq}), (\ref{Vepstarp}) and using $\lambda_{a,\mu}^{+}a^p=\mu(\gamma_{q}-1)\lVert u_{a,\mu}^{+}\rVert_{q}^q$, we obtain
            \begin{align*}
            	m^{-}(a,\mu)&\leqslant m^{+}(a,\mu)+S^{\frac{N}{p}}\Big(\frac{1}{p}\tau_{\varepsilon}^p-\frac{1}{p^*}\tau_{\varepsilon}^{p^*}\Big)-\tau_{\varepsilon}^{p^*-1}\int_{\mathbb{R}^N}u_{a,\mu}^{+}u_{\varepsilon}^{p^*-1}dx+O(\varepsilon^{\frac{(N-p)(p+1)}{2p(p-1)}})\nonumber\\
            	&\leqslant m^{+}(a,\mu)+\frac{1}{N}S^{\frac{N}{p}}-C\varepsilon^{\frac{N-p}{p}}+O(\varepsilon^{\frac{(N-p)(p+1)}{2p(p-1)}})\nonumber\\
            	&<m^{+}(a,\mu)+\frac{1}{N}S^{\frac{N}{p}},
            \end{align*}
        by taking $\varepsilon$ sufficiently small.
    \end{proof}
    \begin{remark}
    	{\rm Although Lemma \ref{m-<m+1} and Lemma \ref{m-<m+2} both obtain the same estimates, the method of proof were slightly different, so we state these two results separately.}
    \end{remark}

    For every $0<a<\big(\mu^{-1}\alpha\big)^{1/(q(\gamma_{q}-1))}$, let $u\in\mathcal{P}_{a,\mu}^{\pm}$, then $u_{b}=\frac{b}{a}u\in S_{b}$ for every $b>0$. By lemma \ref{structure}, there exists unique $t_{\pm}(b)\in\mathbb{R}$ such that $t_{\pm}(b)\star u_{b}\in\mathcal{P}_{b,\mu}^{\pm}$ for every $0<b<\big(\mu^{-1}\alpha\big)^{1/(q(\gamma_{q}-1))}$. Clearly, $t_{\pm}(a)=0$.
    \begin{lemma}
    	For every $0<a<\big(\mu^{-1}\alpha\big)^{1/(q(\gamma_{q}-1))}$, $t'_{\pm}(a)$ exist and
    	    \begin{equation}\label{tderivative}
    	    	t'_{\pm}(a)=\frac{\mu q\gamma_{q}\lVert u\rVert_{q}^q+p^*\lVert u\rVert_{p^*}^{p^*}-p\lVert\nabla u\rVert_{p}^p}{a\big(p\lVert\nabla u\rVert_{p}^p-\mu q\gamma_{q}^2\lVert u\rVert_{q}^q-p^*\lVert u\rVert_{p^*}^{p^*}\big)}.
    	    \end{equation}
    \end{lemma}
    \begin{proof}
    	Since $t_{\pm}(b)\star u_{b}\in\mathcal{P}_{b,\mu}^{\pm}$, we have
    	    \[\Big(\frac{b}{a}\Big)^pe^{pt_{\pm}(b)}\lVert\nabla u\rVert_{p}^p=\mu\gamma_{q}\Big(\frac{b}{a}\Big)^{q}e^{q\gamma_{q}t_{\pm}(b)}\lVert u\rVert_{q}^q+\Big(\frac{b}{a}\Big)^{p^*}e^{p^*t_{\pm}(b)}\lVert u\rVert_{p^*}^{p^*}.\]
    	Considering the function
    	    \[\Phi(b,t)=\Big(\frac{b}{a}\Big)^pe^{pt}\lVert\nabla u\rVert_{p}^p-\mu\gamma_{q}\Big(\frac{b}{a}\Big)^{q}e^{q\gamma_{q}t}\lVert u\rVert_{q}^q-\Big(\frac{b}{a}\Big)^{p^*}e^{p^*t}\lVert u\rVert_{p^*}^{p^*},\]
    	then $\Phi(a,0)=0$ and $\Phi(b,t)$ has a continuous derivative in some neighborhood of $(a,0)$. Moreover, since $u\in\mathcal{P}_{a,\mu}^{\pm}$, we have
    	    \[\partial_{t}\Phi(a,0)=p\lVert\nabla u\rVert_{p}^p-\mu q\gamma_{q}^2\lVert u\rVert_{q}^q-p^*\lVert u\rVert_{p^*}^{p^*}\neq 0.\]
    	Now, by the implict function theorem, we know $t'_{\pm}(a)$ exist and (\ref{tderivative}) holds.
    \end{proof}

    \begin{lemma}\label{decreasing}
    	$m^{\pm}(a,\mu)$ is non-increasing for $0<a<\big(\mu^{-1}\alpha\big)^{1/(q(\gamma_{q}-1))}$.
    \end{lemma}
    \begin{proof}
    	Since
    	   \[E_{\mu}\big(t_{\pm}(b)\star u_{b}\big)=\frac{1}{p}\Big(\frac{b}{a}\Big)^pe^{pt_{\pm}(b)}\lVert\nabla u\rVert_{p}^p-\frac{\mu}{q}\Big(\frac{b}{a}\Big)^{q}e^{q\gamma_{q}t_{\pm}(b)}\lVert u\rVert_{q}^q-\frac{1}{p^*}\Big(\frac{b}{a}\Big)^{p^*}e^{p^*t_{\pm}(b)}\lVert u\rVert_{p^*}^{p^*},\]
    	and $u\in\mathcal{P}_{a,\mu}^{\pm}$, we have
    	    \begin{align*}
    	    	\frac{dE_{\mu}\big(t_{\pm}(b)\star u_{b}\big)}{db}|_{b=a}&=\frac{1}{a}\big(\lVert\nabla u\rVert_{p}^p-\mu\lVert u\rVert_{q}^q-\lVert u\rVert_{p^*}^{p^*}\big)+\big(\lVert\nabla u\rVert_{p}^p-\mu\gamma_{q}\lVert u\rVert_{q}^q-\lVert u\rVert_{p^*}^{p^*}\big)t'_{\pm}(a)\\
    	    	&=\frac{\mu(\gamma_{q}-1)\lVert u\rVert_{q}^q}{a}<0,
    	    \end{align*}
        which implies $E_{\mu}\big(t_{\pm}(b)\star u_{b}\big)<E_{\mu}(u)$ for $a<b<\big(\mu^{-1}\alpha\big)^{1/(q(\gamma_{q}-1))}$. Therefore, $m^{\pm}(a,\mu)\geqslant m^{\pm}(b,\mu)$ for $a<b<\big(\mu^{-1}\alpha\big)^{1/(q(\gamma_{q}-1))}$.
    \end{proof}

    \begin{lemma}\label{msigma}
    	We have $m^{-}(a,\mu)=m_{r}^{-}(a,\mu)=\sigma(a,\mu)$.
    \end{lemma}
    \begin{proof}
    	By the definition of $m^{-}(a,\mu)$ and $m_{r}^{-}(a,\mu)$, we have $m^{-}(a,\mu)\leqslant m_{r}^{-}(a,\mu)$. For every $u\in\mathcal{P}_{a,\mu}^{-}$, let $v=\lvert u\rvert^*$, the Schwarz rearrangement of $\lvert u\rvert$, then
    	\[E_{\mu}(s\star v)\leqslant E_{\mu}(s\star u)\quad\forall s\in\mathbb{R}.\]
    	Therefore, by lemma \ref{structure},
    	\[m_{r}^{-}(a,\mu)\leqslant E_{\mu}(t_{v}\star v)\leqslant E_{\mu}(t_{v}\star u)\leqslant E_{\mu}(u),\]
    	which implies $m_{r}^{-}(a,\mu)\leqslant m^{-}(a,\mu)$.
    	
    	Next, we prove $m_{r}^{-}(a,\mu)=\sigma(a,\mu)$.
    	For every $u\in\mathcal{P}_{a,\mu}^{-}\cap S_{a,r}$, choosing $s_{0}$ such that $E_{\mu}(s_{0}\star u)\leqslant 2m^{+}(a,\mu)$ and defining
    	\[\gamma_{u}: \tau\in [0,1]\longmapsto\big(0,((1-\tau)s_{u}+\tau s_{0}\star u)\big)\in\mathbb{R}\times S_{a,r}.\]
    	By lemma \ref{structure}, $\gamma_{u}\in\Gamma$. Thus
    	\begin{align*}
    		\sigma(a,\mu)&\leqslant\max_{\tau\in [0,1]}\tilde{E}_{\mu}\big(0,((1-\tau)s_{u}+\tau s_{0})\star u\big)\\
    		&=\max_{\tau\in [0,1]}E_{\mu}\big(((1-\tau)s_{u}+\tau s_{0})\star u\big)\\
    		&\leqslant E_{\mu}(t_{u}\star u)=E_{\mu}(u),
    	\end{align*}
    	which implies $\sigma(a,\mu)\leqslant m_{r}^{-}(a,\mu)$.
    	
    	For every $\gamma\in\Gamma$, we have $\gamma(0)\in(0,\mathcal{P}_{a,\mu}^{+})$ and $\gamma(1)\in E^{2m^{+}(a,\mu)}$. Then, by Lemma \ref{structure}, we know
    	$t_{\theta(0)}>0>t_{\theta(1)}$ and since $t_{\alpha(\tau)\star\theta(\tau)}$ is continuous for $\tau$, there exists $\tau_{\gamma}\in [0,1]$ such that $t_{\alpha(\tau_{\gamma})\star\theta(\tau_{\gamma})}=0$. This implies
    	\begin{equation*}
    		\max_{(\alpha,\theta)\in\gamma([0,1])}\tilde{E}_{\mu}(\alpha,\theta)\geqslant\tilde{E}_{\mu}\big(\alpha(\tau_{\gamma}),\theta(\tau_{\gamma})\big)=E_{\mu}\big(\alpha(\tau_{\gamma})\star\theta(\tau_{\gamma})\big)\geqslant m_{r}^{-}(a,\mu).
    	\end{equation*}
    	Therefore, $\sigma(a,\mu)\geqslant m_{r}(a,\mu)$.
    \end{proof}

    \noindent\textbf{Proof of Theorem \ref{th2}:} Let
    \[X=\mathbb{R}\times S_{a,r},\quad\mathcal{F}=\Gamma,\quad and\quad B=(0,A_{k})\cup(0,E^{2m^{+}(a,\mu)}).\]
    Then, using the terminology in \cite[definition 5.1]{gn}, $\Gamma$ is a homotopy stable family of compact subset of $\mathbb{R}\times S_{a,r}$ with extend closed boundary $(0,A_{k})\cup(0,E^{2m^{+}(a,\mu)})$. Let
    \[\varphi=\tilde{E}_{\mu}(s,u),\quad c=\sigma(a,\mu),\quad and\quad F=\Big\{(s,u)\in\mathbb{R}\times S_{a,r}: \tilde{E}_{\mu}(s,u)\geqslant c\Big\},\]
    we can check that $F$ satisfies assumptions (F'1) and (F'2) in \cite[theorem 5.2]{gn}.

    Taking a minimizing sequence $\{\gamma_{n}=(\alpha_{n},\theta_{n})\}\subset\Gamma$ for $\sigma(a,\mu)$ with properties that $\alpha_{n}\equiv 0$ and $\theta_{n}\geqslant 0$ for every $\tau\in [0,1]$(if this is not the case, we just have to notice that $\{(0,\alpha_{n}\star\theta_{n})\}$ is also a minimizing sequence). Then, by \cite[theorem 5.2]{gn}, there exists a PS sequence $\{(s_{n}\star w_{n})\}\subset\mathbb{R}\times S_{a,r}$ for $\tilde{E}_{\mu}|_{\mathbb{R}\times S_{a,r}}$ at level $\sigma(a,\mu)$, that is
    \begin{equation}\label{PS}
    	\partial_{s}\tilde{E}_{\mu}(s_{n},w_{n})\rightarrow 0,\quad and\quad\lVert\partial_{u}\tilde{E}_{\mu}(s_{n},w_{n})\rVert_{(T_{w_{n}}S_{a,r})^*}\rightarrow 0\quad as\quad n\rightarrow\infty.
    \end{equation}
    Moreover,
    \begin{equation}\label{dist}
    	\lvert s_{n}\rvert+dist_{W^{1,p}}\big(w_{n},\theta_{n}([0,1])\big)\rightarrow 0\quad as\quad n\rightarrow\infty.
    \end{equation}
    Thus, we have
    \[E_{\mu}(s_{n}\star w_{n})=\tilde{E}_{\mu}(s_{n},w_{n})\rightarrow\sigma(a,\mu),\quad as\quad n\rightarrow\infty,\]
    and
    \begin{align*}
    	dE_{\mu}(s_{n}\star w_{n})(s_{n}\star\varphi)&=\partial_{u}\tilde{E}_{\mu}(0,s_{n}\star w_{n})(s_{n}\star\varphi)\\
    	&=\partial_{u}\tilde{E}_{\mu}(s_{n},w_{n})\varphi\\
    	&=o(1)\lVert\varphi\rVert=o(1)\lVert s_{n}\star\varphi\rVert
    \end{align*}
    for every $\varphi\in T_{w_{n}}S_{a,r}$, which implies $\{u_{n}\}:=\{s_{n}\star w_{n}\}$ is a PS sequence for $E_{\mu}|_{S_{a,r}}$ at level $\sigma(a,\mu)$. Since $E_{\mu}$ is invariant under rotations, by \cite[theorem 2.2]{kjom}, $\{u_{n}\}$ is also a PS sequence for $E_{\mu}|_{S_{a}}$ at level $\sigma(a,\mu)$.

    From (\ref{PS}), we have
    \[P_{\mu}(u_{n})=P_{\mu}(s_{n}\star w_{n})=\partial_{s}\tilde{E}_{\mu}(s_{n},w_{n})\rightarrow 0\]
    as $n\rightarrow\infty$. Thus, by proposition \ref{compactnesslemma}, Lemma \ref{m-<m+1} and \ref{m-<m+2}, one of the cases in proposition \ref{compactnesslemma} holds. If case (i) occurs, we have $u_{n}\rightharpoonup u$ in $W^{1,p}(\mathbb{R}^N)$ and
    \begin{equation}\label{eless0}
    	E_{\mu}(u)\leqslant m^{-}(a,\mu)-\frac{1}{N}S^{\frac{N}{p}}.
    \end{equation}
    Since $u$ solves (\ref{equation}) for some $\lambda<0$, by Theorem \ref{th1} and Lemma \ref{decreasing},
        \[E_{\mu}(u)\geqslant m^{+}\big(\lVert u\rVert_{p},\mu\big)\geqslant m^{+}(a,\mu).\]
    therefore,
        \[m^{+}(a,\mu)\leqslant m^{-}(a,\mu)-\frac{1}{N}S^{\frac{N}{p}},\]
    which contradicts with Lemma \ref{m-<m+1} and \ref{m-<m+2}. This implies that case (ii) in proposition \ref{compactnesslemma} holds, that is $u_{n}\rightarrow u\in S_{a,r}$ in $W^{1,p}(\mathbb{R}^N)$, and $u$ solves (\ref{equation}) for some $\lambda<0$. Moreover, noticing that $\theta_{n}(\tau)\geqslant 0$ for every $\tau\in [0,1]$, then (\ref{dist}) implies $u$ is non-negative and hence positive by strong maximum principle.$\hfill\qed$

\section{\textbf{Existence result to the case $q=p+\frac{p^2}{N}$}}\label{critical}

    In this section, we prove Theorem \ref{th3} for $q=p+\frac{p^2}{N}$. Firstly, we analyzed the properties of $E_{\mu}$ and $\mathcal{P}_{a,\mu}$, and then we construct a mini-max structure.
    \begin{lemma}\label{empty2}
     	$\mathcal{P}_{a,\mu}^{0}=\emptyset$, and $\mathcal{P}_{a,\mu}$ is a smooth manifold of co-dimension $2$ in $W^{1,p}(\mathbb{R}^N)$.
    \end{lemma}
    \begin{proof}
 	    If $u\in\mathcal{P}_{a,\mu}^{0}$, we have
 	        \[\lVert\nabla u\rVert_{p}^p=\mu\gamma_{q}\lVert u\rVert_{q}^q+\lVert u\rVert_{p^*}^{p^*},\quad and\quad p\lVert\nabla u\rVert_{p}^p=\mu q\gamma_{q}^2\lVert u\rVert_{q}^q+p^*\lVert u\rVert_{p^*}^{p^*},\]
 	    which implies $\lVert u\rVert_{p^*}=0$(since $q\gamma_{q}=p$). But we know this impossible since $u\in S_{a}$. The rest of the proof is similar to one of the Lemma \ref{empty}, and hence is omitted.
    \end{proof}

    \begin{lemma}\label{structure2}
    	For every $u\in S_{a}$, the function $\Psi_{u}^{\mu}$ has a unique critical point $t_{u}$ which is a strict maximum point at positive level. Moreover:
    	
    	{\rm (i)}
    	\begin{minipage}[t]{\linewidth}
    		$\mathcal{P}_{a,\mu}=\mathcal{P}_{a,\mu}^{-}$, and $s\star u\in\mathcal{P}_{a,\mu}$ if and only if $s=t_{u}$.
    	\end{minipage}

        {\rm (ii)}
        \begin{minipage}[t]{\linewidth}
        	$t_{u}<0$ is and only if $P_{\mu}(u)<0$.
        \end{minipage}
    	
    	{\rm (iii)}
    	\begin{minipage}[t]{\linewidth}
    		The map $u\in  S_{a}\longmapsto t_{u}\in\mathbb{R}$ is of class $C^1$.
    	\end{minipage}
    \end{lemma}
    \begin{proof}
    	Here we just prove $\mathcal{P}_{a,\mu}=\mathcal{P}_{a,\mu}^{-}$, the rest of the proof is similar to Lemma \ref{structure}. For every $u\in\mathcal{P}_{a,\mu}$, we have $t_{u}=0$ and hence $0$ is a strict maximum point of $\Psi_{u}^{\mu}$. Now, $\big(\Psi_{u}^{\mu}\big)''(0)\leqslant 0$ implies $u\in\mathcal{P}_{a,\mu}^{0}\cup\mathcal{P}_{a,\mu}^{-}$. By lemma \ref{empty2}, we obtain $u\in\mathcal{P}_{a,\mu}^{-}$.
    \end{proof}

    \begin{lemma}
    	We have $m(a,\mu)=m^{-}(a,\mu)>0$.
    \end{lemma}
    \begin{proof}
    	By lemma \ref{structure2}, we know $m(a,\mu)=m^{-}(a,\mu)$. If $u\in\mathcal{P}_{a,\mu}$, then
    	    \[\lVert\nabla u\rVert_{p}^p=\mu\gamma_{q}\lVert u\rVert_{q}^q+\lVert u\rVert_{p^*}^{p^*}.\]
    	Using the Gagliardo-Nirenberg inequality and Sobolev inequality, we have
    	    \[\lVert\nabla u\rVert_{p}^p\leqslant\frac{\mu}{q}C_{N,q}^qa^{\frac{p^2}{N}}\lVert\nabla u\rVert_{p}^p+S^{-\frac{p^*}{p}}\lVert\nabla u\rVert_{p}^{p^*}.\]
    	Combining (\ref{muconcri}), we derive that
    	    \[\inf_{u\in\mathcal{P}_{a,\mu}}\lVert\nabla u\rVert_{p}>0.\]
    	For every $u\in\mathcal{P}_{a,\mu}$, we have
    	    \[E_{\mu}(u)=\frac{1}{N}\lVert\nabla u\rVert_{p}^p-\frac{p\mu}{Nq}\lVert u\rVert_{q}^q\geqslant\frac{1}{N}\Big(1-\frac{p}{q}C_{N,q}^{q}\mu a^{\frac{p^2}{N}}\Big)\lVert\nabla u\rVert_{p}^p,\]
    	which implies $m(a,\mu)>0$.
    \end{proof}

    \begin{lemma}\label{neighbor}
    	There exists $k>0$ sufficiently small such that
    	    \[0<\sup_{u\in A_{k}}E_{\mu}(u)<m(a,\mu),\]
    	and
    	    \[E_{\mu}(u), P_{\mu}(u)>0\quad\forall u\in A_{k},\]
    	where $A_{k}=\Big\{u\in S_{a}: \lVert\nabla u\rVert_{p}\leqslant k\Big\}$.
    \end{lemma}
    \begin{proof}
    	By the Gagliardo-Nirenberg inequality and Sobolev inequality, we have
    	    \[P_{\mu}(u)\geqslant\Big(1-\frac{p}{q}C_{N,q}^q\mu a^{\frac{p^2}{N}}\Big)\lVert\nabla u\rVert_{p}^p-S^{-\frac{p^*}{p}}\lVert\nabla u\rVert_{p}^{p^*},\]
    	and
    	    \[\frac{1}{p}\lVert\nabla u\rVert_{p}^p\geqslant E_{\mu}(u)\geqslant\Big(\frac{1}{p}-\frac{1}{q}C_{N,q}^q\mu a^{\frac{p^2}{N}}\Big)\lVert\nabla u\rVert_{p}^p-\frac{1}{pS^{p^*/p}}\lVert\nabla u\rVert_{p}^{p^*}.\]
    	Thus, we can choose suitable $k>0$ such that the conclusion holds.
    \end{proof}

    By Lemma \ref{neighbor}, we can construct a mini-max structure. Let
        \[E^{c}:=\Big\{u\in S_{a}:E_{\mu}(u)\leqslant c\Big\}.\]
    We introduce the mini-max class
        \[\Gamma:=\Big\{\gamma=(\alpha,\theta)\in C([0,1],\mathbb{R}
        \times S_{a,r}): \gamma(0)\in(0,A_{k}), \gamma(1)\in(0,E^{0})\Big\},\]
    with associated mini-max level
        \[\sigma(a,\mu):=\inf_{\gamma\in\Gamma}\max_{(\alpha,\theta)\in\gamma([0,1])}\tilde{E}_{\mu}(\alpha,\theta),\]
    where
        \[\tilde{E}_{\mu}(s,u):=E_{\mu}(s*u).\]

    In order to use Proposition \ref{compactnesslemma}, we need the following lemmas.
    \begin{lemma}\label{msigma2}
    	We have  $m(a,\mu)=m_{r}(a,\mu)=\sigma(a,\mu)$.
    \end{lemma}
    \begin{proof}
    	The proof of $m(a,\mu)=m_{r}(a,\mu)$ is similar to Lemma \ref{msigma}, and hence we omit it. Next, we prove $m_{r}(a,\mu)=\sigma(a,\mu)$.
    	
    	For every $u\in\mathcal{P}_{a,\mu}\cap S_{a,r}$. By Lemma \ref{structure2}, $t_{u}=0$. Choosing $s_{0}<0<s_{1}$ such that $\lVert\nabla(s_{0}\star u)\rVert_{p}\leqslant k$ and $E_{\mu}(s_{1}\star u)\leqslant 0$, and Defining
    	    \[\gamma_{u}: \tau\in [0,1]\longmapsto\big(0,((1-\tau)s_{0}+\tau s_{1})\star u\big)\in\mathbb{R}\times S_{a,r},\]
    	then $\gamma_{u}\in\Gamma$. Thus
    	    \begin{align*}
    	    	\sigma(a,\mu)&\leqslant\max_{\tau\in [0,1]}\tilde{E}_{\mu}\big(0,((1-\tau)s_{0}+\tau s_{1})\star u\big)\\
    	    	&=\max_{\tau\in [0,1]}E_{\mu}\big(((1-\tau)s_{0}+\tau s_{1})\star u\big)\\
    	    	&\leqslant E_{\mu}(t_{u}\star u)=E_{\mu}(u),
    	    \end{align*}
        which implies $\sigma(a,\mu)\leqslant m_{r}(a,\mu)$.

        For every $\gamma\in\Gamma$, since $\gamma(0)\in(0,A_{k})$, by Lemma \ref{neighbor}, we have
        $P_{\mu}\big(\theta(0)\big)>0$. Now we claim that $P_{\mu}\big(\theta(1)\big)<0$. Indeed, since $\gamma(1)\in(0,E^{0})$, we have $E_{\mu}\big(\theta(1)\big)\leqslant 0$, that is $\Psi_{\theta(1)}^{\mu}(0)\leqslant 0$. Then, by Lemma \ref{structure2}, $t_{\theta(1)}<0$ and hence $P_{\mu}\big(\theta(1)\big)<0$. We know $\tau\longmapsto\alpha(\tau)\star\theta(\tau)$ is continuous for $[0,1]$ to $W^{1,p}(\mathbb{R}^N)$, so there exists $\tau_{\gamma}\in(0,1)$ such that $P_{\mu}\big(\alpha(\tau_{\gamma})\star\theta(\tau_{\gamma})\big)=0$. This implies
            \begin{equation*}
            	\max_{(\alpha,\theta)\in\gamma([0,1])}\tilde{E}_{\mu}(\alpha,\theta)\geqslant\tilde{E}_{\mu}\big(\alpha(\tau_{\gamma}),\theta(\tau_{\gamma})\big)=E_{\mu}\big(\alpha(\tau_{\gamma})\star\theta(\tau_{\gamma})\big)\geqslant m_{r}(a,\mu).
            \end{equation*}
        Therefore, $\sigma(a,\mu)\geqslant m_{r}(a,\mu)$.
    \end{proof}

    \begin{lemma}\label{energy2}
    	We have $m(a,\mu)<\frac{1}{N}S^{\frac{N}{p}}$.
    \end{lemma}
    \begin{proof}
    	Let
    	    \begin{equation}
    		    W_{\varepsilon}(x)=\big(a^{-1}\lVert u_{\varepsilon,\tau}\rVert_{p}\big)^{\frac{N-p}{p}}u_{\varepsilon}\big(a^{-1}\lVert u_{\varepsilon}\rVert_{p}x\big),
    	    \end{equation}
    	Then, we have
    	    \[\lVert W_{\varepsilon}\rVert_{p}^p=a^p,\quad\lVert\nabla W_{\varepsilon}\rVert_{p}^p=\lVert\nabla u_{\varepsilon}\rVert_{p}^p,\quad\lVert W_{\varepsilon}\rVert_{p^*}^{p^*}=\lVert u_{\varepsilon}\rVert_{p^*}^{p^*},\]
    	and
    	    \[\lVert W_{\varepsilon}\rVert_{q}^q=(a\lVert u_{\varepsilon}\rVert_{p}^{-1})^{q(1-\gamma_{q})}\lVert u_{\varepsilon}\rVert_{q}^{q}.\]
    	Thus, there exists unique $\tau_{\varepsilon}\in\mathbb{R}$ such that $\tau_{\varepsilon}\star W_{\varepsilon}\in\mathcal{P}_{a,\mu}$. By the definition of $m(a,\mu)$, we have
    	    \begin{align}\label{m}
    	    	m(a,\mu)&\leqslant E_{\mu}(\tau_{\varepsilon}\star W_{\varepsilon})\nonumber\\
    	    	&=\frac{1}{p}e^{p\tau_{\varepsilon}}\lVert\nabla u_{\varepsilon}\rVert_{p}^p-\frac{\mu}{q}e^{q\gamma_{q}\tau_{\varepsilon}}\big(a\lVert u_{\varepsilon}\rVert_{p}^{-1}\big)^{q(1-\gamma_{q})}\lVert u_{\varepsilon}\rVert_{q}^q-\frac{1}{p^*}e^{p^*\tau_{\varepsilon}}\lVert u_{\varepsilon}\rVert_{p^*}^{p^*}
    	    \end{align}
        If $\liminf_{\varepsilon\rightarrow 0}=-\infty$ or $\limsup_{\varepsilon\rightarrow 0}=+\infty$, then
            \[m(a,\mu)\leqslant\liminf_{\varepsilon\rightarrow 0}E_{\mu}(\tau_{\varepsilon}\star W_{\varepsilon})\leqslant 0,\]
        which contradicts with Lemma \ref{neighbor}. Therefore, there exists $t_{2}>t_{1}$ such that $\tau_{\varepsilon}\in[t_{1},t_{2}]$. Now, (\ref{m}) implies
             \begin{align*}
            	m(a,\mu)&\leqslant S^{\frac{N}{p}}\Big(\frac{1}{p}e^{p\tau_{\varepsilon}}-\frac{1}{p^*}e^{p^*\tau_{\varepsilon}}\Big)+O(\varepsilon^{\frac{N-p}{p-1}})-C\lVert u_{\varepsilon}\rVert_{p}^{q(\gamma_{q}-1)}\lVert u_{\varepsilon}\rVert_{q}^q\\
            	&\leqslant\frac{1}{N}S^{\frac{N}{p}}+O(\varepsilon^{\frac{N-p}{p-1}})-\left\{\begin{array}{ll}
            		C&N>p^2\\
            		C\lvert\log\varepsilon\rvert^{-\frac{p}{N}}&N=p^2\\
            		C\varepsilon^{\frac{p(p^2-N)}{N(p-1)}}&p^{\frac{3}{2}}<N<p^2\\
            		C\varepsilon^{\frac{p^{3/2}-p}{p-1}}\lvert\log\varepsilon\rvert&N=p^\frac{3}{2}
            	\end{array}\right.\\
            &<\frac{1}{N}S^{\frac{N}{p}},
            \end{align*}
        by taking $\varepsilon$ sufficiently small.
    \end{proof}

    Now, we give the proof of Theorem \ref{th3} in the case $q=p+\frac{p^2}{N}$.\\

    \noindent\textbf{Proof of Theorem \ref{th3}:} Let
        \[X=\mathbb{R}\times S_{a,r},\quad\mathcal{F}=\Gamma,\quad and\quad B=(0,A_{k})\cup(0,E^{0}).\]
    Then, using the terminology in \cite[definition 5.1]{gn}, $\Gamma$ is a homotopy stable family of compact subset of $\mathbb{R}\times S_{a,r}$ with extend closed boundary $(0,A_{k})\cup(0,E^{0})$. Let
        \[\varphi=\tilde{E}_{\mu}(s,u),\quad c=\sigma(a,\mu),\quad and\quad F=\Big\{(s,u)\in\mathbb{R}\times S_{a,r}: \tilde{E}_{\mu}(s,u)\geqslant c\Big\},\]
    we can check that $F$ satisfies assumptions (F'1) and (F'2) in \cite[theorem 5.2]{gn}.

    Similar to the proof of Theorem \ref{th2}, there exists a PS sequence $\{(s_{n}\star w_{n})\}\subset\mathbb{R}\times S_{a,r}$ for $\tilde{E}_{\mu}|_{\mathbb{R}\times S_{a,r}}$ at level $\sigma(a,\mu)$
    and we can check that $\{u_{n}\}:=\{s_{n}\star w_{n}\}$ is a PS sequence for $E_{\mu}|_{S_{a,r}}$ at level $\sigma(a,\mu)$. Thus, $\{u_{n}\}$ is also a PS sequence for $E_{\mu}|_{S_{a}}$ at level $\sigma(a,\mu)$.

    Since
        \[P_{\mu}(u_{n})=P_{\mu}(s_{n}\star w_{n})=\partial_{s}\tilde{E}_{\mu}(s_{n},w_{n})\rightarrow 0\]
    as $n\rightarrow\infty$, by proposition \ref{compactnesslemma} and Lemma \ref{energy2}, one of the cases in proposition \ref{compactnesslemma} holds. If case (i) occurs, we have $u_{n}\rightharpoonup u$ in $W^{1,p}(\mathbb{R}^N)$ and
        \begin{equation}\label{eless02}
        	E_{\mu}(u)\leqslant m(a,\mu)-\frac{1}{N}S^{\frac{N}{p}}<0.
        \end{equation}
     Since $u$ solves (\ref{equation}) for some $\lambda<0$, by the Pohozaev identity $P_{\mu}(u)=0$, we can derive that
         \[E_{\mu}(u)=\frac{1}{N}\lVert u\rVert_{P^*}^{p^*}>0,\]
    a contradiction with (\ref{eless02}). This implies that case (ii) in proposition \ref{compactnesslemma} holds, that is $u_{n}\rightarrow u\in S_{a,r}$ in $W^{1,p}(\mathbb{R}^N)$, and $u$ solves (\ref{equation}) for some $\lambda<0$. Moreover, we can choose $u$ is non-negative and hence positive by strong maximum principle.

    It remains to show that $u$ is a ground state. This is a direct result by Proposition \ref{Pohozaev} and Lemma \ref{msigma}.$\hfill\qed$

\section{\textbf{Existence result for the case $p+\frac{p^2}{N}<q<p^*$}}
    In this section, we always assume that the condition in Theorem \ref{th3} holds and $p+\frac{p^2}{N}<q<p^*$. We will omit some of the proofs, since it is very similar to the proofs in Section \ref{critical}.

    \begin{lemma}\label{empty3}
    	$\mathcal{P}_{a,\mu}^{0}=\emptyset$, and $\mathcal{P}_{a,\mu}$ is a smooth manifold of co-dimension $2$ in $W^{1,p}(\mathbb{R}^N)$.
    \end{lemma}
    \begin{proof}
    	If $u\in\mathcal{P}_{a,\mu}^{0}$, we have
    	    \[\lVert\nabla u\rVert_{p}^p=\mu\gamma_{q}\lVert u\rVert_{q}^q+\lVert u\rVert_{p^*}^{p^*},\quad and\quad p\lVert\nabla u\rVert_{p}^p=\mu q\gamma_{q}^2\lVert u\rVert_{q}^q+p^*\lVert u\rVert_{p^*}^{p^*},\]
    	which implies
    	    \[\mu\gamma_{q}(q\gamma_{q}-p)\lVert u\rVert_{q}^q+(p^*-p)\lVert u\rVert_{p^*}^{p^*}=0.\]
    	Since $q\gamma_{q}>p$, we know this impossible.
    \end{proof}

    \begin{lemma}\label{structure3}
    	For every $u\in S_{a}$, the function $\Psi_{u}^{\mu}$ has a unique critical point $t_{u}$ which is a strict maximum point at positive level. Moreover:
    	
    	{\rm (i)}
    	\begin{minipage}[t]{\linewidth}
    		$\mathcal{P}_{a,\mu}=\mathcal{P}_{a,\mu}^{-}$, and $s\star u\in\mathcal{P}_{a,\mu}$ if and only if $s=t_{u}$.
    	\end{minipage}
    	
    	{\rm (ii)}
    	\begin{minipage}[t]{\linewidth}
    		$t_{u}<0$ is and only if $P_{\mu}(u)<0$.
    	\end{minipage}
    	
    	{\rm (iii)}
    	\begin{minipage}[t]{\linewidth}
    		The map $u\in  S_{a}\longmapsto t_{u}\in\mathbb{R}$ is of class $C^1$.
    	\end{minipage}
    \end{lemma}

    \begin{lemma}
    	We have $m(a,\mu)=m^{-}(a,\mu)>0$.
    \end{lemma}
    \begin{proof}
    	If $u\in\mathcal{P}_{a,\mu}$, then
    	    \[\lVert\nabla u\rVert_{p}^p=\mu\gamma_{q}\lVert u\rVert_{q}^q+\lVert u\rVert_{p^*}^{p^*}.\]
    	Using the Gagliardo-Nirenberg inequality and Sobolev inequality, we have
    	    \[\lVert\nabla u\rVert_{p}^p\leqslant\frac{\mu}{q}C_{N,q}^qa^{q(1-\gamma_{q})}\lVert\nabla u\rVert_{p}^{q\gamma_{q}}+S^{-\frac{p^*}{p}}\lVert\nabla u\rVert_{p}^{p^*}.\]
    	Since $q\gamma_{q}>p$, we derive that
    	    \[\inf_{u\in\mathcal{P}_{a,\mu}}\lVert\nabla u\rVert_{p}^p>0,\]
    	and hence by $P_{\mu}(u)=0$, we know
    	    \[\inf_{u\in\mathcal{P}_{a,\mu}}\big(\lVert u\rVert_{q}^q+\lVert u\rVert_{p^*}^{p^*}\big)>0\]
    	For every $u\in\mathcal{P}_{a,\mu}$, we have
    	    \[E_{\mu}(u)=\mu\gamma_{q}\Big(\frac{1}{p}-\frac{1}{q\gamma_{q}}\Big)\lVert u\rVert_{q}^q+\frac{1}{N}\lVert u\rVert_{p^*}^{p^*},\]
    	which implies $m(a,\mu)>0$.
    \end{proof}

    \begin{lemma}\label{neighbor2}
    	There exists $k>0$ sufficiently small such that
    	\[0<\sup_{u\in A_{k}}E_{\mu}(u)<m(a,\mu),\]
    	and
    	\[E_{\mu}(u), P_{\mu}(u)>0\quad\forall u\in A_{k},\]
    	where $A_{k}=\Big\{u\in S_{a}: \lVert\nabla u\rVert_{p}\leqslant k\Big\}$.
    \end{lemma}
    \begin{proof}
    	By the Gagliardo-Nirenberg inequality and Sobolev inequality, we have
    	    \[P_{\mu}(u)\geqslant\lVert\nabla u\rVert_{p}^p-\mu\gamma_{q}C_{N,p,q}^qa^{q(1-\gamma_{q})}\lVert u\rVert_{q}^q-S^{-\frac{p^*}{p}}\lVert\nabla u\rVert_{p}^{p^*},\]
    	and
    	    \[\frac{1}{p}\lVert\nabla u\rVert_{p}^p\geqslant E_{\mu}(u)\geqslant\lVert\nabla u\rVert_{p}^p-\frac{\mu}{q}C_{N,p,q}^qa^{q(1-\gamma_{q})}\lVert u\rVert_{q}^q-\frac{1}{pS^{p^*/p}}\lVert\nabla u\rVert_{p}^{p^*}.\]
    	Thus, we can choose suitable $k>0$ such that the conclusion holds.
    \end{proof}

    \[E^{c}:=\Big\{u\in S_{a}:E_{\mu}(u)\leqslant c\Big\}.\]
    \[\Gamma:=\Big\{\gamma=(\alpha,\theta)\in C([0,1],\mathbb{R}
    \times S_{a,r}): \gamma(0)\in(0,A_{k}), \gamma(1)\in(0,E^{0})\Big\},\]
    \[\sigma(a,\mu):=\inf_{\gamma\in\Gamma}\max_{(\alpha,\theta)\in\gamma([0,1])}\tilde{E}_{\mu}(\alpha,\theta),\]
    \[\tilde{E}_{\mu}(s,u):=E_{\mu}(s*u).\]

    \begin{lemma}\label{msigma3}
    	We have  $m(a,\mu)=m_{r}(a,\mu)=\sigma(a,\mu)$.
    \end{lemma}

        \begin{lemma}\label{energy3}
    	We have $m(a,\mu)<\frac{1}{N}S^{\frac{N}{p}}$.
    \end{lemma}
    \begin{proof}
    	Similar to Lemma \ref{energy2}, we have
    	\begin{align*}
    		m(a,\mu)&\leqslant\frac{1}{N}S^{\frac{N}{p}}+O(\varepsilon^{\frac{N-p}{p-1}})-C\lVert u_{\varepsilon}\rVert_{p}^{q(\gamma_{q}-1)}\lVert u_{\varepsilon}\rVert_{q}^q\\
    		&\leqslant\frac{1}{N}S^{\frac{N}{p}}+O(\varepsilon^{\frac{N-p}{p-1}})-\left\{\begin{array}{ll}
    			C&N>p^2\\
    			C\lvert\log\varepsilon\rvert^{\frac{q(\gamma_{q}-1)}{p}}&N=p^2\\
    			C\varepsilon^{N-\frac{q(p-\gamma_{q})(N-p)}{p(p-1)}}&p^{\frac{3}{2}}\leqslant N<p^2
    		\end{array}\right.\\
    		&<\frac{1}{N}S^{\frac{N}{p}},
    	\end{align*}
    	by taking $\varepsilon$ sufficiently small.
    \end{proof}

    Now, we give the proof of Theorem \ref{th3} in the case $p+\frac{p^2}{N}<q<p^*$.

    \noindent\textbf{Proof of Theorem \ref{th3}:} Similar to Section \ref{critical}, we can obtain a PS sequence $\{u_{n}\}$ for $E_{\mu}|_{S_{a}}$ at level $\sigma(a,\mu)$ with the property $P_{\mu}(u_{n})\rightarrow 0$ as $n\rightarrow\infty$. Therefore, by Proposition \ref{compactnesslemma} and Lemma \ref{energy3}, one of the cases in proposition \ref{compactnesslemma} holds. If case (i) occurs, we have $u_{n}\rightharpoonup u$ in $W^{1,p}(\mathbb{R}^N)$ and
        \begin{equation}\label{eless03}
    	    E_{\mu}(u)\leqslant m(a,\mu)-\frac{1}{N}S^{\frac{N}{p}}<0.
        \end{equation}
    Since $u$ solves (\ref{equation}) for some $\lambda<0$, by the Pohozaev identity $P_{\mu}(u)=0$, we can derive that
        \[E_{\mu}(u)=\mu\gamma_{q}\Big(\frac{1}{p}-\frac{1}{q\gamma_{q}}\Big)\lVert u\rVert_{q}^q+\frac{1}{N}\lVert u\rVert_{p^*}^{p^*}>0,\]
    a contradiction with (\ref{eless03}). This implies that case (ii) in proposition \ref{compactnesslemma} holds. The rest of the proof is same to Similar to Section \ref{compactnesslemma}.$\hfill\qed$

\section{\textbf{Asymptotic behavior of $u_{a,\mu}^{\pm}$}}
    In this section, the dependence of parameter $a$ will not be considered, so we write $u_{a,\mu}^{\pm}$, $\mathcal{P}_{a,\mu}$, $S_{a}$, $m(a,\mu)$, $\lambda_{a,\mu}$, ... as $u_{\mu}^{\pm}$, $\mathcal{P}_{\mu}$, $S$, $m(\mu)$, $\lambda_{\mu}$,....

\subsection{Asymptotic behavior of $u_{\mu}^{+}$ as $\mu\rightarrow 0$}
    \
    \newline
    \indent In this subsection, we always assume that the assumptions of Theorem \ref{th4}(1) hold. In fact, we can prove that $u_{\mu}^{-}\rightarrow 0$ in $D^{1,p}(\mathbb{R}^N)$ as $\mu\rightarrow 0$. Therefore, we need more accurate estimate of how fast $u_{\mu}^{-}$ approaches $0$.
    \begin{lemma}\label{asylammu}
    	We have
    	    \[-\lambda_{\mu}^{+}\sim\lVert\nabla u_{\mu}^{+}\rVert_{p}^p\sim\mu^{\frac{p}{p-q\gamma_{q}}}.\]
    \end{lemma}
    \begin{proof}
    	Since $u_{\mu}^{+}\in\mathcal{P}_{\mu}^{+}$, we have
    	    \begin{equation*}
    	    	\lVert\nabla u_{\mu}^{+}\rVert_{p}^p=\mu\gamma_{q}\lVert u_{\mu}^{+}\rVert_{q}^q+\lVert u_{\mu}^{+}\rVert_{p^*}^{p^*},
    	    \end{equation*}
        and
            \[p\lVert\nabla u_{\mu}^{+}\rVert_{p}^p>\mu q\gamma_{q}^2\lVert u_{\mu}^{+}\rVert_{q}^q+p^*\lVert u_{\mu}^{+}\rVert_{p^*}^{p^*}.\]
        It follows from the Gagliardo-Nirenberg inequality that
            \begin{equation}\label{asynablaq}
            	(p^*-p)\lVert\nabla u_{\mu}^{+}\rVert_{p}^p\leqslant\mu\gamma_{q}(p^*-q\gamma_{q})\lVert u_{\mu}^{+}\rVert_{q}^q\leqslant\mu\gamma_{q}(p^*-q\gamma_{q})C_{N,p,q}^qa^{q(1-\gamma_{q})}\lVert\nabla u_{\mu}^{+}\rVert_{p}^{q\gamma_{q}},
            \end{equation}
    	which together with $q\gamma_{q}<p$ for $p<q<p+\frac{p^2}{N}$, implies
    	    \[\lVert\nabla u_{\mu}^{+}\rVert_{p}^p\leqslant C\mu^{\frac{p}{p-q\gamma_{q}}}.\]
    	Using Gagliardo-Nirenberg inequality again, we know
    	    \begin{equation}\label{asyuq}
    	    	\lVert u_{\mu}^{+}\rVert_{q}^q\leqslant C\mu^{\frac{q\gamma_{q}}{p-q\gamma_{q}}}.
    	    \end{equation}

        Let $u\in S$ be fixed. Then, there exists unique $s_{u}(\mu)\in\mathbb{R}$ such that $s_{u}(\mu)\star\mathcal{P}_{\mu}^{+}$, that is
            \[e^{ps_{u}(\mu)}\lVert\nabla u\rVert_{p}^p=\mu\gamma_{q}e^{q\gamma_{q}s_{u}(\mu)}\lVert u\rVert_{q}^q+e^{p^*s_{u}(\mu)}\lVert u\rVert_{p^*}^{p^*},\]
        and
            \[pe^{ps_{u}(\mu)}\lVert\nabla u\rVert_{p}^p>\mu q\gamma_{q}^2e^{q\gamma_{q}s_{u}(\mu)}\lVert u\rVert_{q}^q+p^*e^{p^*s_{u}(\mu)}\lVert u\rVert_{p^*}^{p^*}.\]
        It follows that
            \[\bigg(\frac{\mu\gamma_{q}\lVert u\rVert_{q}^q}{\lVert\nabla u\rVert_{p}^p}\bigg)^{\frac{1}{p-q\gamma_{q}}}<e^{s_{u}(\mu)}<\bigg(\frac{\mu\gamma_{q}(p^*-q\gamma_{q})\lVert u\rVert_{q}^q}{(p^*-p)\lVert\nabla u\rVert_{p}^p}\bigg)^{\frac{1}{p-q\gamma_{q}}},\]
        which implies $e^{s_{u}(\mu)}\sim\mu^{\frac{1}{p-q\gamma_{q}}}$. Thus, by $s_{u}(\mu)\star u\in\mathcal{P}_{\mu}^{+}$ and $q\gamma_{q}<p$ for $p<q<p+\frac{p^2}{N}$, we have
            \[E_{\mu}(s_{u}(\mu)\star u)=\Big(\frac{1}{p}-\frac{1}{q\gamma_{q}}\Big)e^{ps_{u}(\mu)}\lVert\nabla u\rVert_{p}^p+\Big(\frac{1}{q\gamma_{q}}-\frac{1}{p^*}\Big)e^{p^*s_{u}(\mu)}\lVert u\rVert_{q}^q\sim-\mu^{\frac{p}{p-q\gamma_{q}}}.\]
        Therefore, by $E_{\mu}(s_{u}\star u)\geqslant m^{+}(\mu)$ and
            \[m^{+}(\mu)=\mu\gamma_{q}\Big(\frac{1}{p}-\frac{1}{q\gamma_{q}}\Big)\lVert u_{\mu}^{+}\rVert_{q}^q+\frac{1}{N}\lVert u_{\mu}^{+}\rVert_{p^*}^{p^*}>\mu\gamma_{q}\Big(\frac{1}{p}-\frac{1}{q\gamma_{q}}\Big)\lVert u_{\mu}^{+}\rVert_{q}^q,\]
        we obtain
            \begin{equation}\label{asyuqgeq}
            	\lVert u_{\mu}^{+}\rVert_{q}^q\geqslant C\mu^{\frac{q\gamma_{q}}{p-q\gamma_{q}}}.
            \end{equation}

        Now, (\ref{asyuq}) and (\ref{asyuqgeq}) implies
            \[\lVert u_{\mu}^{+}\rVert_{q}^q\sim\mu^{\frac{q\gamma_{q}}{p-q\gamma_{q}}}.\]
        By the Pohozaev identity, we know
            \[\lambda_{\mu}^{+}a^p=\mu(\gamma_{q}-1)\lVert u_{\mu}^{+}\rVert_{q}^q.\]
        Therefore, by (\ref{asynablaq}),
            \[-\lambda_{\mu}^{+}\sim\lVert\nabla u_{\mu}^{+}\rVert_{p}^p\sim\mu^{\frac{p}{p-q\gamma_{q}}}.\]
    \end{proof}
    \begin{remark}
    	 {\rm It is clear that $u_{\mu}^{+}\rightarrow 0$ in $D^{1,p}(\mathbb{R}^N)$ and $m^{+}(\mu)\rightarrow 0$ as $\mu\rightarrow 0$.}
    \end{remark}

    \noindent\textbf{Proof of Theorem \ref{th4}(1)} Let
        \[w_{\mu}=\mu^{-\frac{N}{p(p-q\gamma_{q})}}u_{\mu}^{+}\Big(\mu^{-\frac{1}{p-q\gamma_{q}}}\cdot\Big)\in S.\]
    By Lemma \ref{asylammu},
        \[\lVert\nabla w_{\mu}\rVert_{p}^p=\mu^{-\frac{p}{p-q\gamma_{q}}}\lVert\nabla u_{\mu}^{+}\rVert_{p}^p\leqslant C,\]
    which implies $\{w_{\mu}\}$ is bounded in $W^{1,p}(\mathbb{R}^N)$. Thus, there exists $w\in W^{1,p}(\mathbb{R}^N)$ such that $w_{\mu}\rightharpoonup w$ in $W^{1,p}(\mathbb{R}^N)$, $w_{\mu}\rightarrow w$ in $L^q(\mathbb{R}^N)$, $w_{\mu}\rightarrow w$ a.e. in $\mathbb{R}^N$.

    We know $u_{\mu}^{+}$ solves (\ref{equation}) for some $\lambda_{\mu}^{+}$. Direct calculations show that
       \[-\Delta_{p}w_{\mu}=\lambda_{\mu}^{+}\mu^{-\frac{p}{p-q\gamma_{q}}}w_{\mu}^{p-1}+w_{\mu}^{q-1}+\mu^{\frac{p^2}{(N-p)(p-q\gamma_{q})}}w_{\mu}^{p^*-1}.\]
    By Lemma \ref{asylammu}, we know $\{\lambda_{\mu}^{+}\mu^{-\frac{p}{p-q\gamma_{q}}}\}$ is bounded, hence there exists $\sigma_{0}>0$ such that $\lambda_{\mu}^{+}\mu^{-\frac{p}{p-q\gamma_{q}}}\rightarrow-\sigma_{0}$ as $\mu\rightarrow 0$. Considering the mapping $T:W^{1,p}(\mathbb{R}^N)\rightarrow\Big(W^{1,p}(\mathbb{R}^N)\Big)^*$ which is given by
        \[<Tu,v>=\int_{\mathbb{R}^N}\big(\lvert\nabla u\rvert^{p-2}\nabla u\cdot\nabla v+\sigma_{0}\lvert u\rvert^{p-2}uv\big)dx.\]
    Then, similar to \cite[Lemma 3.6]{hdly}, we can prove that $w_{\mu}\rightarrow w$ in $W^{1,p}(\mathbb{R}^N)$. Thus, $w$ satisfies the equation
        \[-\Delta_{p}w+\sigma_{0}w^{p-1}=w^{q-1}.\]
    Let $\tilde{w}=\sigma_{0}^{\frac{1}{p-q}}w\big(\sigma_{0}^{-\frac{1}{p}}\cdot\big)$. It is not difficult to shows that
        \[\sigma_{0}^{\frac{1}{p-q}}\mu^{-\frac{N}{p(p-q\gamma_{q})}}u_{\mu}^{+}\big(\sigma_{0}^{-\frac{1}{p}}\mu^{-\frac{1}{p-q\gamma_{q}}}\cdot\big)\rightarrow\tilde{w}\]
    in $W^{1,p}(\mathbb{R}^N)$, and $\tilde{w}$ satisfies (\ref{uniqueness}). By the regularity and properties of $u_{\mu}^{+}$, we can derive that $\tilde{w}$ is the "ground state" of (\ref{uniqueness}) and hence $\tilde{w}=\phi_{0}$. Now, using $u_{\mu}^{+}\in S_{a}$, we can obtain that
        \[\sigma_{0}=\bigg(\frac{a^p}{\lVert\phi_{0}\rVert_{p}^p}\bigg)^{\frac{p(q-p)}{p^2-N(q-p)}}.\]
    $\hfill\qed$

\subsection{Asymptotic behavior of $u_{\mu}^{-}$ as $\mu\rightarrow 0$}
    \
    \newline
    \indent In this subsection, we always assume that the assumptions of Theorem \ref{th4}(2) or (3) hold. Unlike $u_{\mu}^{+}$, we can prove that $\lVert\nabla u_{\mu}^{-}\rVert_{p}^p,\lVert u_{\mu}^{-}\rVert_{p^*}^{p^*}\rightarrow S^{\frac{N}{p}}$ as $\mu\rightarrow 0$.
    \begin{lemma}\label{infmax}
    	Let $\mu>0$ satisfies {\rm (\ref{muconsub})} for $p<q<p+\frac{p^2}{N}$ and { \rm(\ref{muconcri})} for $q=p+\frac{p^2}{N}$. Then,
    	    \[m^{-}(\mu)=\inf_{u\in S}\max_{s\in\mathbb{R}}E_{\mu}(s\star u).\]
    \end{lemma}
    \begin{proof}
    	For every $v\in\mathcal{P}_{\mu}^{-}$, by Lemma \ref{structure}, \ref{structure2} and \ref{structure3}, we know
    	    \[E_{\mu}(v)=\max_{s\in\mathbb{R}}E_{\mu}(s\star v)\geqslant\inf_{u\in S}\max_{s\in\mathbb{R}}E_{\mu}(s\star u),\]
    	and hence
    	    \[m^{-}(\mu)\geqslant\inf_{u\in S}\max_{s\in\mathbb{R}}E_{\mu}(s\star u).\]
    	
    	For every $v\in S_{a}$, by Lemma \ref{structure}, \ref{structure2} and \ref{structure3}, we know
    	    \[\max_{s\in\mathbb{R}}E_{\mu}(s\star v)=E_{\mu}(t_{v}\star v)\geqslant\inf_{u\in\mathcal{P}_{\mu}^{-}}E_{\mu}(u),\]
    	and hence
    	    \[\inf_{u\in S}\max_{s\in\mathbb{R}}E_{\mu}(s\star u)\geqslant m^{-}(\mu).\]
    \end{proof}

    \begin{lemma}\label{noninc}
    	Let $\tilde{\mu}>0$ satisfies {\rm (\ref{muconsub})} for $p<q<p+\frac{p^2}{N}$ and { \rm(\ref{muconcri})} for $q=p+\frac{p^2}{N}$. Then, the function $\mu\in(0,\tilde{\mu}]\mapsto m^{-}(\mu)\in\mathbb{R}$ is non-increasing.
    \end{lemma}
    \begin{proof}
    	Let $0<\mu_{1}<\mu_{2}\leqslant\tilde{\mu}$. By Lemma \ref{infmax}, we have
    	    \begin{align*}
    	    	m^{-}(\mu_{2})&=\inf_{u\in S}\max_{s\in\mathbb{R}}E_{\mu_{2}}(s\star u)=\inf_{u\in S}\max_{s\in\mathbb{R}}\Big(E_{\mu_{1}}(s\star u)-\frac{\mu_{2}-\mu_{1}}{q}\lVert u\rVert_{q}^q\Big)\\
    	    	&\leqslant\inf_{u\in S}\max_{s\in\mathbb{R}}E_{\mu_{1}}(s\star u)=m^{-}(\mu_{1}).
    	    \end{align*}
    \end{proof}

    \begin{lemma}\label{asyuneg}
    	We have $\lVert\nabla u_{\mu}^{-}\rVert_{p}^p, \lVert u_{\mu}^{-}\rVert_{p^*}^{p^*}\rightarrow S^{\frac{N}{p}}$ and $m^{-}(\mu)\rightarrow S^{\frac{N}{p}}/N$ as $\mu\rightarrow 0$.
    \end{lemma}
    \begin{proof}
    	Using the fact $m^{-}(\mu)<S^{\frac{N}{p}}/N$ and slightly modifying the proof of Lemma \ref{bounded}, we know $\{u_{\mu}^{-}\}$ is bounded in $W^{1,p}(\mathbb{R}^N)$. Thus, we can assume that $\lVert\nabla u_{\mu}^{-}\rVert_{p}^p\rightarrow l$ as $\mu\rightarrow 0$.
    	
    	We claim that $l\neq 0$. Suppose by contradiction that $l=0$, then $E_{\mu}(u_{\mu}^{-})\rightarrow 0$ as $\mu\rightarrow 0$. However, by Lemma \ref{mplus}, \ref{neighbor}, \ref{neighbor2} and \ref{noninc},, we know $E_{\mu}(u_{\mu}^{-})\geqslant m^{-}(\tilde{\mu})>0$ for every $0<\mu\leqslant\tilde{\mu}$, a contradiction.
    	
    	Now, by $P_{\mu}(u_{\mu}^{-})=0$, we deduce that
    	    \[\lVert u_{\mu}^{-}\rVert_{p^*}^{p^*}=\lVert\nabla u_{\mu}^{-}\rVert_{p}^p-\mu\gamma_{q}\lVert u_{\mu}^{-}\rVert_{q}^q\rightarrow l\]
    	as $\mu\rightarrow 0$. Therefore, by the Sobolev inequality we have $l\geqslant Sl^{\frac{p}{p^*}}$ which implies $l\geqslant S^{\frac{N}{p}}$. On the other hand, since
    	    \[\frac{l}{N}=\lim_{\mu\rightarrow 0}\bigg(\frac{1}{N}\lVert\nabla u_{\mu}^{-}\rVert_{p}^p-\mu\gamma_{q}\Big(\frac{1}{q\gamma_{q}}-\frac{1}{p^*}\Big)\lVert u_{\mu}^{-}\rVert_{p^*}^{p^*}\bigg)=\lim_{\mu\rightarrow 0}E_{\mu}(u_{\mu}^{-})\leqslant\frac{1}{N}S^{\frac{N}{p}},\]
    	we obtain that $l=S^{\frac{N}{p}}$. We complete the proof.
    \end{proof}

    \noindent\textbf{Proof of Theorem \ref{th4}(2)} Lemma \ref{asyuneg} implies $\{u_{\mu}^{-}\}$ is a minimizing sequence of the following minimizing problem:
        \[S=\inf_{u\in D^{1,p}(\mathbb{R}^N)\backslash\{0\}}\frac{\lVert\nabla u\rVert_{p}^p}{\lVert u\rVert_{p^*}^{p^*}}.\]
    Since $u_{\mu}^{-}$ is radially symmetric, by \cite[Theorem 4.9]{sm}, there exists $\sigma_{\mu}>0$ such that
        \[w_{\mu}=\sigma_{\mu}^{\frac{N-p}{p}}u_{\mu}^{-}(\sigma_{\mu}\cdot)\rightarrow U_{\varepsilon_{0}}\]
    in $D^{1,p}(\mathbb{R}^N)$ as $\mu\rightarrow 0$ for some $\varepsilon_{0}>0$. We know $U_{\varepsilon_{0}}\notin S$ for $N\leqslant p^2$, and $\lVert w_{\mu}\rVert_{p}^p=a^p/\sigma_{\mu}^p$, by the Fatou lemma, $\sigma_{\mu}\rightarrow 0$ as $\mu\rightarrow 0$.$\hfill\qed$

    $U_{\varepsilon_{0}}\notin S$ for $N\leqslant p^2$ implies $w_{\mu}$ will not converge to $U_{\varepsilon_{0}}$ in $W^{1,p}(\mathbb{R})$ as $\mu\rightarrow 0$. However, since $U_{\varepsilon_{0}}\in S$ for $N>p^2$, in the next, we will prove $u_{\mu}^{-}\rightarrow U_{\varepsilon_{0}}$ in $W^{1,p}(\mathbb{R}^N)$ as $\mu\rightarrow 0$ for $N>p^2$. \\

    \noindent\textbf{Proof of Theorem \ref{th4}(3)} Since $\lVert U_{\varepsilon}\rVert_{p}^p$=$\varepsilon^p\lVert U_{1}\rVert_{p}^p$, we can choose $\varepsilon_{0}>0$ satisfies $\lVert U_{\varepsilon_{0}}\rVert_{p}^p=a^p$. Hence, there exists unique $t(\mu)\in\mathbb{R}$ such that $t(\mu)\star U_{\varepsilon_{0}}\in\mathcal{P}_{\mu}^{-}$, that is
        \[e^{pt(\mu)}S^{\frac{N}{p}}=\mu\gamma_{q}e^{q\gamma_{q}t(\mu)}\lVert U_{\varepsilon_{0}}\rVert_{q}^q+e^{p^*t(\mu)}S^{\frac{N}{p}}.\]
    Clearly, $t(0)=0$. Now, using implicit function theorem, $t(\mu)$ is of class $C^1$ in a neighborhood of $0$. By direct calculation, we have
        \[t'(0)=-\frac{\gamma_{q}\lVert U_{\varepsilon_{0}}\rVert_{q}^q}{(p^*-p)S^{N/p}},\]
    which implies
        \[t(\mu)=t(0)+t'(0)\mu+o(\mu)=-\frac{\gamma_{q}\lVert U_{\varepsilon_{0}}\rVert_{q}^q}{(p^*-p)S^{N/p}}\mu+o(\mu).\]
    Consequently,
        \begin{align*}
        	E_{\mu}\big(t(\mu)\star U_{\varepsilon_{0}}\big)&=\Big(\frac{1}{p}-\frac{1}{q\gamma_{q}}\Big)e^{pt(\mu)}\lVert U_{\varepsilon_{0}}\rVert_{p}^p+\Big(\frac{1}{q\gamma_{q}}-\frac{1}{p^*}\Big)e^{p^*t(\mu)}\lVert U_{\varepsilon_{0}}\rVert_{p^*}^{p^*}\\
        	&=\Big(\frac{1}{p}-\frac{1}{q\gamma_{q}}\Big)\bigg(1-\frac{p\gamma_{q}\lVert U_{\varepsilon_{0}}\rVert_{q}^q}{(p^*-p)S^{N/p}}\mu+o(\mu)\bigg)S^{\frac{N}{p}}+\\
        	&\qquad\qquad\Big(\frac{1}{q\gamma_{q}}-\frac{1}{p^*}\Big)\bigg(1-\frac{p^*\gamma_{q}\lVert U_{\varepsilon_{0}}\rVert_{q}^q}{(p^*-p)S^{N/p}}\mu+o(\mu)\bigg)S^{\frac{N}{p}}\\
        	&=\frac{1}{N}S^{\frac{N}{p}}-\frac{\lVert U_{\varepsilon_{0}}\rVert_{q}^q}{q}\mu+o(\mu).
        \end{align*}
    By the definition of $m^{-}(\mu)$, we have
        \begin{equation}\label{asyumuue}
        	m^{-}(\mu)=\frac{1}{N}\lVert\nabla u_{\mu}^{-}\rVert_{p}^p-\mu\gamma_{q}\Big(\frac{1}{q\gamma_{q}}-\frac{1}{p^*}\Big)\lVert u_{\mu}^{-}\rVert_{q}^q\leqslant\frac{1}{N}S^{\frac{N}{p}}-\frac{\lVert U_{\varepsilon_{0}}\rVert_{q}^q}{q}\mu+o(\mu).
        \end{equation}

    By the Sobolev inequality, we have
        \begin{align*}
        	\lVert u_{\mu}^{-}\rVert_{p}^p&\geqslant S\lVert u_{\mu}^{-}\rVert_{p^*}^p=S\Big(\lVert\nabla u_{\mu}\rVert_{p}^p-\mu\gamma_{q}\lVert u_{\mu}^{-}\rVert_{q}^q\Big)^{\frac{p}{p^*}}\\
        	&=S\lVert\nabla u_{\mu}^{-}\rVert_{p}^{\frac{p^2}{p^*}}\bigg(1-\frac{\gamma_{q}\lVert u_{\mu}^{-}\rVert_{q}^q}{S^{N/p}}\mu+o(\mu)\bigg)^{\frac{p}{p^*}}\\
        	&=S\lVert\nabla u_{\mu}^{-}\rVert_{p}^{\frac{p^2}{p^*}}\bigg(1-\frac{p\gamma_{q}\lVert u_{\mu}^{-}\rVert_{q}^q}{p^*S^{N/p}}\mu+o(\mu)\bigg).
        \end{align*}
    Thus,
        \[\lVert\nabla u_{\mu}^{-}\rVert_{p}^p\geqslant S^{\frac{N}{p}}\bigg(1-\frac{p\gamma_{q}\lVert u_{\mu}^{-}\rVert_{q}^q}{p^*S^{N/p}}+o(\mu)\bigg)^{\frac{N}{p}}=S^{\frac{N}{p}}-\frac{(N-p)\gamma_{q}\lVert u_{\mu}^{-}\rVert_{q}^q}{p}\mu+o(\mu),\]
    which together with (\ref{asyumuue}), implies
        \begin{equation}\label{umugeque}
        	\lVert u_{\mu}^{-}\rVert_{q}^q\geqslant\lVert U_{\varepsilon_{0}}\rVert_{q}^q+o(1).
        \end{equation}

    Since $\{u_{\mu}^{-}\}$ is bounded in $W^{1,p}(\mathbb{R}^N)$, there exists $u\in W^{1,p}(\mathbb{R}^N)$ such that $u_{\mu}^{-}\rightharpoonup u$ in $W^{1,p}(\mathbb{R}^N)$, $u_{\mu}^{-}\rightarrow u$ in $L^q(\mathbb{R}^N)$ and $u_{\mu}\rightarrow u$ a.e. in $\mathbb{R}^N$ as $\mu\rightarrow 0$. By (\ref{umugeque}), we know $u\neq 0$.

    By the Pohozaev identity, we know
        \[\lambda_{\mu}^{-}a^p=\mu(\gamma_{q}-1)\lVert u_{\mu}^{-}\rVert_{q}^q\rightarrow 0\]
    as $\mu\rightarrow 0$. Thus, by the weak convergence, $u$ is the solution of the equation
        \[-\Delta_{p}U=U^{p^*-1},\]
    which implies $u=U_{\varepsilon,y}$ for some $(\varepsilon,y)\in(\mathbb{R}^+,\mathbb{R}^N)$. Since $u_{\mu}^{-}$ is radially symmetric, we have $y=0$ and hence $u=U_{\varepsilon}$. Now, by the Fatou lemma and (\ref{umugeque}), we obtain
        \[\lVert U_{\varepsilon}\rVert_{p}^p=\lVert u\rVert_{p}^p\leqslant a^p=\lVert U_{{\varepsilon_{0}}}\rVert_{p}^p,\quad\lVert U_{\varepsilon}\rVert_{q}^q=\lVert u\rVert_{q}^q=\lVert U_{\varepsilon_{0}}\rVert_{q}^q.\]
    Therefore, $\varepsilon=\varepsilon_{0}$ and $u=U_{\varepsilon_{0}}$. Finally, since
        \[\lVert U_{\varepsilon_{0}}\rVert_{p}^p=\lim_{\mu\rightarrow 0}\lVert u_{\mu}^{-}\rVert_{p}^p=a^p,\quad and\quad\lVert\nabla U_{\varepsilon_{0}}\rVert_{p}^p=\lim_{\mu\rightarrow 0}\lVert\nabla u_{\mu}^{-}\rVert_{p}^p=S^{\frac{N}{p}},\]
    the Br\'{e}zis-Lieb lemma \cite{bhle} implies $u_{\mu}\rightarrow U_{\varepsilon_{0}}$ in $W^{1,p}(\mathbb{R}^N)$ as $\mu\rightarrow 0$.$\hfill\qed$

\subsection{Asymptotic behavior of $u_{\mu}^{-}$ as $\mu$ goes to its upper bound}
    \
    \newline
    \indent In this subsection, we always assume that the assumptions of Theorem \ref{th5} hold. Firstly, we assume that $q=p+\frac{p^2}{N}$ and we prove $\bar{\alpha}$ is the upper bound of $\mu$ when $q=p+\frac{p^2}{N}$.
    \begin{lemma}\label{infty}
    	We have
    	    \[\sup_{u\in S}\frac{\lVert\nabla u\rVert_{p}^p}{\lVert u\rVert_{q}^q}=+\infty.\]
    \end{lemma}
    \begin{proof}
    	By the Sobolev inequality, we just have to prove that
    	    \[\sup_{u\in S}\frac{\lVert u\rVert_{p^*}^p}{\lVert u\rVert_{q}^q}=+\infty.\]
    	Let
    	    \[u_{k}(x)=\frac{A_{k}\varphi_{k}(x)}{(1+\lvert x\rvert^2)^{a_{k}}}\in S,\]
    	where $A_{k}>0$ is a constant dependent on $k$,
    	    \[a_{k}=\frac{N-p}{2p}-\frac{1}{\log\log(k+2)},\]
    	and $\varphi_{k}\in C_{c}^{\infty}(\mathbb{R}^N)$ is a radial cut-off function satisfies
    	    \[0\leqslant\varphi_{k}\leqslant 1,\quad\varphi_{k}=1\ in\ B_{k},\quad and\quad\varphi_{k}=0\ in\ B_{k+1}^c.\]
    	Since $u_{k}\in S$ and
    	    \begin{align*}
    	    	\lVert u_{k}\rVert_{p}^p&=A_{k}^p\int_{\mathbb{R}^N}\frac{\varphi_{k}^p(x)}{(1+\lvert x\rvert^2)^{pa_{k}}}dx\sim A_{k}^p\int_{0}^{+\infty}\frac{\varphi_{k}^p(r)r^{N-1}}{(1+r^2)^{pa_{k}}}dr\\
    	    	&\sim A_{k}^p\int_{0}^{k}\frac{r^{N-1}}{(1+r^2)^{pa_{k}}}dr\sim\frac{A_{k}^pk^{N-2pa_{k}}}{N-2pa_{k}}\sim A_{k}^pk^{N-2pa_{k}}
    	    \end{align*}
        as $k\rightarrow\infty$, we have $A_{k}\sim k^{2a_{k}-N/p}$ as $k\rightarrow\infty$. Therefore,
            \[\lVert u_{k}\rVert_{q}^q=A_{k}^q\int_{\mathbb{R}^N}\frac{\varphi_{k}^q(x)}{(1+\lvert x\rvert^2)^{qa_{k}}}dx\sim k^{2qa_{k}-N-p}\int_{0}^{k}\frac{r^{N-1}}{(1+r^2)^{qa_{k}}}dr\sim\frac{1}{k^p},\]
        and
            \[\lVert u_{k}\rVert_{p^*}^{p^*}=A_{k}^{p^*}\int_{\mathbb{R}^N}\frac{A_{k}^{p^*}}{(1+\lvert x\rvert^2)^{p^*a_{k}}}dx\sim k^{2p^*a_{k}-\frac{Np^*}{p}}\int_{0}^{k}\frac{r^{N-1}}{(1+r^2)^{p^*a_{k}}}dr\sim\frac{1}{(N-p^*a_{k})k^{p^*}},\]
        as $k\rightarrow\infty$, which implies
            \[\frac{\lVert u_{k}\rVert_{p^*}^p}{\lVert u_{k}\rVert_{q}^q}\sim\frac{1}{(N-p^*a_{k})^{p/p^*}}\rightarrow+\infty\]
        as $k\rightarrow\infty$.
    \end{proof}

    \begin{lemma}\label{notemp}
    	For $\mu\geqslant\bar{\alpha}$, we have $\mathcal{P}_{\mu}=\mathcal{P}_{\mu}^{-}\neq\emptyset$.
    \end{lemma}
    \begin{proof}
    	For every $\mu\geq\bar{\alpha}$, by Lemma \ref{infty}, there exists $u\in S$ such that $\lVert\nabla u\rVert_{p}^p>\mu\gamma_{q}\lVert u\rVert_{q}^q$. Then,
    	    \[\Psi_{u}^{\mu}(s)=\frac{1}{p}e^{ps}\big(\lVert\nabla u\rVert_{p}^p-\mu\gamma_{q}\lVert u\rVert_{q}^q\big)-\frac{1}{p^*}e^{p^*s}\lVert u\rVert_{p^*}^{p^*}\]
    	has a critical point $t_{u}\in\mathbb{R}$. By proposition \ref{Pohozaev}, we know $t_{u}\star u\in\mathcal{P}_{\mu}$ which implies $\mathcal{P}_{\mu}\neq\emptyset$.
    	
    	If there exists $v\in\mathcal{P}_{\mu}^{0}\cup\mathcal{P}_{\mu}^{+}$, we have
    	    \[\lVert\nabla v\rVert_{p}^p=\mu\gamma_{q}\lVert v\rVert_{q}^q+\lVert v\rVert_{p^*}^{p^*},\quad and\quad p\lVert\nabla v\rVert_{p}^p\geqslant\mu q\gamma_{q}^2\lVert v\rVert_{q}^q+p^*\lVert v\rVert_{p^*}^{p^*},\]
    	which implies $\lVert v\rVert_{p^*}^{p^*}\leqslant 0$(since $q\gamma_{q}=p$), a contradiction since $v\in S$.
    \end{proof}

    \begin{lemma}\label{asycrinonexi}
    	For $\mu\geqslant\bar{\alpha}$, there is $m^{-}(\mu)=0$, and $m^{-}(\mu)$ can not be attained by any $u\in S$.
    \end{lemma}
    \begin{proof}
    	For every $\mu\geqslant\bar{\alpha}$, by Lemma \ref{infty}, there exists $\{u_{n}\}\subset S$ such that
    	    \[\frac{\lVert\nabla u_{n}\rVert_{p}^p}{\lVert u_{n}\rVert_{q}^q}>\mu\gamma_{q},\quad and\quad\frac{\lVert\nabla u_{n}\rVert_{p}^p}{\lVert u_{n}\rVert_{q}^q}\rightarrow\mu\gamma_{q}\]
    	as $n\rightarrow\infty$. Without loss of generality, by scaling $\frac{as_{n}^{N/p}}{\lVert u_{n}\rVert_{p}^p}u_{n}(sx)$ if necessary, we may assume that $\lVert u_{n}\rVert_{q}^q=1$. Then, we have $\lVert\nabla u_{n}\rVert_{p}^p>\mu\gamma_{q}$ and $\lVert\nabla u_{n}\rVert_{p}^p\rightarrow\mu\gamma_{q}$ as $n\rightarrow\infty$.
    	
    	Now, we can obtain the function
    	    \[\Psi_{u_{n}}^{\mu}(s)=\frac{1}{p}e^{ps}\big(\lVert\nabla u_{n}\rVert_{p}^p-\mu\gamma_{q}\big)-\frac{1}{p^*}e^{p^*s}\lVert u_{n}\rVert_{p^*}^{p^*}\]
    	has a critical point $t_{n}\in\mathbb{R}$. Hence, $t_{n}\star u_{n}\in\mathcal{P}_{\mu}^{-}$ and we have
    	    \begin{equation}\label{tn}
    	    	\lVert u_{n}\rVert_{p}^p-\mu\gamma_{q}=e^{(p^*-p)t_{n}}\lVert u_{n}\rVert_{p^*}^{p^*}.
    	    \end{equation}
        By the Sobolev inequality $S\lVert u_{n}\rVert_{p}^p\leqslant\lVert\nabla u_{n}\rVert_{p}^p$ and H\"older inequality $\lVert u_{n}\rVert_{q}^q\leqslant\lVert u_{n}\rVert_{p^*}^{q\gamma_{q}}\lVert u_{n}\rVert_{p}^{q(1-\gamma_{q})}$, we obtain $\lVert u_{n}\rVert_{p^*}^{p^*}\sim 1$. Thus, (\ref{tn}) implies $t_{n}\rightarrow-\infty$ as $n\rightarrow\infty$.

        By the definition of $m^{-}(\mu)$, we have $m^{-}(\mu)\leqslant E_{\mu}(t_{n}\star u_{n})$, that is
            \[m^{-}(\mu)\leqslant\frac{1}{N}e^{pt_{n}}\big(\lVert\nabla u_{n}\rVert_{p}^p-\mu\gamma_{q}\big)\rightarrow 0\]
        as $n\rightarrow\infty$, which implies $m^{-}(\mu)\leqslant 0$. For every $u\in\mathcal{P}_{\mu}^{-}$, we have
            \[E_{\mu}(u)=\frac{1}{N}\lVert u\rVert_{p^*}^{p^*},\]
        which implies $m^{-}(\mu)\geqslant 0$. Therefore, $m^{-}(\mu)=0$.

        If there exists $u\in\mathcal{P}_{\mu}^{-}$ such that $E_{\mu}(u)=0$. Then, it must have $u\equiv 0$ which contradicts with $u\in S$.
    \end{proof}

    Theorem \ref{th3} and Lemma \ref{asycrinonexi} implies $\bar{\alpha}$ is the upper bound of $\mu$. Therefore, we can study the asymptotic behavior of $u_{a,\mu}^{-}$ as $\mu\rightarrow\bar{\alpha}$. We give the asymptotic behavior of $\lambda_{\mu}^{-}$ as follows.

    \begin{lemma}\label{asycrilam}
    	For $\mu<\bar{\alpha}$, we have
    	    \[-\lambda_{\mu}^{-}\sim\lVert\nabla u_{\mu}^{-}\rVert_{p}^p\sim(\bar{\alpha}-\mu)^{\frac{N-p}{p}},\]
    	as $\mu\rightarrow\bar{\alpha}$.
    \end{lemma}
    \begin{proof}
    	By the Gagliardo-Nirenberg inequality and Sobolev inequality,
    	    \[\lVert\nabla u_{\mu}^{-}\rVert_{p}^p=\mu\gamma_{q}\lVert u_{\mu}^{-}\rVert_{q}^q+\lVert u_{\mu}^{-}\rVert_{p^*}^{p^*}\leqslant\frac{\mu}{\bar{\alpha}}\lVert\nabla u_{\mu}^{-}\rVert_{p}^p+S^{\frac{p^*}{p}}\lVert\nabla u_{\mu}^{-}\rVert_{p}^{p^*},\]
    	which implies
    	    \begin{equation}\label{asycrinab}
    	    	\lVert\nabla u_{\mu}^{-}\rVert_{p}^p\geqslant C(\alpha-\mu)^{\frac{N-p}{p}},
    	    \end{equation}
        as $\mu\rightarrow\bar{\alpha}$.

        Let $\varphi=\frac{a}{\lVert\psi_{0}\rVert_{p}}\psi_{0}\in S$, where $\psi_{0}$ is a minimizer of Gagliardo-Nirenberg inequality. Then, direct calculations shows that $t_{\mu}\star\varphi\in\mathcal{P}_{\mu}^{-}$, where
            \[e^{t_{\mu}}=\frac{\lVert\psi_{0}\rVert_{p}\lVert\nabla\psi_{0}\rVert_{p}^{p/(p^*-p)}}{a\lVert\psi_{0}\rVert_{p^*}^{p^*/(p^*-p)}}\Big(1-\frac{\mu}{\bar{\alpha}}\Big)^{\frac{1}{p^*-p}}.\]
        Therefore,
            \[m^{-}(\mu)\leqslant E_{\mu}(t_{\mu}\star\varphi)=\frac{a^p\lVert\nabla\psi_{0}\rVert_{p}^p}{N\lVert\psi_{0}\rVert_{p}^p}\Big(1-\frac{\mu}{\bar{\alpha}}\Big)e^{pt_{\mu}}
            =\frac{1}{N}\Big(1-\frac{\mu}{\bar{\alpha}}\Big)^{\frac{N}{p}}\frac{\lVert\nabla\psi_{0}\rVert_{p}^N}{\lVert\psi_{0}\rVert_{p^*}^N}.\]
        Since
            \[E_{\mu}(u_{\mu}^{-})=\frac{1}{N}\lVert u_{\mu}^{-}\rVert_{p^*}^{p^*},\]
        we have
            \[\lVert u_{\mu}^{-}\rVert_{p^*}^{p^*}\leqslant C(\bar{\alpha}-\mu)^{\frac{N}{p}}.\]
        Now, by the H\"older inequality, we obtain
            \[\lVert\nabla u_{\mu}^{-}\rVert_{p}^p=\mu\gamma_{q}\lVert u_{\mu}^{-}\rVert_{q}^q+\lVert u_{\mu}^{-}\rVert_{p^*}^{p^*}\leqslant\frac{\mu}{\bar{\alpha}}\lVert u_{\mu}^{-}\rVert_{p^*}^p+\lVert u_{\mu}^{-}\rVert_{p^*}^{p^*}\leqslant C(\bar{\alpha}-\mu)^{\frac{N-p}{p}},\]
        as $\mu\rightarrow\bar{\alpha}$, which together with (\ref{asycrinab}), implies
            \[\lVert\nabla u_{\mu}^{-}\rVert_{p}^p\sim(\bar{\alpha}-\mu)^{\frac{N-p}{p}}.\]

        Using the Gagliardo-Nirenberg inequality and Sobolev inequality again, we have
            \[\lVert u_{\mu}^{-}\rVert_{q}^q\leqslant C\lVert\nabla u_{\mu}^{-}\rVert_{p}^p\leqslant C(\bar{\alpha}-\mu)^{\frac{N-p}{p}},\]
        and
            \[\mu\gamma_{q}\lVert u_{\mu}^{-}\rVert_{q}^q=\lVert\nabla u_{\mu}^{-}\rVert_{p}^p-\lVert u_{\mu}^{-}\rVert_{p}^{p^*}\geqslant\lVert\nabla u_{\mu}^{-}\rVert_{p}^p-S^{\frac{p^*}{p}}\lVert\nabla u_{\mu}^{-}\rVert_{p}^{p^*}\geqslant C(\bar{\alpha}-\mu)^{\frac{N-p}{p}},\]
        as $\mu\rightarrow\bar{\alpha}$. Therefore,
            \[\lVert u_{\mu}^{-}\rVert_{q}^q\sim(\bar{\alpha}-\mu)^{\frac{N-p}{p}},\]
        as $\mu\rightarrow\bar{\alpha}$. By the Pohozaev identity $\lambda_{\mu}^{-}a^p=\mu(\gamma_{q}-1)\lVert u_{\mu}^{-}\rVert_{q}^q$, we know
            \[-\lambda_{\mu}^{-}\sim(\bar{\alpha}-\mu)^{\frac{N-p}{p}},\]
        as $\mu\rightarrow\bar{\alpha}$.
    \end{proof}
    \begin{remark}
    	{\em We have $u_{\mu}^{-}\rightarrow 0$ in $D^{1,p}(\mathbb{R}^N)$ and $m^{-}(\mu)\rightarrow 0$ as $\mu\rightarrow\bar{\alpha}$.}
    \end{remark}

    \noindent\textbf{Proof of Theorem \ref{th5}(1)} Let
        \[w_{\mu}=s_{\mu}^{\frac{N}{p}}u_{\mu}^{-}(s_{\mu}\cdot)\in S,\]
    where $s_{\mu}=(\bar{\alpha}-\mu)^{-(N-p)/p^2}$. By Lemma \ref{asycrilam},
        \[\lVert\nabla w_{\mu}\rVert_{p}^p=s_{\mu}^p\lVert\nabla u_{\mu}^{-}\rVert_{p}^p\leqslant C,\]
    which implies $\{w_{\mu}\}$ is bounded in $W^{1,p}(\mathbb{R}^N)$. Thus, there exists $w\in W^{1,p}(\mathbb{R}^N)$ such that $w_{\mu}\rightharpoonup w$ in $W^{1,p}(\mathbb{R}^N)$, $w_{\mu}\rightarrow w$ in $L^q(\mathbb{R}^N)$, $w_{\mu}\rightarrow w$ a.e. in $\mathbb{R}^N$.

    Direct calculations shows that
        \[-\Delta_{p}w_{\mu}=\lambda_{\mu}^{-}s_{\mu}^pw_{\mu}^{p-1}+\mu w_{\mu}^{q-1}+s_{\mu}^{-\frac{p^2}{N-p}}w_{\mu}^{p^*-1}.\]
    By Lemma \ref{asycrilam}, we know $\{\lambda_{\mu}^{+}s_{\mu}^p\}$ is bounded, hence there exists $\sigma_{0}>0$ such that $\lambda_{\mu}^{+}s_{\mu}^p\rightarrow-\sigma_{0}$ as $\mu\rightarrow 0$. Similar to the proof of Theorem \ref{th4}(1), we can prove that $w_{\mu}\rightarrow w$ in $W^{1,p}(\mathbb{R}^N)$. Thus, $w$ satisfies the equation
        \[-\Delta_{p}w+\sigma_{0}w^{p-1}=\bar{\alpha}w^{q-1}.\]
    Let $\tilde{w}=(\bar{\alpha}\sigma_{0})^{\frac{1}{p-q}}w\big(\sigma_{0}^{-\frac{1}{p}}\cdot\big)$. It is not difficult to show that
        \[(\bar{\alpha}\sigma_{0})^{\frac{1}{p-q}}s_{\mu}^{\frac{N}{p}}u_{\mu}^{-}\big(\sigma_{0}^{-\frac{1}{p}}s_{\mu}\cdot\big)\rightarrow\tilde{w}\]
    in $W^{1,p}(\mathbb{R}^N)$, and $\tilde{w}$ satisfies (\ref{uniqueness}). By the regularity and properties of $u_{\mu}^{-}$, we can derive that $\tilde{w}$ is the "ground state" of (\ref{uniqueness}) and hence $\tilde{w}=\phi_{0}$. Now, using $w\in S_{a}$, we can obtain that
        \[\sigma_{0}=\bar{\alpha}^{\frac{p^2}{N(q-p)-p^2}}\bigg(\frac{a^p}{\lVert\phi_{0}\rVert_{p}^p}\bigg)^{\frac{p(q-p)}{p^2-N(q-p)}}.\]
    $\hfill\qed$

    Now, we assume that $p+\frac{p^2}{N}<q<p^*$. Obviously, the upper bound of $\mu$ is $+\infty$.
    \begin{lemma}\label{asysuplam}
    	We have
    	    \[-\lambda_{\mu}^{-}\sim\lVert\nabla u_{\mu}^{-}\rVert_{p}^p\sim\mu^{-\frac{p}{q\gamma_{q}-p}},\]
        as $\mu\rightarrow+\infty$.
    \end{lemma}
    \begin{proof}
    	By the Gagliardo-Nirenberg inequality and Sobolev inequality,
    	    \[\lVert\nabla u_{\mu}^{-}\rVert_{p}^p=\mu\gamma_{q}\lVert u_{\mu}^{-}\rVert_{q}^q+\lVert u_{\mu}^{-}\rVert_{p^*}^{p^*}\leqslant\mu\gamma_{q} a^{q(1-\gamma_{q})}C_{N,p,q}^q\lVert\nabla u_{\mu}^{-}\rVert_{p}^{q\gamma_{q}}+S^{\frac{p^*}{p}}\lVert\nabla u_{\mu}^{-}\rVert_{p}^{p^*},\]
    	which implies
    	    \begin{equation}\label{asysupnab}
    		    \lVert\nabla u_{\mu}^{-}\rVert_{p}^p\geqslant C\mu^{-\frac{p}{q\gamma_{q}-p}},
    	    \end{equation}
    	as $\mu\rightarrow+\infty$.
    	
    	Let $u\in S$ be fixed. Then, there exists $t_{\mu}\in\mathbb{R}$ such that $t_{\mu}\star u\in\mathcal{P}_{\mu}^{-}$, that is
    	    \[e^{pt_{\mu}}\lVert\nabla u\rVert_{p}^p=\mu\gamma_{q}e^{q\gamma_{q}t_{\mu}}\lVert u\rVert_{q}^q+e^{p^*t_{\mu}}\lVert u\rVert_{p^*}^{p^*},\]
    	which implies
    	    \[e^{t_{\mu}}\leqslant C\mu^{-\frac{1}{q\gamma_{q}-p}},\]
    	as $\mu\rightarrow+\infty$. Therefore,
    	    \[E_{\mu}(t_{\mu}\star u)=\frac{1}{N}e^{pt_{\mu}}\lVert\nabla u\rVert_{p}^p+\mu\gamma_{q}\Big(\frac{1}{q\gamma_{q}}-\frac{1}{p^*}\Big)e^{p^*t_{\mu}}\lVert u\rVert_{q}^q\leqslant C\mu^{-\frac{p}{q\gamma_{q}-p}},\]
    	as $\mu\rightarrow+\infty$. Since,
    	    \[E_{\mu}(u_{\mu}^{-})=\frac{1}{N}\lVert\nabla u_{\mu}^{-}\rVert_{p}^{p}+\mu\gamma_{q}\Big(\frac{1}{q\gamma_{q}}-\frac{1}{p^*}\Big)\lVert u_{\mu}^{-}\rVert_{q}^q\geqslant\frac{1}{N}\lVert\nabla u_{\mu}^{-}\rVert_{p}^{p},\]
    	we have
    	    \[\lVert\nabla u_{\mu}^{-}\rVert_{p}^{p}\leqslant E_{\mu}(t_{\mu}\star u)\leqslant C\mu^{-\frac{p}{q\gamma_{q}-p}},\]
    	as $\mu\rightarrow+\infty$, which together with (\ref{asysupnab}), implies
    	    \[\lVert\nabla u_{\mu}^{-}\rVert_{p}^p\sim\mu^{-\frac{p}{q\gamma_{q}-p}}\]
    	as $\mu\rightarrow+\infty$.
    	
    	Using the Gagliardo-Nirenberg inequality and Sobolev inequality again, we have
    	    \[\lVert u_{\mu}^{-}\rVert_{q}^q\leqslant C\lVert\nabla u_{\mu}^{-}\rVert_{p}^{q\gamma_{q}}\leqslant C\mu^{-\frac{q\gamma_{q}}{q\gamma_{q}-p}},\]
    	and
    	    \[\mu\gamma_{q}\lVert u_{\mu}^{-}\rVert_{q}^q=\lVert\nabla u_{\mu}^{-}\rVert_{p}^p-\lVert u_{\mu}^{-}\rVert_{p}^{p^*}\geqslant\lVert\nabla u_{\mu}^{-}\rVert_{p}^p-S^{\frac{p^*}{p}}\lVert\nabla u_{\mu}^{-}\rVert_{p}^{p^*}\geqslant C\mu^{-\frac{p}{q\gamma_{q}-p}}\]
    	as $\mu\rightarrow+\infty$. Therefore,
    	    \[\lVert u_{\mu}^{-}\rVert_{q}^q\sim\mu^{-\frac{q\gamma_{q}}{q\gamma_{q}-p}}\]
    	as $\mu\rightarrow+\infty$. By the Pohozaev identity $\lambda_{\mu}^{-}a^p=\mu(\gamma_{q}-1)\lVert u_{\mu}^{-}\rVert_{q}^q$, we know
    	    \[-\lambda_{\mu}^{-}\sim\mu^{-\frac{p}{q\gamma_{q}-p}}\]
    	as $\mu\rightarrow+\infty$.
    \end{proof}
    \begin{remark}
    	{\em We have $u_{\mu}^{-}\rightarrow 0$ in $D^{1,p}(\mathbb{R}^N)$ and $m^{-}(\mu)\rightarrow 0$ as $\mu\rightarrow+\infty$.}
    \end{remark}

    \noindent\textbf{Proof of Theorem \ref{th5}(2)} Let
        \[w_{\mu}=\mu^{\frac{N}{p(q\gamma_{q}-p)}}u_{\mu}^{-}\Big(\mu^{\frac{1}{q\gamma_{q}-p}}\cdot\Big)\in S.\]
    Similar to the proof of Theorem \ref{th4}(1), we can prove that there exists $w\in W^{1,p}(\mathbb{R}^N)$ such that $w_{\mu}\rightarrow w$ in $W^{1,p}(\mathbb{R}^N)$ and $w$ satisfies
        \[-\Delta_{p}w+\sigma_{0}w^{p-1}=w^{q-1}\]
    for some $\sigma_{0}>0$.

    Let $\tilde{w}=\sigma_{0}^{\frac{1}{p-q}}w\big(\sigma_{0}^{-\frac{1}{p}}\cdot\big)$. Then
        \[\sigma_{0}^{\frac{1}{p-q}}\mu^{\frac{N}{q\gamma_{q}-p}}u_{\mu}^{-}\Big(\sigma_{0}^{-\frac{1}{p}}\mu^{\frac{1}{q\gamma_{q}-p}}\cdot\Big)\rightarrow\tilde{w}\]
    in $W^{1,p}(\mathbb{R}^N)$ as $\mu\rightarrow+\infty$ and we can prove that $\tilde{w}=\phi_{0}$. Finally, using $w\in S_{a}$, we have
        \[\sigma_{0}=\bigg(\frac{a^p}{\lVert\phi_{0}\rVert_{p}^p}\bigg)^{\frac{p(q-p)}{p^2-N(q-p)}}.\]
    $\hfill\qed$

\section{\textbf{Nonexistence result}}
    In this section, we prove the nonexistence result for $\mu<0$. The proof is not complicated and is a direct application of the result in \cite{asnsb}.\\

    \noindent\textbf{Proof of Theorem \ref{th6}} Let $u$ be a critical point of $E_{\mu}|_{S_{a}}$. Then, $u$ solves (\ref{equation}) for some $\lambda\in\mathbb{R}$. By the Pohozaev identity, we have
        \[\lambda a^p=\mu(\gamma_{q}-1)\lVert u\rVert_{q}^q,\]
    which implies $\lambda<0$, since $\mu<0$, $\gamma_{q}<1$ and $u\in S_{a}$.

    Using the Sobolev inequality and the fact that $P_{\mu}(u)=0$, we deduce that
        \[\lVert\nabla u\rVert_{p}^p=\mu\gamma_{q}\lVert u\rVert_{q}^q+\lVert u\rVert_{p^*}^{p^*}<\lVert u\rVert_{p^*}^{p^*}\leqslant S^{-\frac{p^*}{p}}\lVert\nabla u\rVert_{p}^{p^*},\]
    which implies $\lVert\nabla u\rVert_{p}^p>S^{\frac{N}{p}}$. Therefore,
        \[E_{\mu}(u)=\frac{1}{N}\lVert\nabla u\rVert_{p}^p-\mu\gamma_{q}\Big(\frac{1}{q\gamma_{q}}-\frac{1}{p^*}\Big)\lVert u\rVert_{q}^q>\frac{1}{N}S^{\frac{N}{p}}.\]
    We complete the proof of (1).

    In order to prove (2), we use corollary 4.2 in \cite{asnsb}. Let $Q=\Delta_{p}$, $\gamma=0$ and $g(u)=\lambda u+\mu u^{q-1}+u^{p^*-1}$. We know
        \[\alpha^*=\frac{N-p}{N-1}>0,\]
    thus,
        \[\sigma^*=p-1+\frac{p-\gamma}{\alpha^*}=\frac{N(p-1)}{N-p}.\]
    By (1), since $\lambda<0$, we have
        \[\liminf_{s\rightarrow 0^+}s^{-\sigma^*}g(s)=+\infty.\]
     Now, by \cite[corollary 4.2]{asnsb}, (\ref{equation}) has no positive solution for any $\mu<0$.$\hfill\qed$

\section{\textbf{Multiplicity result}}
    In this section, we will prove the multiplicity result. Thus, we always  assume that the assumptions of Theorem \ref{th7} holds.

    Firstly, we introduce the concept of genus. Let $X$ be a Banach space and $A$ be a subset of $X$. The set $A$ is said to be symmetric if $u\in A$ implies $-u\in A$. Denote by $\Sigma$ the family of closed symmetric subsets $A$ of $X$ such that $0\notin A$, that is
        \[\Sigma=\{A\subset X\backslash\{0\}: \mbox{A is closed and symmetric with respect to the origin}\}.\]
    For $A\in\Sigma$, define the genus $\gamma(A)$ by
        \[\gamma(A)=\min\{k\in\mathbb{N}: \exists\phi\in C(A,\mathbb{R}^k\backslash\{0\})\ \mbox{and}\ \phi(x)=-\phi(-x), \forall x\in X\}.\]
    If such odd map $\phi$ does not exist, we define $\gamma(A)=+\infty$. For all $k\in\mathbb{N}_{+}$, let
        \[\Sigma_{k}=\{A:A\in\Sigma\ and\ \gamma(A)\geqslant k\}.\]

    For every $\delta>0$ and $A\in\Sigma$, let
        \[A_{\delta}=\{x\in X:\inf\nolimits_{y\in A}\lVert x-y\rVert_{X}\leqslant\delta\}.\]
    We have following lemma with respect to genus.
    \begin{lemma}\label{genus}
    	{\rm\cite[section 7]{rph}} Let $A,B\in\Sigma$. Then the following statements hold.
    	
    	\noindent{\rm (i)}
    	\begin{minipage}[t]{\linewidth}
    		If $\gamma(A)\geqslant 2$, then $A$ contains infinitely many distinct points.
    	\end{minipage}

    	\noindent{\rm (ii)}
    	\begin{minipage}[t]{\linewidth}
    		If there exists an odd mapping $f\in C(A,B)$, then $\gamma(A)\leqslant\gamma(B)$. In particular, if $f$ is a homeomorphism between $A$ and $B$, then $\gamma(A)=\gamma(B)$.
    	\end{minipage}

        \noindent{\rm (iii)}
        \begin{minipage}[t]{\linewidth}
        	Let $\mathbb{S}^{N-1}$ is the sphere in $\mathbb{R}^N$, then $\gamma(\mathbb{S}^{N-1})=N$.
        \end{minipage}

        \noindent{\rm (iv)}
        \begin{minipage}[t]{\linewidth}
        	If $\gamma(B)<+\infty$, then $\gamma(\overline{A-B})\geqslant\gamma(A)-\gamma(B)$.
        \end{minipage}

        \noindent{\rm (v)}
        \begin{minipage}[t]{\linewidth}
        	If $A$ is compact, then $\gamma(A)<\infty$ and there exists $\delta>0$ such that $\gamma(A)=\gamma(A_{\delta})$.
        \end{minipage}
    \end{lemma}

    Let $\varphi\in C^{1}(X,\mathbb{R})$ be an even functional and
        \[V=\{v\in X:\psi(v)=1\},\]
    where $\psi\in C^{2}(X,\mathbb{R})$ and $\psi'(v)\neq 0$ for all $v\in V$. We define the set of critical points of $\varphi|_{V}$ at level $c$ as
        \[K^c=\{u\in V:\varphi(u)=c,\varphi|_{V}'(u)=0\}.\]
    The following conclusion is the key to proving the result of multiplicity.
    \begin{proposition}\label{minimax}
    	Assume that $\varphi|_{V}$ is bounded from below and satisfies the $(PS)_{c}$ condition for all $c<0$. Moreover, we also assume that $\Sigma_{k}\neq\emptyset$ for $k\in\mathbb{N}_{+}$. Define a sequence of mini-max values $-\infty<c_{1}\leqslant c_{2}\leqslant...\leqslant c_{n}\leqslant...<+\infty$ as follows
    	    \[c_{k}:=\inf_{A\in\Sigma_{k}}\sup_{u\in A}\varphi(u)\quad\forall k=1,2,..,n.\]
    	We have following statements hold.
    	
    	\noindent{\rm (i)}
    	\begin{minipage}[t]{\linewidth}
    		If $c_{k}<0$, then $c_{k}$ is a critical value of $\varphi|_{V}$.
    	\end{minipage}

        \noindent{\rm (ii)}
        \begin{minipage}[t]{\linewidth}
        	If there exists $c<0$ such that
        	\[c_{k}=c_{k+1}=...=c_{k+l}=c,\]
        	then $\gamma(K^c)\geqslant l+1$. In particular, $\varphi|_{V}$ has infinitely many critical points at level $c$ if $l\geqslant 2$.
        \end{minipage}
    \end{proposition}
    \begin{proof}
    	The proof is very similar to \cite[Theorem 2.1]{jlls}, if we replace \cite[Lemma 2.3]{jlls} with following quantitative deformation lemma.
    \end{proof}

    For every $c,d\in\mathbb{R}$ with $c<d$, define
        \[\varphi|_{V}^{c}:=\{u\in V:\varphi(u)\leqslant c\},\quad and\quad\varphi^{-1}([c,d]):=\{u\in X:c\leqslant\varphi(u)\leqslant d\}.\]
    Then, we have following quantitative deformation lemma.
    \begin{lemma}
    	{\rm \cite[Lemma 5.15]{wm}} Let $\varphi\in C^{1}(X,\mathbb{R})$, $W\subset V$, $c\in\mathbb{R}$, and $\varepsilon,\delta>0$ such that
    	    \[\lVert\varphi|_{V}'(u)\rVert\geqslant\frac{8\varepsilon}{\delta}\quad\forall u\in\varphi^{-1}([c-2\varepsilon,c+2\varepsilon]\cap W_{2\delta}).\]
    	Then there exists $\eta\in C([0,1]\times V,V)$ such that
    	
    	\noindent{\rm (i)} $\eta(t,u)=u$ if $t=0$ or if $u\notin\varphi^{-1}([c-2\varepsilon,c+2\varepsilon])\cap W_{2\delta}$.
    	
    	\noindent{\rm (ii)} $\eta(1,\varphi^{c+\varepsilon}\cap W)\in\varphi^{c-\varepsilon}$.
    	
    	\noindent{\rm (iii)} $\varphi(\eta(\cdot,u))$ is non-increasing in $t\in [0,1]$ for all $v\in V$.
    	
    	\noindent{\rm (iv)} $\eta(t,u)$ is odd in $V$ for all $t\in [0,1]$ if $\varphi$ is even in $V$.
    \end{lemma}

    Let $\tau\in C^{\infty}(\mathbb{R}^{+},[0,1])$ be a non-increasing function satisfies
        \[\tau(t)=1\ \mbox{for}\ t\in[0,R_{0}],\quad\mbox{and}\quad\tau(t)=0\ \mbox{for}\ t\in[R_{1},+\infty),\]
    where $R_{0}$ and $R_{1}$ are obtained by Lemma \ref{hfunction}. Define the truncated functional as follows
        \[E_{\tau}(u)=\frac{1}{p}\lVert\nabla u\rVert_{p}^p-\frac{\mu}{q}\lVert u\rVert_{q}^q-\frac{1}{p^*}\tau(\lVert\nabla u\rVert_{p})\lVert u\rVert_{P^*}^{p^*}.\]
    For $u\in S_{a}$, by the Gagliardo-Nirenberg inequality and Sobolev inequality, there is
        \[E_{\tau}(u)\geqslant\frac{1}{p}\lVert\nabla u\rVert_{p}^p-\frac{\mu}{q}C_{N,p,q}^qa^{q(1-\gamma_{q})}\lVert\nabla u\rVert_{p}^{q\gamma_{q}}-\frac{1}{p^*S^{p^*/p}}\tau(\lVert\nabla u\rVert_{p})\lVert\nabla u\rVert_{p}^{p^*}=\tilde{h}(\lVert\nabla u\rVert_{p}),\]
    where
        \[\tilde{h}(t)=\frac{1}{p}t^p-\frac{\mu}{q}C_{N,p,q}^qa^{q(1-\gamma_{q})}t^{q\gamma_{q}}-\frac{\tau(t)}{p^*S^{p^*/p}}t^{p^*}.\]
    By Lemma \ref{hfunction}, we know that $\tilde{h}(t)<0$ for $t\in(0,R_{0})$ and $\tilde{h}(t)>0$ for $t\in(R_{0},+\infty)$.

    \begin{lemma}\label{infps}
    	We have
    	\noindent{\rm (i)}
    	\begin{minipage}[t]{\linewidth}
    		$E_{\tau}\in C^{1}(W_{rad}^{1,p}(\mathbb{R}^N),\mathbb{R})$.
    	\end{minipage}

        \noindent{\rm (ii)}
        \begin{minipage}[t]{\linewidth}
        	$E_{\tau}|_{S_{a,r}}$ is coercive and bounded from below. Moreover, if $E_{\tau}(u)\leqslant 0$ on $S_{a,r}$, then $\lVert\nabla u\rVert_{p}\leqslant R_{0}$ and $E_{\tau}(u)=E_{\mu}(u)$.
        \end{minipage}

        \noindent{\rm (iii)}
        \begin{minipage}[t]{\linewidth}
        	$E_{\tau}|_{S_{a,r}}$ satisfies the $(PS)_{c}$ condition for all $c<0$.
        \end{minipage}
    \end{lemma}
    \begin{proof}
    	(i) In fact, we just have to prove $I(u)=\tau(\lVert\nabla u\rVert_{p})\in C^{1}(W_{rad}^{1,p}(\mathbb{R}^N),\mathbb{R})$. For every $u\in W_{rad}^{1,p}(\mathbb{R}^N)$, direct calculations show that
    	    \[I'(u)v=\frac{1}{p}\tau'(\lVert\nabla u\rVert_{p})\lVert\nabla u\rVert_{p}^{1-p}\int_{\mathbb{R}^N}\lvert\nabla u\rvert^{p-2}\nabla u\cdot\nabla vdx\quad\forall v\in W_{rad}^{1,p}(\mathbb{R}^N).\]
    	
    	(ii) For every $u\in S_{a,r}$, since $E_{\tau}(u)\geq\tilde{h}(\lVert\nabla u\rVert_{p})$ and $\tilde{h}(t)\rightarrow+\infty$ as $t\rightarrow+\infty$, we know $E_{\tau}|_{S_{a,r}}$ is coercive and bounded from below. If $E_{\tau}(u)\leqslant 0$ on $S_{a,r}$. Then, using $E_{\tau}(u)\geq\tilde{h}(\lVert\nabla u\rVert_{p})$ again, since $\tilde{h}(t)>0$ on $(R_{0},+\infty)$, we have $\lVert\nabla u\rVert_{p}\leqslant R_{0}$ and $E_{\tau}(u)=E_{\mu}(u)$.
    	
    	(iii) Let $\{u_{n}\}\in S_{a,r}$ be a PS sequence for $E_{\tau}|_{S_{a,r}}$ at level $c<0$. Then, by (ii), we know $\lVert\nabla u_{n}\rVert_{p}<R_{0}$ for $n$ sufficiently large and hence $E_{\tau}(u_{n})=E_{\mu}(u_{n})$. Therefore, $\{u_{n}\}$ is also a PS sequence for $E_{\mu}|_{S_{a,r}}$. Since $\{u_{n}\}$ is bounded in $W_{rad}^{1,p}(\mathbb{R}^N)$, similar to the proof of Proposition \ref{compactnesslemma}, by concentration compactness lemma, we can prove that one of the cases in Proposition \ref{compactnesslemma} holds. However, similar to the proof of Theorem \ref{th1}, we can prove that case (i) does not occurs under assumption $\mu<\alpha a^{q(\gamma_{q}-1)}$. Thus, $\{u_{n}\}$ converges strongly in $W_{rad}^{1,p}(\mathbb{R}^N)$.
    \end{proof}

    \begin{lemma}\label{gengeqn}
    	Given $n\in\mathbb{N}_{+}$, there exists $\varepsilon=\varepsilon(n)$ such that $\gamma(E_{\tau}|_{S_{a,r}}^{-\varepsilon})\geqslant n$.	
    \end{lemma}
    \begin{proof}
    	The main idea of this proof comes from \cite{gajpai}. For every $n\in\mathbb{N}_{+}$ and $R>1$, let
    	    \[u_{k}(x)=A_{k,R}(1+\lvert x\rvert^2)^k\varphi_{k,R}(x)\in S_{a}\quad \mbox{for}\quad k=1,...,n,\]
    	where $A_{k,R}$ is a constant and $\varphi_{k,R}\in C_{c}^{\infty}(\mathbb{R}^N)$ is radial cut-off function satisfies $0\leqslant\varphi_{k,R}\leqslant 1$,
    	    \[\varphi_{k,R}=1\ \mbox{in}\ B_{(2k+\frac{1}{2})R}\backslash B_{(2k-\frac{1}{2})R}^c,\quad\varphi_{k,R}=0\ \mbox{in}\ B_{(2k-1)R}\cup B_{(2k+1)R}^c\quad\mbox{and}\quad\lvert\nabla\varphi_{k,R}\rvert\leqslant\frac{4}{R}.\]
    	Since $u_{k}\in S_{a}$ and
    	    \begin{align*}
    	    	\lVert u_{k}\rVert_{p}^p&=A_{k,R}^p\int_{\mathbb{R}^N}(1+\lvert x\rvert^2)^{kp}\varphi_{k,R}^p(x)dx\sim A_{k,R}^p\int_{0}^{+\infty}(1+r^2)^{kp}r^{N-1}\varphi_{k,R}^p(r)dr\\
    	    	&\sim A_{k,R}^p\int_{(2k-\frac{1}{2})R}^{(2k+\frac{1}{2})R}(1+r^2)^{kp}r^{N-1}dr\sim A_{k,R}^pR^{2kp+N}
    	    \end{align*}
        as $R\rightarrow+\infty$, we have $A_{k,R}\sim R^{-2k-\frac{N}{p}}$ as $R\rightarrow+\infty$. we know
            \[\nabla u_{k}(x)=A_{k,R}(1+\lvert x\rvert^2)^k\nabla\varphi_{k,R}(x)+2kA_{k,R}(1+\lvert x\rvert^2)^{k-1}\varphi_{k,R}(x)x.\]
        Then, direct calculations show that
            \begin{equation}\label{infnab}
            	\lVert\nabla u_{k}\rVert_{p}^p\leqslant\frac{C}{R^p}\quad\mbox{for}\quad k=1,...,n
            \end{equation}
        as $R\rightarrow+\infty$.

        It is clear that $u_{1},...,u_{n}$ are linearly independent in $W_{rad}^{1,p}(\mathbb{R}^N)$. Thus, we can define a $n$-dimensional subspace of $W_{rad}^{1,p}(\mathbb{R}^N)$ by
            \[E_{n}=\mbox{span}\{u_{1},...,u_{n}\}.\]
         For every $v_{n}\in S_{a,r}\cap E_{n}$, there exists $a_{1},...,a_{n}$ such that $v_{n}=a_{1}u_{1}+...+a_{n}u_{n}$. Since $v_{n}\in S_{a,r}$ and
            \[\lVert a_{1}u_{1}+...+a_{n}u_{n}\rVert_{p}^p=\lVert a_{1}u_{1}\rVert_{p}^p+...+\lVert a_{n}u_{n}\rVert_{p}^p=(\lvert a_{1}\rvert^p+...+\lvert a_{n}\rvert^p)a^p,\]
        we have
            \[\lvert a_{1}\rvert^p+...+\lvert a_{n}\rvert^p=1.\]
        Therefore, by (\ref{infnab}),
            \begin{align}\label{infetau}
            	E_{\tau}(v_{n})&=\frac{1}{p}\lVert\nabla v_{n}\rVert_{p}^p-\frac{\mu}{q}\lVert v_{n}\rVert_{q}^q-\frac{\tau(\lVert\nabla v_{n}\rVert_{p})}{p^*}\lVert v_{n}\rVert_{p^*}^{p^*}\nonumber\\
            	&\leqslant\frac{1}{p}(\lvert a_{1}\rvert^p\lVert\nabla u_{1}\rVert_{p}^p+...+\lvert a_{n}\rvert^p\lVert\nabla u_{n}\rVert_{p}^p)-\frac{\mu}{q}\lVert v_{n}\rVert_{q}^q\nonumber\\
            	&\leqslant\frac{C}{R^p}-\frac{\mu}{q}\lVert v_{n}\rVert_{q}^q
            \end{align}
        as $R\rightarrow+\infty$. By the H\"older inequality
            \begin{align*}
            	a^p&=\lVert v_{n}\rVert_{p}^p=\int_{B_{(2n+1)R}}\lvert v_{n}\rvert^pdx\leqslant\Big(\int_{B_{(2n+1)R}}\lvert v_{n}\rvert^qdx\Big)^{\frac{p}{q}}\Big(\int_{B_{(2n+1)R}}dx\Big)^{\frac{q-p}{p}}\\
            	&=(2n+1)^{\frac{N(q-p)}{q}}\omega_{N}^{\frac{q-p}{p}}R^{\frac{N(q-p)}{q}}\lVert v_{n}\rVert_{q}^p,
            \end{align*}
        which implies
            \begin{equation}\label{infq}
            	\lVert v_{v}\rVert_{q}^q\geqslant\frac{C}{R^{q\gamma_{q}}}
            \end{equation}
        as $R\rightarrow+\infty$. Now, combining (\ref{infetau}) and (\ref{infq}), we obtain
            \[E_{\tau}(v_{n})\leqslant\frac{C-R^{p-q\gamma_{q}}}{R^p}<-\varepsilon,\]
        by taking $R$ sufficiently large and $\varepsilon$ sufficiently small. Since $v_{n}\in S_{a,r}\cap E_{n}$ is arbitrary, this means that $S_{a,r}\cap E_{n}\subset E_{\tau}^{-\varepsilon}$. We know $E_{n}$ is a space of finite dimension, so all the norms in $E_{n}$ are equivalent. Then, by Lemma \ref{genus},
            \[\gamma(E_{\tau}^{-\varepsilon})\geqslant\gamma(S_{a,r}\cap E_{n})=\gamma(\mathbb{S}^{n-1})=n.\]
    \end{proof}

    Now, we can use Proposition \ref{minimax} to prove our multiplicity result.\\

    \noindent\textbf{Proof of Theorem \ref{th7}} Let $\varphi=E_{\tau}$, $X=W_{rad}^{1,p}(\mathbb{R}^N)$. Since $p>2$, by \cite[Proposition 1.12]{wm},
        \[\psi(u)=\frac{1}{a^p}\int_{\mathbb{R}^N}\lvert u\rvert^pdx\in C^2(W_{rad}^{1,p},\mathbb{R}),\]
    which implies we can set $V=S_{a,r}$. By Lemma \ref{infps}, $E_{\tau}|_{S_{a,r}}$ is bounded from below and satisfies $(PS)_{c}$ condition for all $c<0$. Moreover, by lemma \ref{gengeqn}, $\Sigma_{k}\neq\emptyset$ and $c_{k}<0$ for all $k\in\mathbb{N}_{+}$. Thus, Proposition \ref{minimax} implies $E_{\tau}|_{S_{a,r}}$ has infinitely many solutions at negative level. Using Lemma \ref{infps} again, we know $E_{\tau}|_{S_{a,r}}=E_{\mu}|_{S_{a,r}}$ at negative level. Thus, $E_{\mu}|_{S_{a,r}}$ has infinitely many solutions. Finally, by the principle of symmetric criticality(see \cite[Theorem 2.2]{kjom}), we know $E_{\mu}|_{S_{a}}$ has infinitely many solutions.

\appendix
\section{\textbf{Some useful estimates}}\label{A}
    For every $\varepsilon>0$, we define
        \[u_{\varepsilon}(x)=\varphi(x)U_{\varepsilon}(x)=\varphi(x)d_{N,p}\varepsilon^{\frac{N-p}{p(p-1)}}\big(\varepsilon^{\frac{p}{p-1}}+\lvert x\rvert^{\frac{p}{p-1}}\big)^{\frac{p-N}{p}},\]
    where $\varphi\in C_{c}^\infty(\mathbb{R}^N)$ is a radial cut off function with $\varphi=1$ in $B_{1}$, $\varphi=0$ in $B_{2}^c$, and $\varphi$ radially decreasing. Then, we have following estimates for $u_{\varepsilon}$.
    \begin{lemma}
    	Let $N\geqslant 2$, $1<p<N$, $1\leqslant r<p^*$. Then, we have
    	    \[\lVert\nabla u_{\varepsilon}\rVert_{p}^p=S^{\frac{N}{p}}+O(\varepsilon^{\frac{N-p}{p-1}}),\quad\lVert u_{\varepsilon}\rVert_{p^*}^{p^*}=S^{\frac{N}{p}}+O(\varepsilon^{\frac{N}{p-1}}),\]
    	    \begin{align*}
    	    	\lVert\nabla u_{\varepsilon}\rVert_{r}^{r}\sim\left\{\begin{array}{ll}
    	    		\varepsilon^{\frac{N(p-r)}{p}}&\frac{N(p-1)}{N-1}<r<p\\
    	    		\varepsilon^{\frac{N(N-p)}{(N-1)p}}\lvert\log\varepsilon\rvert&r=\frac{N(p-1)}{N-1}\\
    	    		\varepsilon^{\frac{(N-p)r}{p(p-1)}}&1\leqslant r<\frac{N(p-1)}{N-1},
    	    	\end{array}\right.
    	    \end{align*}
        and
            \begin{align*}
            	\lVert u_{\varepsilon}\rVert_{r}^r\sim\left\{\begin{array}{ll}
            		\varepsilon^{N-\frac{(N-p)r}{p}}&\frac{N(p-1)}{N-p}<r<p^*\\
            		\varepsilon^{\frac{N}{p}}\lvert\log\varepsilon\rvert&r=\frac{N(p-1)}{N-p}\\
            		\varepsilon^{\frac{(N-p)r}{p(p-1)}}&1\leqslant r<\frac{N(p-1)}{N-p},
            	\end{array}\right.
            \end{align*}
        as $\varepsilon\rightarrow 0$.
    \end{lemma}

\subsection*{Acknowledgments}

The authors were  supported by National Natural Science Foundation of China 11971392.


\begin{thebibliography}{99}
	\bibitem{am}
	    M. Agueh, {\em Sharp Gagliardo-Nirenberg inequalities via $p$-Laplacian type equations}. NoDEA Nonlinear Differential Equations Appl. 15 (2008), no. 4-5, 457-472.
	\bibitem{acosmasmm}
	    C. O. Alves, M. A. S. Souto, M. Montenegro,
	    {\em Existence of a ground state solution for a nonlinear scalar field equation with critical growth},
	    Calc. Var. Partial Differential Equations, \textbf{43} (2012), no. 3-4, 537-554.
	\bibitem{asnsb}
	    S.N. Armstrong, B. Sirakov, {\em Nonexistence of positive supersolutions of elliptic equations via the maximum principle},
	    Comm. Partial Differential Equations, \textbf{36} (2011), no. 11, 2011-2047.
	\bibitem{bh}
	    H. Br\'{e}zis, Some variational problems with lack of compactness. Nonlinear functional analysis and its applications, Part 1, 165¨C201, Proc. Sympos. Pure Math., 45, Part 1, Amer. Math. Soc., Providence, RI, 1986.
    \bibitem{bhle}
        H. Br\'{e}zis, E. Lieb, {\em A relation between pointwise convergence of functions and convergence of functionals}, Proc. Amer. Math. Soc, \textbf{88} (1983), no. 3, 486-490.
    \bibitem{bhnl}
        H. Br\'{e}zis, L. Nirenberg, {\em Positive solutions of nonlinear elliptic equations involving critical Sobolev exponents}, Comm. Pure Appl. Math, \textbf{36} (1983), no. 4, 437-477.
    \bibitem{fagf}
        A. Ferrero, F. Gazzola, {\em On subcriticality assumptions for the existence of ground states of quasilinear elliptic equations},
        Adv. Differential Equations, \textbf{8} (2003), no. 9, 1081-1106.
    \bibitem{fg}
        G. Fibich, {\em The nonlinear Schr\"{o}dinger equation. Singular solutions and optical collapse}, Applied Mathematical Sciences, 192. Springer, Cham, 2015. xxxii+862 pp.
    \bibitem{gajpai}
        A. J. Garc\'{i}a, A. I. Peral, {\em Multiplicity of solutions for elliptic problems with critical exponent or with a nonsymmetric term.
        Trans}, Amer. Math. Soc, \textbf{323} (1991), no. 2, 877-895.
    \bibitem{gajpai2}
        A. J. Garc\'{i}a, A. I. Peral
        {\em Some results about the existence of a second positive solution in a quasilinear critical problem},
        Indiana Univ. Math. J, \textbf{43} (1994), no. 3, 941-957.
    \bibitem{gfsj}
        F. Gazzola, J. Serrin, {\em Asymptotic behavior of ground states of quasilinear elliptic problems with two vanishing parameters}, Ann. Inst. H. Poincar\'{e} C Anal. Non Lin\'{e}aire, \textbf{19} (2002), no. 4, 477-504.
    \bibitem{gn}
        N. Ghoussoub, {\em Duality and perturbation methods in critical point theory}, With appendices by David Robinson. Cambridge Tracts in Mathematics, 107. Cambridge University Press, Cambridge, 1993. xviii+258 pp.
    \bibitem{hhsca}
        H. Hajaiej, C. A. Stuart, {\em On the variational approach to the stability of standing waves for the nonlinear Schr\"{o}dinger equation}, Adv. Nonlinear Stud, \textbf{4} (2004), no. 4, 469-501.
    \bibitem{hdly}
        D. Huang, Y. Li, {\em Multiplicity of solutions for a noncooperative p-Laplacian elliptic system in $\mathbb{R}^N$}, J. Differential Equations \textbf{215} (2005), no. 1, 206-223.
    \bibitem{jl}
        L. Jeanjean, {\em Existence of solutions with prescribed norm for semilinear elliptic equations}, Nonlinear Anal, \textbf{28} (1997), no. 10, 1633-1659.
    \bibitem{jljjltt}
        L. Jeanjean, J. Jendrej, T. T. Le, {\em Orbital stability of ground states for a Sobolev critical Schr\"{o}dinger equation}, J. Math. Pures Appl, (9) \textbf{164} (2022), 158-179.
    \bibitem{jlltt}
        L. Jeanjean, T. T. Le, {\em Multiple normalized solutions for a Sobolev critical Schr\"{o}dinger equation}, Math. Ann, \textbf{384} (2022), no. 1-2, 101-134.
    \bibitem{jlls}
        L. Jeanjean, S-S. Lu,
        {\em Nonradial normalized solutions for nonlinear scalar field equations}, Nonlinearity, \textbf{32} (2019), no. 12, 4942-4966.
    \bibitem{jlsm}
         L. Jeanjean, M. Squassina, {\em Existence and symmetry of least energy solutions for a class of quasi-linear elliptic equations}, Ann. Inst. H. Poincar\'{e} C Anal. Non Lin\'{e}aire, \textbf{26} (2009), no. 5, 1701-1716.
    \bibitem{kjom}
        J. Kobayashi, M. \^{O}tani, {\em The principle of symmetric criticality for non-differentiable mappings},
        J. Funct. Anal, \textbf{214} (2004), no. 2, 428-449.
    \bibitem{lpl}
        P. L. Lions, {\em The concentration-compactness principle in the calculus of variations. The limit case. I}, Rev. Mat. Iberoamericana, \textbf{1} (1985), no. 1, 145-201.
    \bibitem{lpl2}
        P.L. Lions, {\em The concentration-compactness principle in the calculus of variations. The locally compact case. II}, Ann. Inst. H. Poincar\'{e} Anal. Non Lin\'{e}aire, \textbf{1} (1984), no. 4, 223-283.

    \bibitem{rph}
        P. H. Rabinowitz, Minimax methods in critical point theory with applications to differential equations. CBMS Regional Conference Series in Mathematics, 65. Published for the Conference Board of the Mathematical Sciences, Washington, DC; by the American Mathematical Society, Providence, RI, 1986. viii+100 pp.
    \bibitem{sjtm}
        J. Serrin, M. Tang, {\em Uniqueness of ground states for quasilinear elliptic equations}, Indiana Univ. Math. J, \textbf{49} (2000), no. 3, 897-923.
    \bibitem{sma}
        M.A. Shibata, {\em new rearrangement inequality and its application for $L^2$-constraint minimizing problems}. Math. Z. \textbf{287}(2017), 341-359.
    \bibitem{sn1}
        N. Soave, {\em Normalized ground states for the NLS equation with combined nonlinearities}, J. Differential Equations, \textbf{269} (2020), no. 9, 6941-6987.
    \bibitem{sn2}
        N. Soave, {\em Normalized ground states for the NLS equation with combined nonlinearities: the Sobolev critical case}, J. Funct. Anal, \textbf{279} (2020), no. 6, 108610, 43 pp.
    \bibitem{sm}
        M. Struwe, Variational methods. Applications to nonlinear partial differential equations and Hamiltonian systems. Fourth edition. Springer-Verlag, Berlin, 2008. xx+302 pp.
    \bibitem{tg}
        G. Talenti, {\em Best constant in Sobolev inequality}, Ann. Mat. Pura Appl, (4) \textbf{110} (1976), 353-372.
    \bibitem{ttvmzx}
        T. Tao, M. Visan, X. Zhang, The nonlinear Schr\"{o}dinger equation with combined power-type nonlinearities. Comm. Partial Differential Equations 32 (2007), no. 7-9, 1281-1343.
    \bibitem{tp}
        P. Tolksdorf, {\em Regularity for a more general class of quasilinear elliptic equations},
        J. Differential Equations, \textbf{51} (1984), no. 1, 126-150.
    \bibitem{wcsj}
        C. Wang, J. Sun, {\em Normalized solutions for the p-Laplacian equation with a trapping potential}, Adv. Nonlinear Anal, \textbf{12} (2023), no. 1, Paper No. 20220291, 14 pp.
    \bibitem{wwlqzj}
        W. Wang, Q. Li, J. Zhou, et al, {\em Normalized solutions for p-Laplacian equations with a $L^2$-supercritical growth}, Ann. Funct. Anal, \textbf{12} (2021), no. 1, Paper No. 9, 19 pp.
    \bibitem{wjwy}
        J. Wei, Y. Wu, {\em Normalized solutions for Schr\"{o}dinger equations with critical Sobolev exponent and mixed nonlinearities}, J. Funct. Anal, \textbf{283} (2022), no. 6, Paper No. 109574, 46 pp.
    \bibitem{wm}
        M. Willem, Minimax theorems, Progress in Nonlinear Differential Equations and their Applications, 24. Birkh\"{a}user Boston, Inc, Boston, MA, 1996. x+162 pp.
    \bibitem{yjf}
        J. F. Yang, {\em Positive solutions of quasilinear elliptic obstacle problems with critical exponents}, Nonlinear Anal, \textbf{25} (1995), no. 12, 1283-1306.
    \bibitem{zz}
        Z. Zhang, {\em Variational, topological, and partial order methods with their applications.
        Developments in Mathematics}, 29. Springer, Heidelberg, 2013. xii+330 pp.
    \bibitem{zzzz}
        Z. Zhang, Z. Zhang, {\em Normalized solutions to $p$-Laplacian equations with combined nonlinearities}, Nonlinearity, \textbf{35} (2022), no. 11, 5621-5663.
\end{thebibliography}
\end{document}